\documentclass[12pt, a4paper]{amsart}
\usepackage{amsmath}
\usepackage{geometry,amsthm,graphics,tabularx,amssymb,shapepar}
\usepackage{amscd}
\usepackage{bm}
\usepackage[all]{xypic}
\newcommand{\Res}{\mathrm{Res}}
\newcommand{\fgl}{\mathfrak{gl}}
\newcommand{\fsl}{\mathfrak{sl}}

\newcommand{\fg}{\mathfrak g}
\newcommand{\fh}{\mathfrak h}

\newcommand{\wt}{\widetilde}
\newcommand{\wh}{\widehat}

\newcommand{\rG}{\mathrm{G}}
\newcommand{\rGL}{\mathrm{GL}}
\newcommand{\rH}{\mathrm{H}}
\newcommand{\rN}{\mathrm{N}}

\newcommand{\ot}{\otimes}

\newcommand{\CH}{\mathcal{H}}

\newcommand{\CL}{\mathcal{L}}

\newcommand{\C}{\mathbb{C}}

\newcommand{\N}{\mathbb N}
\newcommand{\Z}{\mathbb Z}
\newcommand{\rd}{\mathrm{d}}
\newcommand{\rk}{\mathrm{k}}

\newcommand{\ba}{\begin {eqnarray}}
\newcommand{\ea}{\end {eqnarray}}
\newcommand{\baa}{\begin {eqnarray*}}
\newcommand{\eaa}{\end {eqnarray*}}
\newcommand{\be}{\begin {equation}}
\newcommand{\ee}{\end {equation}}
\newcommand{\bee}{\begin {equation*}}
\newcommand{\eee}{\end {equation*}}

\newcommand{\U}{\mathcal{U}}

\newcommand{\te}[1]{\textnormal{{#1}}}

\theoremstyle{Theorem}

\theoremstyle{Theorem}

\newtheorem{thm}{Theorem}[section]
\newtheorem{cort}[thm]{Corollary}
\newtheorem{lemt}[thm]{Lemma}
\newtheorem{prpt}[thm]{Proposition}
\newtheorem{thmt}[thm]{Theorem}
\newtheorem{remt}[thm]{Remark}
\newtheorem{ext}[thm]{Example}
\newtheorem{dfnt}[thm]{Definition}

\def\({\left(}

\def\){\right)}

\newlength{\dhatheight}

\def \<{{\langle}}
\def \>{{\rangle}}

\allowdisplaybreaks

\numberwithin{equation}{section}
\title[EALA, vertex algebra, and reductive group]{Extended affine Lie algebras, vertex algebras, and reductive groups}

\author{Fulin Chen$^1$}
\address{School of Mathematical Sciences, Xiamen University,
 Xiamen, China 361005} \email{chenf@xmu.edu.cn}
 \thanks{$^1$Partially supported by China NSF grant (No.11971397) and the Fundamental
Research Funds for the Central Universities (No.20720190069).}
 \author{Haisheng Li}\address{Department of Mathematical Sciences,
Rutgers University, Camden, NJ 08102, USA} \email{hli@camden.rutgers.edu}
 \author{Shaobin Tan$^2$}\address{School of Mathematical Sciences, Xiamen University,
 Xiamen, China 361005} \email{tans@xmu.edu.cn }\thanks{$^2$Partially supported by China NSF grants (Nos.11531004, 11971397)}
\author{Qing Wang$^3$}\address{School of Mathematical Sciences, Xiamen University,
 Xiamen, China 361005} \email{qingwang@xmu.edu.cn }\thanks{$^3$Partially supported by
 China NSF grants (Nos.11531004, 11622107)}


\begin{document}

\begin{abstract}
In this paper, we explore natural connections among the representations
of the extended affine Lie algebra $\wh{\fsl_N}(\C_q)$ with $\C_q=\C_q[t_0^{\pm 1},t_1^{\pm 1}]$ an irrational
quantum $2$-torus, the simple affine vertex algebra $L_{\wh{\fsl_\infty}}(\ell,0)$ with $\ell$ a positive integer, and  Levi subgroups $\mathrm{GL}_{\bf I}$
of $\mathrm{GL}_\ell(\C)$.
First, we give a canonical isomorphism between the category of integrable restricted
$\wh{\fsl_N}(\C_q)$-modules of level $\ell$ and that of
 equivariant quasi $L_{\wh{\fsl_\infty}}(\ell,0)$-modules.
Second, we classify irreducible $\N$-graded equivariant quasi $L_{\wh{\fsl_\infty}}(\ell,0)$-modules.
Third, we establish a duality between irreducible $\N$-graded equivariant quasi  $L_{\wh{\fsl_\infty}}(\ell,0)$-modules
 and irreducible regular $\mathrm{GL}_{\bf I}$-modules on certain fermionic Fock spaces.
 Fourth, we obtain an explicit realization of every irreducible  $\N$-graded equivariant quasi $L_{\wh{\fsl_\infty}}(\ell,0)$-module.
Fifth, we completely determine the following branchings: \te{(i)} The branching from
 $L_{\wh{\fsl_\infty}}(\ell,0)\ot L_{\wh{\fsl_\infty}}(\ell',0)$ to $L_{\wh{\fsl_\infty}}(\ell+\ell',0)$ for quasi modules.
\te{(ii)} The branching from $\wh{\fsl_N}(\C_q)$ to its Levi subalgebras.
\te{(iii)} The branching from $\wh{\fsl_N}(\C_q)$
to its subalgebras $\wh{\fsl_N}(\C_q[t_0^{\pm M_0},t_1^{\pm M_1}])$.
\end{abstract}
\maketitle

\section{Introduction}
This paper is to establish and explore a natural connection between the extended affine Lie algebras of type $A_{N-1}$
coordinated by irrational quantum tori and vertex algebras.
Extended affine Lie algebra, written as EALA in short,  was
first introduced by Hoegh-Krohn and Torresani in \cite{H-KT} under the name of quasi-simple Lie algebra,
and since then it has been extensively studied in literature (see \cite{AABGP,BGK,N} and the references therein).
 By definition, an EALA  is a Lie algebra together with a finite-dimensional
ad-diagonalizable subalgebra  and a nondegenerate symmetric invariant bilinear form, satisfying a list of conditions.
One of the conditions is that the group generated by the isotropic roots is a
free abelian group of finite rank called the {\em nullity}.
It is known that EALAs of nullity $0$ are exactly finite-dimensional simple Lie algebras,
while EALAs of nullity $1$ are precisely affine Kac-Moody algebras (see \cite{ABGP}).

Note that affine Kac-Moody algebras were classified as untwisted affine Lie algebras and twisted affine Lie algebras
(see \cite{Kac}).
It is known that untwisted affine Lie algebras through their highest weight modules can be naturally associated with
vertex operator algebras and modules (see \cite{FZ}), whereas twisted affine Lie algebras
can be associated to twisted modules for the vertex operator algebras associated to
the corresponding untwisted affine Lie algebras (see \cite{FLM}, \cite{li-twisted}).
These (affine) vertex operator algebras, being the major building blocks in vertex operator algebra theory,
play an important role.  An eminent problem is to establish and explore natural connections of vertex algebras with all EALAs.

The structure of a general EALA is now pretty well understood and
 the EALAs with positive nullity are like affine Kac-Moody algebras in many ways.
Among these Lie algebras, toroidal extended affine Lie algebras are multi-loop generalizations of
untwisted affine Lie algebras.
 It was proved (see \cite{BGK,ABFP, N}) that every (maximal) EALA of positive nullity  is isomorphic to a toroidal EALA, or
 the subalgebra of the fixed points in a toroidal EALA
with respect to a finite abelian automorphism group,
 or an EALA of type $A$ coordinated by an irrational quantum torus.
 Meanwhile, there have been several studies on EALAs and their related algebras where vertex algebras
  have played an important role.
 In \cite{BBS}, Berman, Billig, and Szmigielski established a natural connection between
 toroidal Lie algebras and affine vertex operator algebras.
 Later on, Billig studied in \cite{B1} the structure of the vertex  operator algebras associated to
  full toroidal Lie algebras.
By using the vertex operator algebras constructed in \cite{B1}, some irreducible integrable modules
for  toroidal EALAs  and their subalgebras of the fixed-points under certain automorphism groups
were constructed in \cite{B2,CLT} and \cite{BL},  respectively.
As for the EALAs coordinated by irrational quantum tori, an explicit natural connection with vertex algebras
is yet to be established. This is the main concern of this present paper.

In this paper, we concentrate ourselves to nullity $2$ EALAs coordinated by irrational quantum tori.
Let $q$ be a generic complex number. By definition, the quantum $2$-torus, denoted by
$\C_q[t_0^{\pm 1},t_1^{\pm 1}]$ or simply by $\C_q$,
 is the associative algebra with underlying space $\C[t_0^{\pm 1},t_1^{\pm 1}]$
and with the basic commutation relation $t_1t_0=qt_0t_1$.
On the other hand, let $N$ be a positive integer with $N\ge 2$.
 Denote by $\fgl_N(\C_q)$ the matrix Lie algebra over $\C_q$ and
 set $\fsl_N(\C_q)=[\fgl_N(\C_q),\fgl_N(\C_q)]$, the derived subalgebra.
 Furthermore, let $\wh{\fsl_N}(\C_q)$ denote the universal central extension of $\fsl_N(\C_q)$.
The nullity $2$ EALA  $\wt{\fsl_N}(\C_q)$ coordinated by $\C_q$ by definition is the extension of
 $\wh{\fsl_N}(\C_q)$ by  the two canonical degree-zero derivations (see \cite{BGK}).
 Representations of $\wh{\fsl_N}(\C_q)$ and $\wt{\fsl_N}(\C_q)$ have been extensively
 studied in literature (cf. \cite{BS,BGT,G1,G2,G3,GZ,ER,CT}).
 In particular,  a fermionic construction of the basic modules for  $\wh{\fsl_N}(\C_q)$
was obtained by Gao in \cite{G3} and integrable highest weight modules for $\wh{\fsl_N}(\C_q)$
were studied by Rao (see \cite{ER}).
As for this paper, our focus will be on the Lie algebra $\wh{\fsl_N}(\C_q)$.

Note that the EALAs coordinated by irrational quantum tori cannot be directly associated to modules or twisted modules
for vertex algebras, due to the fact that their canonical generating functions are not ``local.''
Then we come to a theory of what were called quasi modules for vertex algebras.
The notion of quasi module was introduced in \cite{li-gamma} to associate vertex algebras to a certain family of Lie algebras.
Indeed, with this new theory a much wider variety of Lie algebras, including the quantum $2$-torus Lie algebra,
can be associated with vertex algebras.
The theory of quasi modules was developed further in \cite{li-tlie}, where
a notion of vertex $\Gamma$-algebra with $\Gamma$ a group was introduced and
an enhanced notion of quasi module, called equivariant quasi module for a vertex $\Gamma$-algebra,  was introduced.
The notion of quasi module from definition generalizes that of module. On the other hand,
this notion also generalizes the notion of twisted module in a certain natural way (see  \cite{li-twisted-quasi}).

In this paper, we shall intensively study a natural connection of the Lie algebra $\wh{\fsl_N}(\C_q)$ with
vertex algebras and their equivariant quasi modules. As the first step, we relate $\wh{\fsl_N}(\C_q)$ to a (general) affine Lie algebra.
Let $\fgl_{\infty}$ be the Lie algebra of infinite order complex matrices with only finitely many nonzero entries and
set $\fsl_{\infty}=[\fgl_{\infty},\fgl_{\infty}]$ (the derived subalgebra).
Equip $\fgl_{\infty}$ with a suitable nondegenerate symmetric invariant bilinear form.
Then we have an affine Lie algebra
$\widehat{\fsl_{\infty}}=\fsl_{\infty}\otimes \C[t,t^{-1}]\oplus \C {\bf k}$.
(Note that $\wh{\fsl_\infty}$ is different from the infinite rank affine Kac-Moody algebra $\overline{\fsl_\infty}$
which is a completion of $\fsl_\infty$ (see \cite{Kac}).)
For any complex number $\ell$, we have a vertex algebra $V_{\widehat{\fsl_{\infty}}}(\ell,0)$
and its simple quotient vertex algebra $L_{\wh{\fsl_{\infty}}}(\ell,0)$.
We show that $\wh{\fsl_N}(\C_q)$ has an intrinsic connection with the affine Lie algebra $\wh{\fsl_\infty}$.
As the first main result, using this  intrinsic connection we establish a canonical category isomorphism
between the categories of
restricted $\wh{\fsl_N}(\C_q)$-modules of level $\ell$ and equivariant quasi modules for $V_{\wh{\fsl_\infty}}(\ell,0)$.
Furthermore, assuming that $\ell$ is a positive integer, we show that under the canonical category isomorphism,
integrable and restricted $\wh{\fsl_N}(\C_q)$-modules of level $\ell$ correspond exactly to
equivariant quasi $L_{\wh{\fsl_\infty}}(\ell,0)$-modules.

In representation theory, one of the prominent notions is that of dual pair
which was first studied by Howe for reductive groups (see \cite{H1,H2}).
In \cite{F}, Igor Frenkel discovered a duality for affine Lie algebras of type $A$,
which is now commonly known as the level-rank duality.
Wang in \cite{W2} obtained certain dualities
between a completed infinite rank affine Kac-Moody algebra and certain reductive Lie groups.
 An interesting level-rank duality was also found for vertex operator algebras
of types $A$ by Jiang and Lin in \cite{JLin} and $B,D$ by Jiang and Lam in \cite{JLam}.
In \cite{VV1}, the  usual Schur $(\fgl_N(\C),S_\ell)$-duality was generalized to the two-parameters quantum toroidal
algebra and the classical limit (see \cite{VV2}) gives rise to a duality between  $\wh{\fsl_N}(\C_q)$ and
a distinguished double affine Hecke algebra.

As the second main result of this paper, we obtain a level-rank type duality between irreducible integrable highest weight
$\wh{\fsl_N}(\C_q)$-modules of level $\ell\in \Z_+$ and irreducible regular modules for Levi subgroups $\rGL_{\bf I}$
of $\mathrm{GL}_\ell(\C)$, where the parameter ${\bf I}$ is a partition of $\{1,\dots,\ell\}$.
To a certain extent, this duality is analogous to the classical skew
$(\fgl_N(\C),\mathrm{GL}_\ell(\C))$-duality (see \cite{H2}).
By using this $(\wh{\fsl_N}(\C_q),\rGL_{\bf I})$-duality, we furthermore
obtain a fermionic realization of every irreducible integrable highest weight $\wh{\fsl_N}(\C_q)$-module.
Consequently, we obtain a realization of every irreducible $\N$-graded equivariant quasi
$L_{\wh{\fsl_\infty}}(\ell,0)$-module.

Note that one of important applications of dual pair is to study branchings.
For a pair $\rH\subset \rG$ of (complex) reductive  groups,
the branching law from $\rG$ to $\rH$  is  a description of the irreducible $\rH$-modules and their multiplicities that occur
in the decomposition of each irreducible regular $\rG$-module.
Among the better known classical branchings are the branching from $\rH\times \rH$ to $\rH$ and
 the branching from a group $\rG$ to a Levi subgroup $\rH$.
 Analogously, the two better branchings in vertex operator algebra theory are
the branchings from
$L_{\wh{\fg}}(\ell,0)\ot L_{\wh{\fg}}(\ell',0)$ to $L_{\wh{\fg}}(\ell+\ell',0)$ and
from $L_{\wh{\fg}}(\ell,0)$ to $L_{\wh{\fh}}(\ell,0)$,
where $\ell, \ell'$ are positive integers and $\fh$ is a Levi subalgebra of $\fg$.
On the other hand, Igor Frenkel studied  in \cite{F} the $\wh\fg\rightarrow \wh\fg^{(M)}$ branching,
where
$\wh\fg^{(M)}=\fg\ot \C[t^M,t^{-M}]\oplus \C \bm\rk\subset \wh\fg$ with $M$ a positive integer.
Frenkel made a conjecture on the decomposition of the basic modules and he confirmed the conjecture for $\fg=\fsl_N$
by using the level-rank duality for $\wh{\fsl_N}$.

 As the third part of this paper, we study three analogous branchings.
More specifically, we determine the following branchings:
\te{(i)} The branching from
  $L_{\wh{\fsl_\infty}}(\ell,0)$
   $\ot L_{\wh{\fsl_\infty}}(\ell',0)$ to $L_{\wh{\fsl_\infty}}(\ell+\ell',0)$ on equivariant quasi modules.
  \te{(ii)} The branching from $\wh{\fsl_N}(\C_q)$ to its Levi subalgebras.
  \te{(iii)} The branching from $\wh{\fsl_N}(\C_q)$
  to its subalgebra $\wh{\fsl_N}(\C_q[t_0^{\pm M_0},t_1^{\pm M_1}])$.
 It turns out that all the multiplicities in these branchings are \emph{finite}, which
 can be calculated by the Littlewood-Richardson rule (cf. \cite{GW}).
 This is quite different from the branching from
  $L_{\wh{\fg}}(\ell,0)$  $\ot L_{\wh{\fg}}(\ell',0)$ to $L_{\wh{\fg}}(\ell+\ell',0)$
  on ordinary modules with $\wh{\fg}$ a finite rank affine Kac-Moody algebra, where
the multiplicities are in general infinite and determining such branchings is an important and difficult problem
(cf. \cite{JLin, JLam, KMPX}). 

Now, we start a more detailed introduction section by section.
In Section 2, we recall the Lie algebra $\wh{\fsl_{N}}(\C_q)$ and establish a canonical isomorphism
between $\wh{\fsl_{N}}(\C_q)$ and what was called
the covariant algebra of the affine Lie algebra $\wh{\fsl_{\infty}}$ (see \cite{gkk}, \cite{li-gamma}).
Recall that associated to any Lie algebra $\fg$ with a nondegenerate symmetric invariant bilinear
form $\langle\cdot,\cdot\rangle$ we have an affine Lie algebra $\wh\fg$.
Furthermore, let $G$ be an automorphism group of $\fg$, which preserves the form $\langle\cdot,\cdot\rangle$,
and let $\chi: G\rightarrow \C^{\times}$ be a linear character such that for any $a,b\in \fg$,
$$[ga,b]=0\  \mbox{ and }\  \langle ga,b\rangle=0 \   \mbox{ for all but finitely many }g\in G.$$
The  covariant algebra $\wh\fg[G]$ of the affine Lie algebra $\wh\fg$ is defined to be the Lie algebra
with  $\wh\fg$ as the generating space space, subject to the following relations
\begin{align*}
g(a\otimes t^n)&=\chi(g)^{n}ga\otimes t^n, \\
[a\otimes t^m,b\otimes t^n]&=\sum_{g\in G}\chi(g)^{m}\left([ga,b]\otimes t^{m+n}+m\delta_{m+n,0}\langle ga,b\rangle {\bf k}\right)
\end{align*}
for $g\in G,\ a,b\in \fg,\ m, n\in \Z$, and $[{\bf k},\wh\fg]=0$.
For the Lie algebra $\fgl_{\infty}$, take the automorphism group
$\Z=\langle \sigma^N\rangle$, where $\sigma E_{m,n}=E_{m+1,n+1}$ for $m,n\in \Z$,
together with the linear character $\chi_q$ defined by $\chi_q(n)=q^{n}$ for $n\in \Z$.
It is proved that $\wh{\fsl_{N}}(\C_q)$ is canonically isomorphic to the covariant algebra $\wh{\fsl_{\infty}}[\Z]$.

In Section 3, we give a precise connection between $\wh{\fsl_N}(\C_q)$-modules of level $\ell$
and equivariant quasi modules for vertex algebras $V_{\wh{\fsl_\infty}}(\ell,0)$ and $L_{\wh{\fsl_\infty}}(\ell,0)$.
We first recall the notion of quasi module for a  vertex algebra $V$ and some basic results.
The notion of quasi module is defined by simply replacing
the Jacobi identity in the definition of a module with the ``quasi Jacobi identity''
stating that for any $u,v\in V$,  the usual Jacobi identity after multiplied by
a nonzero polynomial depending on $u,v$ holds.
For the notion of vertex $\Gamma$-algebra,
$\Gamma$ is a group with  a linear character $\chi: \Gamma\rightarrow \C^{\times}$, where
a {\em vertex $\Gamma$-algebra} is a vertex algebra $V$ equipped with a representation
$R$ of $\Gamma$ on $V$ such that $R_{g}({\bf 1})={\bf 1}$ for $g\in \Gamma$ and
 $$R_gY(v,x)R_g^{-1}=Y(R_gv,\chi(g)^{-1}x)\   \   \   \mbox{ for }g\in \Gamma,\ v\in V.$$
Note that $\Gamma$ does not act on $V$ as an automorphism group in general. On the other hand,
suppose $V$ is a $\Z$-graded vertex algebra in the sense that $V$ is a vertex algebra with
a $\Z$-grading $V=\oplus_{n\in \Z}V_{(n)}$ such that ${\bf 1}\in V_{(0)}$ and
$$u_{r}V_{(n)}\subset V_{(m+n-r-1)}\   \   \   \mbox{ for }u\in V_{(m)},\ m,r,n\in \Z.$$
Then  for any automorphism group $\Gamma$ of $V$ such that $gV_{(n)}\subset V_{(n)}$ for $g\in \Gamma,\ n\in \Z$
 and  for any linear character $\chi$ of $\Gamma$,
$V$ becomes a vertex $\Gamma$-algebra with $R_{g}=\chi(g)^{-L(0)}g$ for $g\in \Gamma$,
where $L(0)$ denotes the grading operator on $V$.
For a vertex $\Gamma$-algebra $V$,
an equivariant quasi module is a quasi module $(W,Y_W)$ satisfying the conditions that
$$Y_{W}(R_gv,x)=Y_W(v,\chi(g)x)\  \  \  \mbox{ for }g\in \Gamma,\ v\in V$$
and that for $u,v\in V$,  there exists a nonzero polynomial
$f(z)$ whose roots are contained in $\chi(\Gamma)$ such that
$$f(x_1/x_2)[Y_{W}(u,x_1),Y_W(v,x_2)]=0.$$
For any complex number $\ell$, we have a $\Z$-graded vertex algebra $V_{\widehat{\fsl_{\infty}}}(\ell,0)$ and
its simple quotient $L_{\wh{\fsl_{\infty}}}(\ell,0)$. On the other hand, the automorphism $\sigma$
of the Lie algebra $\fsl_{\infty}$ induces an automorphism of $V_{\widehat{\fsl_{\infty}}}(\ell,0)$ and $L_{\wh{\fsl_{\infty}}}(\ell,0)$.
Take $\Z=\langle\sigma^N\rangle$ with the linear character $\chi_q$.
Then $V_{\widehat{\fsl_{\infty}}}(\ell,0)$ and $L_{\wh{\fsl_{\infty}}}(\ell,0)$
are both vertex $(\Z,\chi_q)$-algebras.
As the main results of this section, we establish a canonical category isomorphism between the categories of
restricted $\wh{\fsl_N}(\C_q)$-modules of level $\ell$ and equivariant quasi modules for $V_{\wh{\fsl_\infty}}(\ell,0)$.
Furthermore, assuming that $\ell$ is a positive integer, we prove that under the canonical category isomorphism,
integrable and restricted $\wh{\fsl_N}(\C_q)$-modules of level $\ell$ correspond exactly to
equivariant quasi $L_{\wh{\fsl_\infty}}(\ell,0)$-modules.

In Section 4,  we classify irreducible $\N$-graded equivariant quasi modules
for $L_{\wh{\fsl_\infty}}(\ell,0)$ with $\ell$ any positive integer.
First,  we show that every
irreducible $\N$-graded equivariant quasi $L_{\wh{\fsl_\infty}}(\ell,0)$-module is
an irreducible integrable highest weight  $\wh{\fsl_N}(\C_q)$-module of level $\ell$.
Then, slightly generalizing a result of Rao in \cite{ER}
we obtain a classification of irreducible integrable highest weight  $\wh{\fsl_N}(\C_q)$-modules
of level $\ell$, which
 are parameterized by a positive integer $k$ and a pair
\[(\bm\lambda,\bm{c})=((\lambda_1,\dots,\lambda_k),(c_1,\dots,c_k))\in P_+^k\times (\C^\times)^k\]
such that $\lambda_1+\cdots+\lambda_k$ is of level $\ell$,
where $P_+$ denotes the set of dominant integral weights for the affine Lie algebra $\wh{\fsl_N}$.
From the canonical category isomorphism given in Section 3, we obtain an explicit classification of
irreducible $\N$-graded equivariant quasi $L_{\wh{\fsl_\infty}}(\ell,0)$-modules.

In Sections 5 and 6, we present a level-rank duality between irreducible integrable highest weight $\wh{\fsl_N}(\C_q)$-modules
 of level $\ell$ and finite-dimensional irreducible regular modules for an arbitrary Levi subgroup of the general linear group $\rGL_\ell=\mathrm{GL}_\ell(\C)$.
 As an application, we obtain an explicit realization of every irreducible  $\N$-graded equivariant quasi $L_{\wh{\fsl_\infty}}(\ell,0)$-module.
More specifically,  set
\begin{align*}
(\C^\times)_q^\ell=\{(a_1,\dots,a_\ell)\in (\C^\times)^\ell\mid \te{either}\ a_i=a_j \ \te{or}\ a_i\notin a_j\Gamma_q\ \te{for}\ 1\le i,j\in \ell\},
\end{align*}
where $\Gamma_q=\{ q^n\ |\ n\in \Z\}$.
For any $\bm{a}\in (\C^\times)_q^\ell$,
by using a result of Gao (see \cite{G3}) and the Chari-Pressely evaluation module construction (see \cite{CP}),
we construct a fermionic vacuum module $\mathcal F_N^{\bm{a}}$ for  $\wh{\fsl_N}(\C_q)$ of level $\ell$.
On the other hand, for any partition ${\bf I}$ of $\{1,\dots,\ell\}$,
denote by $\rGL_{\bf I}$ the Levi subgroup of $\rGL_\ell$ associated to $\bf I$.
From \cite{FF}, $\mathcal F_N^{\bm{a}}$ is naturally a  locally regular $\rGL_{\bf I}$-module.
Using Frenkel's level-rank duality, we prove that for any $\bm a\in (\C^\times)_q^\ell$, there is a partition
${\bf I}_{\bm a}$ of $\{1,\dots,\ell\}$ associated to $\bm{a}$ such that $(\wh{\fsl_N}(\C_q), \rGL_{{\bf I}_{\bm a}})$
 (and hence $(L_{\wh{\fsl_\infty}}(\ell,0),\rGL_{{\bf I}_{\bm a}})$)
 form a dual pair on $\mathcal F_N^{\bm{a}}$ in the sense of Howe (see \cite{H1}).
We also determine  the irreducible isotypic decomposition
of $\mathcal F_N^{\bm{a}}$ as an $(\wh{\fsl_N}(\C_q), \rGL_{{\bf I}_{\bm a}})$-module.
Furthermore, it is proved that every finite-dimensional irreducible regular module of
every Levi subgroup of $\rGL_\ell$ occurs in  $\mathcal F_N^{\bm{a}}$ and
every  irreducible integrable highest weight $\wh{\fsl_N}(\C_q)$-module of level $\ell$
 occurs in $\mathcal F_N^{\bm{a}}$ for a suitable choice of $\bm{a}$.

In Sections 7, 8, and 9,  we study the following branchings respectively:
\te{(i)} The branching from
  $L_{\wh{\fsl_\infty}}(\ell,0)\ot L_{\wh{\fsl_\infty}}(\ell',0)$ to $L_{\wh{\fsl_\infty}}(\ell+\ell',0)$.
  \te{(ii)} The branching from $\wh{\fsl_N}(\C_q)$ to its Levi subalgebras.
  \te{(iii)} The branching from $\wh{\fsl_N}(\C_q)$
  to its subalgebra $\wh{\fsl_N}(\C_q[t_0^{\pm M_0},t_1^{\pm M_1}])$.
Note that the first branching is about quasi modules and it amounts to the tensor product decomposition of
irreducible integrable highest weight $\wh{\fsl_N}(\C_q)$-modules.
By employing the reciprocity laws associated to certain seesaw pairs related to
the dual pair $(\wh{\fsl_N}(\C_q),\rGL_{{\bf I}_{\bm a}})$, we completely determine these three branchings.
More specifically, in Section 7,  for every irreducible $\N$-graded
equivariant quasi $L_{\wh{\fsl_\infty}}(\ell,0)\ot L_{\wh{\fsl_\infty}}(\ell',0)$-module $W$, we show that
$W$ as a quasi $L_{\wh{\fsl_\infty}}(\ell+\ell',0)$-module is completely reducible and
we determine the irreducible quasi $L_{\wh{\fsl_\infty}}(\ell+\ell',0)$-modules and their multiplicities in $W$.
We also show that the branching from
$L_{\wh{\fsl_\infty}}(\ell,0)\ot L_{\wh{\fsl_\infty}}(1,0)$ to $L_{\wh{\fsl_\infty}}(\ell+1,0)$ is multiplicity free.
Similarly, we completely determine the other two branchings.

For this paper, we use symbols $\Z_{+}$ and $\N$ for the sets of positive integers and nonnegative integers, respectively.
 We work on the field $\C$ of complex numbers.

\section{Extended affine Lie algebras coordinated by quantum tori}

In this section, we first recall the EALA (and its core) of type $A_{N-1}$ coordinated by the $2$-dimensional quantum torus
and then we give a realization as a covariant Lie algebra of the affine Lie algebra of  $\fsl_\infty$.

\subsection{The extended affine Lie algebra $\wh{\fsl_N}(\C_q)$}
Let $N\ge 2$ be a positive integer. For any unital associative algebra $\mathcal A$ (over $\C$),
we denote by $\fgl_N(\mathcal A)$
 the associative algebra of all $N\times N$ matrices with entries in $\mathcal A$.
Naturally, $\fgl_N(\mathcal A)$ is a Lie algebra with commutator as its Lie bracket. Note that
$\fgl_N(\mathcal A)$ is naturally a (left) $\mathcal A$-module (with each $a\in \mathcal A$ as a scalar).
Set
\begin{eqnarray}
\fsl_N(\mathcal A)=[\fgl_N(\mathcal A),\fgl_N(\mathcal A)],
\end{eqnarray}
the derived Lie subalgebra.
For $1\le i,j\le N$,  $a\in \mathcal A$, we also write $E_{i,j}a=aE_{i,j}$,
the matrix whose only nonzero entry is the $(i,j)$-entry which is $a$.

Let  $q$ be a nonzero complex number, which is fixed throughout this paper.
Denote by $\C_q[t_0^{\pm 1}, t_1^{\pm 1}]$ the $2$-dimensional quantum torus
associated to $q$. By definition,
\[\C_q[t_0^{\pm 1}, t_1^{\pm 1}]=\C[t_0^{\pm 1}, t_1^{\pm 1}]\]
as a vector space, where
\[(t_0^{m_0}t_1^{m_1})(t_0^{n_0}t_1^{n_1})=q^{m_1n_0}t_0^{m_0+n_0}t_1^{m_1+n_1}\]
for $m_0,m_1,n_0,n_1\in \Z$. From now on, we simply write $\C_q$ for $\C_q[t_0^{\pm 1}, t_1^{\pm 1}]$.
Throughout this paper, we assume that $q$ is \emph{not a root of unity}.
Note that this assumption amounts to that $\C_q$ is a simple associative algebra.

Consider the following two-dimensional central extension of Lie algebra $\mathfrak{gl}_N(\C_q)$:
\begin{align}
\wh{\fgl_N}(\C_q)=\mathfrak{gl}_N(\C_q)\oplus \C \bm{\rk}_0\oplus \C\bm{\rk}_1,
\end{align}
where $\bm\rk_0,\bm\rk_1$ are (linearly independent) central elements, and
\begin{equation}\begin{split}\label{commutator1}
&[E_{i,j} t_0^{m_0}t_1^{m_1}, E_{k,l} t_0^{n_0}t_1^{n_1}]\\
=&\ \delta_{j,k}q^{m_1n_0}E_{i,l} t_0^{m_0+n_0}t_1^{m_1+n_1}-\delta_{i,l}q^{n_1m_0}E_{k,j} t_0^{m_0+n_0}t_1^{m_1+n_1}\\
&\ +\delta_{j,k}\delta_{i,l}\delta_{m_0+n_0,0}\delta_{m_1+n_1,0}
q^{m_1n_0}(m_0\bm\rk_0+m_1\bm\rk_1)
\end{split}\end{equation}
for $1\le i,j,k,l\le N$ and for $m_0,m_1,n_0,n_1\in \Z$.

It is clear that $\wh{\fgl_N}(\C_q)$ is a $\Z\times \Z$-graded Lie algebra with
$$\deg \bm\rk_0=(0,0),\   \   \deg \bm\rk_1=(0,0),\   \   \deg (E_{i,j} t_0^{m_0}t_1^{m_1})=(m_0,m_1)$$
for $1\le i,j\le N,\ m_0,m_1\in \Z$.
Define degree-zero derivations $\bm\rd_0, \bm\rd_1$ of $\wh{\fgl_N}(\C_q)$ by
\begin{align}\label{def-derivations-d}
\bm\rd_r (E_{i,j} t_0^{m_0}t_1^{m_1})=m_rE_{i,j} t_0^{m_0}t_1^{m_1}
\end{align}
for $r=0,1$ and for $1\le i,j\le N, m_0,m_1\in \Z$.
Adding derivations $\bm\rd_0, \bm\rd_1$ to $\wh{\fgl_N}(\C_q)$, we obtain a  Lie algebra
\begin{eqnarray}
\wt{\fgl_N}(\C_q)=\wh{\fgl_N}(\C_q)\oplus \C\bm\rd_0\oplus \C\bm\rd_1.
\end{eqnarray}
Set
\begin{eqnarray}
\wh{\fsl_N}(\C_q)=[\wh{\fgl_N}(\C_q),\wh{\fgl_N}(\C_q)]=\mathfrak{sl}_N(\C_q)\oplus \C\bm\rk_0\oplus \C\bm\rk_1,
\end{eqnarray}
the  derived subalgebra of $\wh{\fgl_N}(\C_q)$.
It is known  (see \cite{BGK}) that $\wh{\fsl_N}(\C_q)$ is a universal central extension of $\fsl_N(\C_q)$ and
\begin{align}\label{decwhfgl}
\wh{\fgl_N}(\C_q)=\wh{\fsl_N}(\C_q)\oplus \C I_N,
\end{align}
 where $I_N=E_{1,1}+\cdots+E_{N,N}$ (the identity matrix).

Set
\begin{eqnarray}
\wt{\fsl_N}(\C_q)=\wh{\fsl_N}(\C_q)\oplus \C\bm\rd_0\oplus \C\bm\rd_1\subset \wt{\fgl_N}(\C_q).
\end{eqnarray}
The Lie algebra $\wt{\fsl_N}(\C_q)$ is known to be a nullity-$2$ extended affine Lie algebra of type $A_{N-1}$,
where $\wh{\fsl_N}(\C_q)$ is called the core of $\wt{\fsl_N}(\C_q)$. We refer the reader to \cite{BGK} or \cite{AABGP} for details.

In this paper, we shall focus on the subalgebra $\wh{\fsl_N}(\C_q)$. From \cite{BGK} we have:

\begin{lemt}\label{basis-slCq}
The following elements form a basis of $\wh{\fsl_N}(\C_q)$:
\begin{align}\label{basis1}
E_{i,j}t_0^{m_0}t_1^{m_1},\quad E_{r,r}-E_{r+1,r+1},\quad \bm\rk_0,\quad \bm\rk_1,
\end{align}
where $1\le i,j\le N,\ m_0,m_1\in \Z$ with $(i-j,m_0,m_1)\ne (0,0,0)$ and $1\le r\le N-1$.
\end{lemt}

\subsection{The $(\Z,\chi_q)$-covariant algebra $\wh{\fsl_\infty}[\Z]$ of $\wh{\fsl_\infty}$}
We start with the definition of a covariant algebra of an affine Lie algebra.
Let $\fg$ be a (possibly infinite-dimensional) Lie algebra equipped with a symmetric
invariant bilinear form $\<\cdot,\cdot\>$.
Associated to the pair $(\fg, \<\cdot,\cdot\>)$, we have an affine Lie algebra
\begin{align*}
\wh\fg=\fg\ot \C[t,t^{-1}]\oplus \C \bm\rk,
\end{align*}
where $\bm\rk$ is a (nonzero) central element, and for $a,b\in \fg$, $m,n\in \Z$,
\begin{align*}
[a\ot t^m, b\ot t^n]=[a,b]\ot t^{m+n} + m\delta_{m+n,0}\<a,b\> \bm\rk.
\end{align*}
Assume that $\Gamma$ is a group acting on $\fg$ as an automorphism group preserving the bilinear form
$\<\cdot,\cdot\>$ such that
for $a,b\in \fg$,
\begin{align}\label{conconvar}
[ga,b]=0\quad \te{and}\quad \<ga,b\>=0\quad \te{for all but finitely many}\ g\in \Gamma.
\end{align}
Let $\chi:\Gamma\rightarrow \C^\times$ be any linear character.
We lift the $\Gamma$-action from $\fg$ to  $\wh\fg$ by
\begin{align}
g(a\ot t^m+\beta \bm\rk)=\chi(g)^m(ga\ot t^m)+\beta\bm\rk
\end{align}
for $g\in \Gamma$, $a\in \fg$, $m\in \Z$, $\beta\in \C$.
From \cite{li-tlie}, the {\em $(\Gamma,\chi)$-covariant algebra}  of $\wh\fg$
is the Lie algebra $\wh\fg[\Gamma]$, where
\begin{align}\label{kgamma}
\wh\fg[\Gamma]=\wh\fg/{\rm span}\{gu-u\mid g\in \Gamma, u\in \wh\fg\}
\end{align}
as a vector space and the Lie bracket  is given by
\begin{align*}
[\overline{u},\overline{v}]=\sum_{g\in \Gamma} \overline{[gu,v]}\quad \te{for}\ u,v\in \wh\fg.
\end{align*}
Here and below, for $u\in \wh\fg$, $\overline{u}$ stands for the image of $u$ in $\wh\fg[\Gamma]$
under the natural quotient map
$\wh\fg\rightarrow \wh\fg[\Gamma]$.
Note that  for $a,b\in \fg, \  m,n\in \Z$ and $g\in \Gamma$, we have
\begin{align}\label{relation2}\overline{a\ot t^m}&=\chi(g)^m\overline{ga\ot t^m},\\
\label{commutator2}
[\overline{a\ot t^m},\overline{b\ot t^n}]&=\sum_{g\in \Gamma}\chi(g)^{m}\(\overline{[ga,b]\ot t^{m+n}}+m\delta_{m+n,0}\<ga,b\>\bm\rk\),
\end{align}
where we still denote $\overline{\bm\rk}$ by $\bm\rk$ for simplicity.

The following is straightforward:

\begin{lemt}\label{linear-algebra-fact}
Let $\mathfrak {g}$ be a Lie algebra, $\Gamma$ a group as above.
Suppose that  $\mathfrak{g}_0$ is a subalgebra of $\mathfrak{g}$ which is stable under the action of $\Gamma$.
Then the embedding map of $\mathfrak{g}_0$ into $\mathfrak{g}$ gives rise to a
Lie algebra homomorphism $\psi: \wh{\mathfrak{g}}_0[\Gamma]\rightarrow \wh{\mathfrak{g}}[\Gamma]$.
\end{lemt}

We also formulate a simple lemma  which we shall use to determine a basis of the covariant Lie algebra $\wh{\fg}[\Gamma]$.
Let $G$ be a group and let $U$ be a $G$-module. Define $U/G$ to be the
quotient space of $U$ modulo the subspace spanned by $\{gu-u\ |\ g\in G,\ u\in U\}$.

\begin{lemt}\label{linear-algebra-fact}
Let $G$ be a group, $U$ a $G$-module.
(a) If $U=\oplus_{\alpha\in S}U_{\alpha}$ as a $G$-module, then $U/G\simeq \oplus_{\alpha\in S}U_{\alpha}/G$.
(b) Suppose that $U$ has a basis $B$ with a subset $B_0$ satisfying the condition that
for every $b\in B$,  there exist unique $g\in G,\ b_0\in B_0$ such that  $b\in \C gb_0$.
Then $B_0$ gives rise to a basis for the quotient space $U/G$.
\end{lemt}

\begin{proof} Part (a) follows immediately from a standard result in linear algebra.
For Part (b), from the assumption,
$\{ g(b_0)\ |\ g\in G,\ b_0\in B_0\}$ is also a basis of $U$. It then follows that
$B_0\cup \{g(b_0)-b_0\ |\ g\in G, g\ne e,\ b_0\in B_0\}$ is also a basis of $U$.
Note that $g(hb_0)-hb_0=(ghb_0-b_0)-(hb_0-b_0)$ for $g,h\in G,\ b_0\in B_0$.
Using this and using the basis $\{ g(b_0)\ |\ g\in G,\ b_0\in B_0\}$ of $U$,
we see that $\{g(b_0)-b_0\ |\ g\in G, g\ne e,\ b_0\in B_0\}$ is a basis of
the span of $\{ gu-u\ |\ g\in G,\ u\in U\}$. Consequently, $B_0$ gives rise to a basis for the quotient space $U/G$.
\end{proof}

Let $\fgl_\infty$ be the associative algebra of all doubly infinite complex matrices
with only finitely many nonzero entries, which is also naturally a Lie algebra.
As before, for $m,n\in \Z$, let $E_{m,n}$ denote the unit matrix whose only nonzero entry is the $(m,n)$-entry which is $1$.
Equip $\fgl_\infty$ with the bilinear form  $\<\cdot,\cdot\>$  defined by
\begin{align*}
\<E_{i,j},E_{k,l}\>= \delta_{j,k}\,\delta_{i,l}\quad\te{for}\ i,j,k,l\in \Z,
\end{align*}
which is nondegenerate, symmetric and (associative) invariant.
Set
\begin{align}
\fsl_\infty=[\fgl_\infty,\fgl_\infty],
\end{align}
the derived subalgebra of $\fgl_\infty$.
We see that  $\<\cdot,\cdot\>$ is also nondegenerate on $\fsl_\infty$.
 Then we have the affine Lie algebra $\wh{\fsl_\infty}$
associated to the pair $(\fsl_\infty, \<\cdot,\cdot\>)$.

\begin{dfnt}\label{sigma-def}
{\em Let $\sigma$ be the automorphism of  the  algebra $\fgl_{\infty}$ defined by
\begin{align}\label{sigma-definition}
\sigma(E_{m,n})=E_{m+1,n+1}\   \   \   \   \mbox{ for }m,n\in \Z.
\end{align}}
\end{dfnt}

The following is straightforward:

\begin{lemt}\label{lem:covalg}
The automorphism $\sigma$ of $\fgl_\infty$ preserves the bilinear form $\<\cdot,\cdot\>$
and the Lie subalgebra $\fsl_\infty$. Furthermore, for any positive integer $N$,  the map
\begin{align}\label{defthetaN}
\rho_N:\Z\rightarrow {\rm Aut} (\fgl_\infty)\ (\te{resp}.\ {\rm Aut}(\fsl_\infty)),\quad r\mapsto \sigma^{Nr}\quad(r\in \Z)
\end{align}
 is a one-to-one group homomorphism, and for any $a,b\in \fgl_\infty$ (resp. $\fsl_\infty$),
\begin{align*}
[\sigma^{r}(a),b]=0\quad \te{and}\quad \<\sigma^{r}(a),b\>=0\quad
\te{for all but finitely many}\ r\in \Z.
\end{align*}
\end{lemt}

From now on, we fix the group actions of $\Z$ on $\fgl_\infty$ and $\fsl_\infty$ via the homomorphism $\rho_N$.
Define a linear character $\chi_q:  \Z\rightarrow \C^\times$ by
\begin{align}\label{chiq}
\chi_q (r)=q^r\   \   \   \mbox{ for }r\in \Z.
\end{align}
Then we have the $(\Z,\chi_q)$-covariant algebras $\wh{\fgl_\infty}[\Z]$ and $\wh{\fsl_\infty}[\Z]$.
In what follows, we shall show that Lie algebra $\wh{\fsl_N}(\C_q)$ is isomorphic to  $\wh{\fsl_\infty}[\Z]$.

Define a linear map $\theta_{N,q}:\wh{\fgl_N}(\C_q)\rightarrow \wh{\fgl_\infty}[\Z]$ by
\begin{align}
\theta_{N,q} (E_{i,j}t_0^{m_0}t_1^{m_1})&= \overline{E_{N m_1+i,j}\ot t^{m_0}},\quad
\theta_{N,q} (\bm\rk_0)=\bm\rk,\quad \theta_{N,q} (\bm\rk_1)= 0
\end{align}
for $1\le i,j\le N$, $m_0,m_1\in \Z$. Then we have:

\begin{prpt}\label{lem:thetanq}
The linear map $\theta_{N,q}:\wh{\fgl_N}(\C_q)\rightarrow \wh{\fgl_\infty}[\Z]$
is a surjective Lie homomorphism with $\ker (\theta_{N,q})=\C\bm\rk_1$.
\end{prpt}

\begin{proof} By Lemma \ref{linear-algebra-fact}, we see  that the central element $\bm\rk$ together with the elements
\[\overline{E_{N m_1+i,j}\ot t^{m_0}}\quad \te{for}\  1\le i,j\le N,\   m_0,m_1\in \Z,\]
form a basis of $\wh{\fgl_\infty}[\Z]$. Thus $\theta_{N,q}$ is a surjective linear map with $\ker (\theta_{N,q})=\C\bm\rk_1$.
Let $1\le i,j,k,l\le N$ and $m_0,m_1,n_0,n_1\in \Z$.
By definition (see \eqref{commutator2}) we have
\begin{align}
\nonumber&[\overline{E_{N m_1+i,j}\ot t^{m_0}},\overline{E_{N n_1+k,l}\ot t^{n_0}}]\\
\nonumber=&\, \sum_{r\in \Z}\chi_q(r)^{m_0}\Big(\overline{[E_{N (m_1+r)+i,Nr+j}\ot t^{m_0},E_{N n_1+k,l}\ot t^{n_0}]}\\
\nonumber&\quad +m_0\delta_{m_0+n_0,0}\<E_{N (m_1+r)+i,Nr+j},E_{N n_1+k,l}\>\bm\rk\Big)\\
\nonumber=&\ \delta_{j,k}q^{n_1m_0}\overline{E_{N (m_1+n_1)+i,l}\ot t^{m_0+n_0}}
-\delta_{i,l}q^{-m_1m_0}\overline{E_{N n_1+k,-Nm_1+j}\ot t^{m_0+n_0}}\\
\nonumber&\quad+m_0q^{m_0n_1}\delta_{m_0+n_0,0}\delta_{m_1+n_1,0}\delta_{j,k}\delta_{i,l}\bm\rk\\
\label{commutator3}=&\ \delta_{j,k}q^{n_1m_0}\overline{E_{N (m_1+n_1)+i,l}\ot t^{m_0+n_0}}
-\delta_{i,l}q^{m_1n_0}\overline{E_{N (m_1+n_1)+k,j}\ot t^{m_0+n_0}}\\
\nonumber&\quad+m_0q^{m_0n_1}\delta_{m_0+n_0,0}\delta_{m_1+n_1,0}\delta_{j,k}\delta_{i,l}\bm\rk,
\end{align}
noticing that (see \eqref{relation2} and \eqref{chiq})
\begin{align*}
&\overline{E_{N n_1+k,-Nm_1+j}\ot t^{m_0+n_0}}
=\chi_q(m_1)^{m_0+n_0}\overline{\sigma_{N,m_1}(E_{N n_1+k,-Nm_1+j})\ot t^{m_0+n_0}}\\
=&\ q^{(m_0+n_0)m_1}\overline{E_{N (n_1+m_1)+k,j}\ot t^{m_0+n_0}}.
\end{align*}
It then follows from  \eqref{commutator1} and \eqref{commutator3} that $\theta_{N,q}$ is a Lie algebra homomorphism.
\end{proof}

Now, we continue to study Lie algebra $\wh{\fsl_\infty}[\Z]$.
For $1\le i,j\le N$, $m_0,m_1\in \Z$ with $(i,m_0,m_1)\ne (j,0,0)$, set
\begin{align*}
&e_{i,j}(m_0,m_1)= \overline{E_{Nm_1+i,j}\ot t^{m_0}}\   \   \   \te{if}\ (i,m_1)\ne (j,0),\\
&e_{i,i}(m_0,0)=\frac{1}{1-q^{-m_0}}\overline{(E_{i,i}-E_{N+i,N+i})\ot t^{m_0}}\  \  \   \te{if}\ m_0\ne 0.
\end{align*}
On the other hand, set
\begin{align*}
\bm\rk'=\overline{E_{N+1,N+1}-E_{1,1}}\  \  \mbox{ and }\  \overline{h_r}=\overline{E_{r,r}-E_{r+1,r+1}}
\   \   \mbox{ for }1\le r\le N-1.
\end{align*}
Then we have:

\begin{lemt}\label{lem:slzre}
The following relations hold
for $1\le i,j\le N$, $m_0,m_1,n_1\in \Z$ with $(i,m_1)\ne (j,n_1)$:
\begin{align}\label{slzre1}
\overline{E_{Nm_1+i,Nn_1+j}\ot t^{m_0}}=q^{-m_0n_1}e_{i,j}(m_0,m_1-n_1),
\end{align}
\begin{equation}\begin{split}\label{slzre2}
&\overline{(E_{Nm_1+i,Nm_1+i}-E_{Nn_1+j,Nn_1+j})\ot t^{m_0}}\\
=\,&\begin{cases}
q^{-m_0m_1}e_{i,i}(m_0,0)-q^{-m_0n_1}e_{j,j}(m_0,0)\ &\te{if}\ m_0\ne 0\\
\overline{E_{i,i}-E_{j,j}}+(m_1-n_1)\bm\rk'\ &\te{if}\ m_0=0.
\end{cases}
\end{split}
\end{equation}
Furthermore, the vectors
\begin{align}\label{basis2}
 e_{i,j}(m_0,m_1),\quad \overline{h_r},\quad \bm\rk', \quad  \bm\rk,
\end{align}
where
$1\le i,j\le N$, $m_0,m_1\in \Z$ with $(i-j,m_0,m_1)\ne (0,0,0)$ and $1\le r\le N-1$,
form a basis of $\wh{\fsl_\infty}[\Z]$.
\end{lemt}

\begin{proof} From definition (see  \eqref{relation2}) we get  \eqref{slzre1} as
\begin{equation*}\begin{split}
&\overline{E_{N m_1+i,Nn_1+j}\ot t^{m_0}}
=\chi_q(-n_1)^{m_0}\overline{\sigma^{N(-n_1)}(E_{N m_1+i,Nn_1+j})\ot t^{m_0}}\\
=&\ q^{-m_0n_1}\overline{E_{N (m_1-n_1)+i,j}\ot t^{m_0}}=q^{-m_0n_1}e_{i,j}(m_0,m_1-n_1).
\end{split}\end{equation*}
Assume $m_0\ne 0$.  For convenience, define $e_{r,r}(m_0,0)$ for any $r\in \Z$ in the same way:
$$e_{r,r}(m_0,0)=\frac{1}{1-q^{-m_0}}\overline{(E_{r,r}-E_{N+r,N+r})\ot t^{m_0}}.$$
Note that
\begin{align*}
e_{r+kN,r+kN}(m_0,0)=q^{-km_0}e_{r,r}(m_0,0)
\end{align*}
for any $r,k\in \Z$.
Using this we get
\begin{align*}
&(1-q^{-m_0})\overline{(E_{m,m}-E_{n,n})\ot t^{m_0}}\\
=\ &\(\overline{(E_{m,m}-E_{n,n})\ot t^{m_0}}-\overline{(E_{N+m,N+m}-E_{N+n,N+n})\ot t^{m_0}}\)\\
=\ &\overline{(E_{m,m}-E_{N+m,N+m})\ot t^{m_0}-(E_{n,n}-E_{N+n,N+n})\ot t^{m_0}}\\
=\,&(1-q^{-m_0})\left(e_{m,m}(m_0,0)-e_{n,n}(m_0,0)\right),
\end{align*}
which gives
\begin{align*}
\overline{(E_{m,m}-E_{n,n})\ot t^{m_0}}&=e_{m,m}(m_0,0)-e_{n,n}(m_0,0)
\end{align*}
for any $m,n\in \Z$. Then we have
\begin{align*}
&\overline{(E_{Nm_1+i,Nm_1+i}-E_{Nn_1+j,Nn_1+j})\ot t^{m_0}}\\
=\ &e_{Nm_1+i,Nm_1+i}(m_0,0)-e_{Nn_1+j,Nn_1+j}(m_0,0)\\
=\ &q^{-m_1m_0}e_{i,i}(m_0,0)-q^{-n_1m_0}e_{j,j}(m_0,0).
\end{align*}

Now, assume $m_0=0$. For any integer $r$, as
\begin{align*}
\overline{E_{N+1,N+1}-E_{1,1}}-\overline{E_{N+r,N+r}-E_{r,r}}=\overline{\sigma_{N,1}(E_{1,1}-E_{r,r})-(E_{1,1}-E_{r,r})}=0,
\end{align*}
we have
\begin{align*}
\overline{E_{N+r,N+r}-E_{r,r}}=\overline{E_{N+1,N+1}-E_{1,1}}=\bm\rk'.
\end{align*}
It follows that
\begin{align*}
\overline{E_{nN+i,nN+i}-E_{i,i}}=n\,\bm\rk'\quad \te{for all}\  1\le i\le N,\ n\in \Z.
\end{align*}
Then we have
\begin{align*} &\overline{E_{Nm_1+i,Nm_1+i}-E_{Nn_1+j,Nn_1+j}}\\
=\ &\overline{E_{Nm_1+i,Nm_1+i}-E_{i,i}} +\overline{E_{i,i}-E_{j,j}}
+\overline{E_{j,j}-E_{Nn_1+j,Nn_1+j}}\\
=\ &\overline{E_{i,i}-E_{j,j}}+(m_1-n_1)\bm\rk'.
\end{align*}
This proves the first assertion.

Now, we prove the basis assertion. Note that matrices $E_{m,n}$ and $E_{m,m}-E_{m+1,m+1}$
for $m,n\in \Z$ with $m\ne n$ form a basis of $\fsl_\infty$. Then the elements
$$E_{m,n}\otimes t^{m_0}, \  \  \  \ (E_{m,m}-E_{m+1,m+1})\otimes t^{m_0}$$
for $m,n,m_0\in \Z$ with $m\ne n$ together with ${\bf k}$ form a basis of $\wh{\fsl_\infty}$.
Recall that for $r\in \Z$, $\rho_{N}(r)=(\sigma^{N})^{r}$ on $\fsl_\infty$ is the restriction of the automorphism of $\fgl_\infty$ defined by
$\sigma^{Nr}(E_{m,n})=E_{m+rN,n+rN}$ for $m,n\in \Z$ and that
$$\rho_{N}(r)(E_{m,n}\otimes t^{m_0})=\chi_q(r)^{m_0}E_{m+rN,n+rN}\otimes t^{m_0}
=q^{rm_0}E_{m+rN,n+rN}\otimes t^{m_0} $$
 for $m,n, m_0\in \Z$.
 It then follows from Lemma \ref{linear-algebra-fact} that the elements
 $$\overline{E_{mN+i,j}\otimes t^{m_0}}, \  \  \  \ \overline{(E_{i,i}-E_{i+1,i+1})\otimes t^{m_0}}$$
for $1\le i,j\le N,\  m,m_0\in \Z$ with $(i,m)\ne (j,0)$ together with ${\bf k}$ form a basis of $\wh\fsl_\infty[\Z]$.
When $m_0=0$, it is clear that $\{ \overline{h_1},\dots,\overline{h_{N-1}}, {\bf k}'\}$
is also a basis for the subspace spanned by $\overline{(E_{i,i}-E_{i+1,i+1})\otimes t^{m_0}}$ for $1\le i\le N$.

Let $m_0$ be any nonzero integer. From the first assertion,  the $N$ vectors
\begin{align}\label{i,N+i}
 \overline{(E_{i,i}-E_{N+i,N+i})\otimes t^{m_0}}\   \   \   \mbox{ for }1\le i\le N
 \end{align}
linearly span the $N$-dimensional subspace spanned by
\begin{align}\label{i,1+i}
\overline{(E_{i,i}-E_{i+1,i+1})\otimes t^{m_0}}\   \   \   \mbox{ for }1\le i\le N.
\end{align}
Consequently, the $N$ vectors in (\ref{i,N+i}) also form a basis for the subspace spanned by the vectors in
(\ref{i,1+i}). Therefore,
$$\overline{E_{mN+i,j}\otimes t^{m_0}}, \  \  \  \ \overline{(E_{i,i}-E_{N+i,N+i})\otimes t^{m_0}}$$
for $1\le i,j\le N,\  m,m_0\in \Z$ with $(i,m)\ne (j,0)$ together with ${\bf k}'$ and ${\bf k}$
form a basis of $\wh\fsl_\infty[\Z]$.
Now, the proof is complete.
\end{proof}

As the main result of this section, we have:

\begin{thmt}\label{prop:thetanq}
The Lie algebra  $\wh{\fsl_N}(\C_q)$ is isomorphic to the covariant algebra $\wh{\fsl_\infty}[\Z]$,
where an isomorphism $\theta:\wh{\fsl_N}(\C_q)\rightarrow \wh{\fsl_\infty}[\Z]$ is given by
\begin{align*}
&\theta( \bm\rk_0)=\bm\rk,\quad \theta(\bm\rk_1)=\bm\rk',\\
&\theta(E_{i,j}t_0^{m_0}t_1^{m_1})=e_{i,j}(m_0,m_1),\quad \theta(E_{r,r}-E_{r+1,r+1})=\overline{E_{r,r}-E_{r+1,r+1}}
\end{align*}
for $1\le i,j\le N$, $m_0,m_1\in \Z$ with $(i-j,m_0,m_1)\ne (0,0,0)$ and $1\le r\le N-1$.
\end{thmt}

\begin{proof} In view of Lemma \ref{lem:slzre},  the linear map $\theta$ is a linear isomorphism from
$\wh{\fsl_N}(\C_q)$ onto $\wh{\fsl_\infty}[\Z]$, so it remains to prove that $\theta$ is a Lie algebra homomorphism.
Recall from Proposition \ref{lem:thetanq} that we have a Lie algebra homomorphism
$\theta_{N,q}:\wh{\fgl_N}(\C_q)\rightarrow \wh{\fgl_\infty}[\Z]$.
This gives a Lie algebra homomorphism
from $\wh{\fsl_N}(\C_q)$ to $\wh{\fgl_\infty}[\Z]$, also denoted by $\theta_{N,q}$, such that $\ker (\theta_{N,q})=\C \bm\rk_1$.

On the other hand, by Lemma \ref{linear-algebra-fact} the Lie algebra embedding of $\fsl_\infty$ into $\fgl_\infty$
gives rise to a Lie algebra homomorphism $\psi: \wh{\fsl_\infty}[\Z]\rightarrow \wh{\fgl_\infty}[\Z]$.
Notice that the following relations hold in $\wh{\fgl_\infty}[\Z]$:
\begin{align*}
&\overline{(E_{i,i}-E_{N+i,N+i})\otimes t^{m_0}}=(1-q^{-m_0})\overline{(E_{i,i}\otimes t^{m_0})},\\
&\overline{E_{N+1,N+1}-E_{1,1}}=0
\end{align*}
as $\overline{E_{N+j,N+j}\otimes t^{m_0}}=q^{-m_0}\overline{E_{j,j}\otimes t^{m_0}}$ for $1\le j\le N$.
Thus
$$\psi (e_{i,i}(m_0,0))=\overline{E_{i,i}\otimes t^{m_0}}\   \   \   \mbox{ for }1\le i\le N,\ m_0\ne 0$$
and $\psi( \bm\rk')=0$. It follows that $\ker (\psi)=\C \bm\rk'$ and
$\theta_{N,q}(\wh{\fsl_N}(\C_q))=\psi (\wh{\fsl_\infty}[\Z])$.
Consequently, $\theta$ reduces to a Lie algebra isomorphism from $\wh{\fsl_N}(\C_q)/\C \bm\rk_1$
onto $\wh{\fsl_\infty}[\Z]/\C \bm\rk'$.
To show that $\theta$ is a Lie algebra homomorphism, it suffices to check the central extensions involving
$\bm\rk_1$ and $\bm\rk'$, respectively.
From the Lie commutator relation (\ref{commutator1}), we only need to consider
$[E_{i,j}t_0^{m_0}t_1^{m_1},E_{k,l}t_0^{n_0}t_1^{n_1}]$ in $\wh{\fsl_N}(\C_q)$
with $\delta_{j,k}\delta_{i,l}\delta_{m_0+n_0,0}\delta_{m_1+n_1,0}=1$ and $m_1\ne 0$.

Let $1\le i,j\le N,\  m_0,m_1\in \Z$ with $m_1\ne 0$.
From (\ref{commutator1}) we have
\begin{align*}
[E_{i,j}t_0^{m_0}t_1^{m_1},E_{j,i}t_0^{-m_0}t_1^{-m_1}]
=&\ q^{-m_1m_0}(E_{i,i}-E_{j,j}+m_0\bm\rk_0+m_1\bm\rk_1).
\end{align*}
On the other hand, by definition (see \eqref{commutator2}) we have
\begin{align*}
&[\overline{E_{N m_1+i,j}\ot t^{m_0}},\overline{E_{-N m_1+j,i}\ot t^{-m_0}}]\\
=&\, \sum_{r\in \Z}\chi_q(r)^{m_0}\Big(\overline{[E_{N (m_1+r)+i,Nr+j}\ot t^{m_0},E_{-N m_1+j,i}\ot t^{-m_0}]}\\
&\quad +m_0\delta_{m_0+n_0,0}\<E_{N (m_1+r)+i,Nr+j},E_{-N m_1+j,i}\>\bm\rk\Big)\\
=&\ \overline{q^{-m_1m_0}E_{i,i}
-q^{-m_1m_0}E_{-N m_1+j,-Nm_1+j}}+m_0q^{-m_0m_1}\bm\rk\\
=&\ q^{-m_1m_0}\(\overline{E_{i,i}-E_{j,j}}+m_1\bm\rk'\)+m_0q^{-m_0m_1}\bm\rk,
\end{align*}
where we are using  \eqref{slzre2} for the last equality. We see that the central extension by $\bm\rk_1$
matches the central extension by $\bm\rk'$. Therefore, $\theta$ is a Lie algebra isomorphism
from $\wh{\fsl_N}(\C_q)$ to $\wh{\fsl_\infty}[\Z]$.
\end{proof}




\section{Equivariant quasi modules for simple vertex algebra $L_{\wh{\fsl_\infty}}(\ell,0)$}
In this section, first we recall from \cite{li-gamma} the notion of
vertex $\Gamma$-algebra and the notion of equivariant quasi module for a vertex $\Gamma$-algebra.
Then, as the main result of this section, we  prove that for any nonnegative integer $\ell$,
the category of equivariant quasi modules for the simple vertex algebra
 $L_{\wh{\fsl_\infty}}(\ell,0)$ associated to the affine Lie algebra $\wh{\fsl_\infty}$ is canonically isomorphic to that of
 integrable restricted $\wh{\fsl_N}(\C_q)$-modules of level $\ell$.

\subsection{Vertex $\Gamma$-algebras  and their equivariant quasi modules}

We start with the notion of vertex $\Gamma$-algebra  introduced in \cite{li-gamma}.

\begin{dfnt}
{\em Let $\Gamma$ be a group. A {\em vertex
$\Gamma$-algebra} is a vertex algebra $V$ equipped with two group
homomorphisms
\begin{eqnarray*}
&&R: \Gamma \rightarrow {\rm GL}(V);\ \  g\mapsto R_{g},\\
&&\chi: \Gamma \rightarrow \C^{\times},
\end{eqnarray*}
satisfying the condition that $R_{g}({\bf 1})={\bf 1}$ and
\begin{eqnarray}\label{egenerating-vgamma-algebra}
R_{g}Y(v,x)R_{g}^{-1}=Y(R_{g}(v),\chi(g)^{-1}x) \ \ \mbox{ for }g\in
\Gamma,\ v\in V.
\end{eqnarray}}
 \end{dfnt}

\begin{remt}\label{rem:defvga}{\rm
Let $V$ be a $\Z$-graded vertex algebra in the sense that $V$ is a vertex algebra equipped with a $\Z$-grading
 $V=\oplus_{n\in \Z}V_{(n)}$ such that ${\bf 1}\in V_{(0)}$ and
\begin{eqnarray}
u_{m}V_{(n)}\subset V_{(n+k-m-1)}\ \ \mbox{ for }u\in V_{(k)},\ m,n,
k\in \Z.
\end{eqnarray}
Define a linear operator $L(0)$ on $V$ by $L(0)v=nv$ for $v\in
V_{(n)}$ with $n\in \Z$. Assume that $\Gamma$ is an automorphism group of $V$ as a
$\Z$-graded vertex algebra and let $\chi:
\Gamma\rightarrow \C^{\times}$ be any linear character. For $g\in
\Gamma$, set $R_{g}=\chi(g)^{-L(0)}g$. Then $V$  with group homomorphisms $R$ and $\chi$ is a vertex
$\Gamma$-algebra (see \cite{li-gamma}).}
\end{remt}

\begin{ext}\label{firstex}
{\rm Here we recall the vertex algebras associated to affine Lie algebras.
Let $\fg$ be a Lie algebra equipped with a nondegenerate symmetric invariant bilinear form $\<\cdot,\cdot\>$.
Equip the affine Lie algebra $\wh\fg$ with a
$\Z$-grading $\wh\fg=\oplus_{n\in \Z}\wh\fg_{(n)}$, where
\begin{align}
\wh\fg_{(0)}=\fg\oplus \C\bm\rk\quad \te{and}\quad \wh\fg_{(n)}=\fg\ot \C t^{-n}\quad \te{for}\ n\ne 0.
\end{align}
Let $\ell$ be a complex number.
 View $\C$ as a $\(\fg\otimes \C[t]\oplus \C\bm\rk\)$-module with $\fg\otimes \C[t]$
acting trivially and with $\bm\rk$ acting as scalar $\ell$. Form the induced $\wh\fg$-module
\begin{eqnarray}\label{vwhfsl}
V_{\wh\fg}(\ell,0)=\U(\wh\fg)\otimes_{\U(\fg\ot \C[t]\oplus \C\bm\rk)} \C,
\end{eqnarray}
which is naturally $\N$-graded by defining $\deg \C=0$.
 Set ${\bf 1}=1\otimes 1\in
V_{\wh\fg}(\ell,0)$ and identify $\fg$ as the degree-one subspace of
$V_{\wh\fg}(\ell,0)$ through the linear map
\[a\mapsto
a(-1){\bf 1}\in V_{\wh\fg}(\ell,0)\quad\te{for}\ a\in \fg.\]
Following the standard practice, we alternatively write $a(m)$ for  $a\ot t^m\in \wh\fg$.
It was known (cf. \cite{FZ}) that there
exists a vertex algebra structure on $V_{\wh\fg}(\ell,0)$,
which is uniquely determined by the condition that
${\bf 1}$ is the vacuum vector and
\[Y(a,x)=a(x):=\sum_{m\in \Z}a(m) x^{-m-1}\quad \te{for}\ a\in \fg.\]
Denote by $J_{\wh\fg}(\ell,0)$ the unique maximal graded $\wh\fg$-submodule of $V_{\wh\fg}(\ell,0)$.
Then $J_{\wh\fg}(\ell,0)$ is a (two-sided) ideal of vertex algebra $V_{\wh\fg}(\ell,0)$.
Define
\begin{align}
L_{\wh\fg}(\ell,0)=V_{\wh\fg}(\ell,0)/J_{\wh\fg}(\ell,0),\end{align}
which is a simple $\Z$-graded vertex algebra. One can show
 that the vertex algebra $L_{\wh\fg}(\ell,0)$ is simple.
Assume that $\Gamma$ is a group which acts on $\fg$ as an automorphism group preserving the bilinear form.
It is a simple fact that the action of $\Gamma$ on $\fg$ can be uniquely lifted to an action
on  $V_{\wh\fg}(\ell,0)$ where $\Gamma$ acts as an automorphism group
preserving the $\Z$-grading.
As $J_{\wh\fg}(\ell,0)$ is $\Gamma$-stable, $\Gamma$ naturally acts on $L_{\wh\fg}(\ell,0)$.
From Remark \ref{rem:defvga}, for any linear character  $\chi$  of $\Gamma$,
 $V_{\wh\fg}(\ell,0)$ and $L_{\wh\fg}(\ell,0)$ are vertex $\Gamma$-algebras with
$R_g=\chi(g)^{-L(0)}g$ for $g\in \Gamma$. }
\end{ext}

The following notion was introduced in \cite{li-tlie}:

\begin{dfnt}
{\em Let $V$ be a vertex $\Gamma$-algebra. An
{\em equivariant quasi $V$-module} is a vector space $W$ equipped with a linear map
\begin{eqnarray*}
Y_{W}(\cdot,x):&& V\rightarrow \mathrm{Hom}(W,W((x)))\subset (\mathrm{End}
W)[[x,x^{-1}]]\\
&&v\mapsto Y_{W}(v,x),
\end{eqnarray*}
satisfying the conditions that
\begin{eqnarray}Y_{W}({\bf 1},x)&=&1_{W}\quad\te{(the identity operator on $W$)},\\
Y_{W}(R_{g}v,x)&=&Y_{W}(v,\chi(g)x)\ \ \ \mbox{ for }g\in \Gamma,\
v\in V
\end{eqnarray}
and  that for $u,v\in V$, the {\em quasi Jacobi identity}
\begin{eqnarray}\label{epjacobi}
&&x_{0}^{-1}\delta\left(\frac{x_{1}-x_{2}}{x_{0}}\right)p(x_{1},x_{2})Y_{W}(u,x_{1})Y_{W}(v,x_{2})
\nonumber\\
&&\hspace{1cm}
-x_{0}^{-1}\delta\left(\frac{x_{2}-x_{1}}{-x_{0}}\right)p(x_{1},x_{2})Y_{W}(v,x_{2})Y_{W}(u,x_{1})\nonumber\\
&=&x_{2}^{-1}\delta\left(\frac{x_{1}-x_{0}}{x_{2}}\right)p(x_{1},x_{2})Y_{W}(Y(u,x_{0})v,x_{2})
\end{eqnarray}
holds on $W$ for some polynomial $p(x_{1},x_{2})$ of the form
\begin{eqnarray}
p(x_{1},x_{2})=(x_{1}-\chi(g_{1})x_{2})\cdots
(x_{1}-\chi(g_{k})x_{2}),
\end{eqnarray}
where $g_{1},\dots,g_{k}\in \Gamma$.}
\end{dfnt}

To emphasize the dependence on the group $\Gamma$ and the linear character $\chi$,
we shall also use the term ``$(\Gamma,\chi)$-equivariant quasi $V$-module."
Assume that $V=\oplus_{n\in \Z}V_{(n)}$ is also $\Z$-graded. A {\em $\Z$-graded  equivariant quasi
 $V$-module} is an equivariant quasi $V$-module $W$ with a $\Z$-grading
$W=\oplus_{n\in \Z} W_{(n)}$ such that
\begin{eqnarray}
u_{m}W_{(n)}\subset W_{(n+k-m-1)}\ \ \mbox{ for }u\in V_{(k)},\ m,n,
k\in \Z.
\end{eqnarray}

Recall the $(\Gamma,\chi)$-covariant algebra $\wh\fg[\Gamma]$ of the affine Lie algebra $\wh\fg$ from Section 2.2.
The following result was obtained in \cite{li-tlie}:

\begin{prpt}\label{exquasimo1}
Let $\fg, \Gamma$ be given as in Example \ref{firstex}, let $\chi: \Gamma\rightarrow \C^{\times}$
be a one-to-one group homomorphism, and let $\ell\in \C$.
Then for any restricted $\wh\fg[\Gamma]$-module $W$ of level $\ell$, there is a
$(\Gamma,\chi)$-equivariant quasi $V_{\wh\fg}(\ell,0)$-module structure $Y_{W}(\cdot,x)$
which is uniquely determined by
\begin{align*}
Y_W(a,x)=\bar{a}(x):=\sum_{n\in \Z}\overline{a(n)} x^{-n-1}\quad \te{for}\ a\in \fg.
\end{align*}
On the other hand, any $(\Gamma,\chi)$-equivariant quasi $V_{\wh\fg}(\ell,0)$-module $W$ is a
 restricted $\wh\fg[\Gamma]$-module of level $\ell$ with $\overline{a}(x)=Y_W(a,x)$ for $a\in \fg$.
\end{prpt}

\begin{remt}\label{tensorex}{\rm
Let $V_1,\dots,V_r$ be vertex $\Gamma$-algebras with the same linear character $\chi$. It is straightforward to see that
the tensor product vertex algebra $V_1\ot \cdots\ot V_r$ (see \cite{FHL}) is also a vertex $\Gamma$-algebra.
On the other hand,
 let $W_i$ be a $(\Gamma,\chi)$-equivariant quasi $V_i$-module for $1\le i\le r$.
 Then $W_1\ot \cdots \ot W_r$ is an  equivariant quasi $V_1\ot \cdots\ot V_r$-module with
\begin{align*}
Y_{W_1\ot \cdots\ot W_r}(v_1\ot \cdots\ot v_r,x)
=Y_{W_1}(v_1,x)\ot \cdots \ot Y_{W_r}(v_r,x)\quad \te{for}\ v_i\in V_i.
\end{align*}}
\end{remt}

\subsection{Equivariant quasi $L_{\wh{\fsl_\infty}}(\ell,0)$-modules}
 Recall from  Section 2.2 that $\Z$ acts on $\fsl_\infty$ via the representation $\rho_N$ (see \eqref{defthetaN})
  and $\chi_q$ is the linear character of $\Z$ defined by $\chi_q(n)=q^n$ for $n\in \Z$.
 From Example \ref{firstex},  for any complex number $\ell$ we have $\Z$-graded vertex $\Z$-algebras
 $V_{\wh{\fsl_\infty}}(\ell,0)$ and $L_{\wh{\fsl_\infty}}(\ell,0)$.
 To emphasize the dependence on the $\Z$-action via $\rho_N$,
we shall also denote the vertex $\Z$-algebra $L_{\wh{\fsl_\infty}}(\ell,0)$ by $(L_{\wh{\fsl_\infty}}(\ell,0),\rho_N)$.
  In this subsection, for any nonnegative integer $\ell$, we give a characterization of $(\Z,\chi_q)$-equivariant quasi $L_{\wh{\fsl_\infty}}(\ell,0)$-modules in terms of integrable restricted $\wh{\fsl_N}(\C_q)$-modules of level $\ell$.

\begin{dfnt}
{\em An $\wh{\fsl_N}(\C_q)$-module $W$ is said to be
{\em restricted} if for any $w\in W$ and for $1\le i,j\le N,\ m_0,m_1\in \Z$ with $(i-j,m_0,m_1)\ne (0,0,0)$,
$$(E_{i,j}t_0^{m_0}t_1^{m_1})w=0  \   \    \   \mbox{ for $m_0$ sufficiently large},$$
 and it is said to be {\em of level $\ell\in \C$} if $\bm\rk_0$ acts as
scalar $\ell$.}
\end{dfnt}

For $1\le i,j\le N,\  m\in \Z$ with $(i-j,m)\ne (0,0)$, we form a generating function
\begin{align}
\(E_{i,j}t_1^m\)(x)=\sum_{n\in \Z}(E_{i,j}t_0^nt_1^m) x^{-n-1}\in \wh{\fsl_N}(\C_q)[[x,x^{-1}]],
\end{align}
and we form generating functions
\begin{align}
H_i(x)=\sum_{n\in \Z}(E_{i,i}-E_{i+1,i+1})t_0^n x^{-n-1}\   \   \   \mbox{ for }1\le i\le N-1,
\end{align}
\begin{align}
H_N(x)=\sum_{n\in \Z}(E_{N,N}t_0^{n}-q^{-n}E_{1,1}t_0^n-\delta_{n,0}{\bf k}_1)x^{-n-1}.
\end{align}
It follows from Lemma 2.1 that all the components of these generating functions together with
 ${\bf k}_0$ form a basis of $\wh{\fsl_N}(\C_q)$.

The following result gives an isomorphism between the category of
  $(\Z,\chi_q)$-equivariant quasi $V_{\wh{\fsl_\infty}}(\ell,0)$-modules
and the category of restricted $\wh{\fsl_N}(\C_q)$-modules of level $\ell$ (cf. \cite{li-tlie}):

\begin{prpt}\label{prop:main1}
Let $\ell$ be any complex number.
For any restricted $\wh{\fsl_N}(\C_q)$-module $W$ of level $\ell$, there exists a $(\Z,\chi_q)$-equivariant quasi $V_{\wh{\fsl_\infty}}(\ell,0)$-module
structure $Y_W(\cdot,x)$ on $W$, which is uniquely determined by
\begin{align*}
&Y_W(E_{mN +i,nN +j},x)=q^n \(E_{i,j}t_1^{m-n}\)(q^n x),\\
&Y_W(E_{nN+i,nN+i}-E_{nN+i+1,nN+i+1},x)=q^n H_i(q^n x),
\end{align*}
for $1\le i,j\le N, m,n\in \Z$ with $(i,m)\ne (j,n)$.
On the other hand, for any $(\Z,\chi_q)$-equivariant quasi $V_{\wh{\fsl_\infty}}(\ell,0)$-module $(W,Y_W)$,
$W$ is a  restricted $\wh{\fsl_N}(\C_q)$-module of level $\ell$ with the action uniquely determined by
\[ \(E_{i,j}t_1^{m}\)(x)=Y_W(E_{mN +i,j},x),\quad H_i(x)=Y_W(E_{i,i}-E_{i+1,i+1},x) \]
for $1\le i,j\le N$, $m\in \Z$ with $(i-j,m)\ne (0,0)$.
\end{prpt}

\begin{proof}
Let $W$ be a restricted $\wh{\fsl_N}(\C_q)$-module  of level $\ell$.
It follows from  Proposition \ref{prop:thetanq} that
$W$ is a restricted  $\wh{\fsl_\infty}[\Z]$-module of level $\ell$ on which
\begin{align*}
&\overline{E_{mN+i,j}\ot t^n}=E_{i,j}t_0^nt_1^{m},\
\overline{(E_{k,k}-E_{k+1,k+1})\ot t^n}=(E_{k,k}-E_{k+1,k+1})t_0^n,\\
&\overline{(E_{N,N}-E_{N+1,N+1})\ot t^n}=E_{N,N}t_0^{n}-q^{-n}E_{1,1}t_0^n-\delta_{n,0}{\bf k}_1,
\end{align*}
where $1\le i,j\le N, 1\le k\le N-1$ and $m,n\in \Z$ with $(i-j,m)\ne (0,0)$.
As $\chi_q$ is injective,
from Proposition \ref{exquasimo1}, the $\wh{\fsl_\infty}[\Z]$-module $W$  is naturally a
$(\Z,\chi_q)$-equivariant quasi $V_{\wh{\fsl_\infty}}(\ell,0)$-module with
\begin{align}\label{thm1:2}
Y_W(a,x)=\bar{a}(x)=\sum_{n\in \Z}\overline{a(n)}x^{-n-1}\quad \te{for}\ a\in \fsl_\infty.
\end{align}
Then 
 $(\Z,\chi_q)$-equivariant quasi $V_{\wh\fsl_\infty}(\ell,0)$-module
 structure $Y_W(\cdot,x)$ on $W$ is uniquely determined by
\begin{align*}
&Y_W(E_{mN+i,nN+j},x)=\overline{E_{mN+i,nN+j}}(x)
=q^n\overline{E_{(m-n)N+i,j}}(q^n x)=q^n\(E_{i,j}t_1^{m-n}\)(q^nx),\\
&Y_W(E_{nN+i,nN+i}-E_{nN+i+1,nN+i+1},x)=
q^n \overline{E_{i,i}-E_{i+1,i+1}}(q^n x)
=q^n H_i(q^n x),
\end{align*}
for $1\le i\ne j \le N,\  m,n\in \Z$ with $(i,m)\ne (j,n)$.

On the other hand, let $(W,Y_W)$ be a $(\Z,\chi_q)$-equivariant quasi $V_{\wh{\fsl_\infty}}(\ell,0)$-module.
From Proposition \ref{exquasimo1} also, $W$ is a restricted $\wh{\fsl_\infty}[\Z]$-module of level $\ell$ with
$\bar{a}(x)=Y_W(a,x)$ for $a\in \fsl_\infty.$
By Proposition \ref{prop:thetanq}, $W$ becomes a restricted $\wh{\fsl_N}(\C_q)$-module of level $\ell$  on which
\begin{align*}
 &\(E_{i,j}t_1^{m}\)(x)=\overline{E_{mN+i,j}}(x)=Y_W(E_{mN +i,j},x),\\
& H_i(x)=\overline{E_{i,i}-E_{i+1,i+1}}(x)=Y_W(E_{i,i}-E_{i+1,i+1},x)
\end{align*}
for $1\le i,j\le N, m\in \Z$ with $(i-j,m)\ne (0,0)$.
\end{proof}

We fix a $\Z$-grading on $\wh{\fgl_N}(\C_q)$ given by derivation $-{\bf d}_{0}$ (recall (\ref{def-derivations-d})), i.e.,
\begin{align}
\deg (E_{i,j}t_0^{n}t_1^m)=-n\   \   \   \mbox{ for }1\le i,j\le N,\ m,n\in \Z.
\end{align}
Then $\wh{\fsl_N}(\C_q)=\oplus_{n\in \Z}\wh{\fsl_N}(\C_q)_{(n)}$, where
for nonzero $n\in \Z$,
\begin{align}\label{gradingong}
\wh{\fsl_N}(\C_q)_{(n)}=\sum_{1\le i,j\le N}\sum_{m\in \Z}\C E_{i,j}t_0^{-n}t_1^{m}
\end{align}
and
\begin{align}\label{grading-0}
&\wh{\fsl_N}(\C_q)_{(0)}\\
=& \sum_{1\le i,j\le N,m\in \Z, (i-j,m)\ne (0,0)}\C E_{i,j}t_1^{m}
+\sum_{r=1}^{N-1}\C (E_{r,r}-E_{r+1,r+1})+\C {\bf k}_0+\C {\bf k}_1.\nonumber
\end{align}


\begin{dfnt}\label{C-graded-module}
{\em A {\em $\C$-graded $\wh{\fsl_N}(\C_q)$-module} is an $\wh{\fsl_N}(\C_q)$-module $W$
equipped with a $\C$-grading
$W=\oplus_{\alpha\in \C} W(\alpha)$ such that $\wh{\fsl_N}(\C_q)_{(n)}W(\alpha)\subset W(n+\alpha)$
for $n\in \Z,\ \alpha\in \C$.}
\end{dfnt}

Note that that a $\C$-graded $\wh{\fsl_N}(\C_q)$-module is the same as an
$\wh{\fsl_N}(\C_q)\rtimes \C \bm\rd_0$-module on which $\bm\rd_0$ acts semisimply.
The notions of homomorphism and isomorphism for $\C$-graded $\wh{\fsl_N}(\C_q)$-modules
are defined in the obvious way (which are required to preserve the $\C$-gradings).

\begin{remt}\label{rem:gradedver}
{\rm  It follows from Proposition \ref{prop:main1} that  for any complex number $\ell$,
 a $\Z$-graded $(\Z,\chi_q)$-equivariant quasi $V_{\wh{\fsl_\infty}}(\ell,0)$-module structure on a vector space $W$
  amounts to a $\Z$-graded restricted $\wh{\fsl_N}(\C_q)$-module structure of level $\ell$.}
\end{remt}

Recall that $J_{\wh{\fsl_\infty}}(\ell,0)$ is the maximal $\wh{\fsl_\infty}$-submodule of $V_{\wh{\fsl_\infty}}(\ell,0)$.
The following is a description of $J_{\wh{\fsl_\infty}}(\ell,0)$ for nonnegative integers $\ell$.

\begin{lemt}\label{lem:charintslmod}
Let $\ell$ be a nonnegative integer. Then
\begin{align}
E_{i,j}(-1)^{\ell+1}\bm{1}\in J_{\wh{\fsl_\infty}}(\ell,0) \   \   \mbox{ for any }i\ne j\in \Z.
\end{align}
Furthermore,
$J_{\wh{\fsl_\infty}}(\ell,0)$ as an $\wh{\fsl_\infty}$-module is generated by the vectors
\begin{align}\label{genofj}
(E_{mN+i,nN+j}(-1))^{\ell+1}\bm{1}\quad\te{for}\ 1\le i\ne j\le N,\ m,n\in \Z.
\end{align}
\end{lemt}

\begin{proof} Let $i,j\in \Z$ with $i\ne j$. We first show that
\begin{align}\label{first-mnk}
E_{m,n}(k)E_{i,j}(-1)^{\ell+1}\bm{1}=0,\   \   \   \   (E_{m,m}-E_{n,n})(k)E_{i,j}(-1)^{\ell+1}\bm{1}=0
\end{align}
for all $m,n,k\in \Z$ with $m\ne n,\ k>0$. Let $m,n\in \Z$ with $m\ne n$.
Pick a finite interval $I$ of $\Z$, containing $m,n,i,j$.
Set
\[\fsl_I=\te{Span}\{E_{i,j}, E_{i,i}-E_{j,j}\mid i,j\in I \  \mbox{with }i\ne j\}.\]
It follows from the P-B-W theorem that  the $\wh{\fsl_I}$-submodule $\U(\wh{\fsl_I})\bm 1$ of $V_{\wh{\fsl_\infty}}(\ell,0)$
is isomorphic to $V_{\wh{\fsl_I}}(\ell,0)$.
On the other hand, it was known (cf. \cite{LL}) that
\begin{align*}
\(\fsl_I\ot t\C[t]\)E_{i,j}(-1)^{\ell+1}\bm{1}=0\  \  \te{ in }V_{\wh{\fsl_I}}(\ell,0)
\end{align*}
for all $i, j\in I$ with $i\ne j$. Thus \eqref{first-mnk} holds. Then it follows from the P-B-W theorem that
 $\U(\wh{\fsl_\infty})E_{i,j}(-1)^{\ell+1}\bm{1}$ is a proper graded submodule of
$V_{\wh{\fsl_\infty}}(\ell,0)$, and hence a submodule of $J_{\wh{\fsl_\infty}}(\ell,0)$. This proves the first assertion.

For the second assertion, denote by $J'$ the $\wh{\fsl_\infty}$-submodule of $V_{\wh{\fsl_\infty}}(\ell,0)$
generated by the vectors in \eqref{genofj}.
Set $V'=V_{\wh{\fsl_\infty}}(\ell,0)/J'$ and $\bm 1'=1+J'\in V'$. Then it suffices to show that
$V'$ is an irreducible $\wh{\fsl_\infty}$-module.
Let $v$ be any nonzero vector in $V'$. 
Fix a finite interval $I=[Nm+1,Nn]$ of $\Z$ with $m<n$
  such that $v\in \U(\wh{\fsl_I})\bm 1'$.
Note that $E_{mN+1,nN}$ is a highest root vector in $\fsl_I$ and that (see \eqref{genofj})
\begin{align}\label{nmnn}
E_{mN+1,nN}(-1)^{\ell+1}\bm 1'=0.
\end{align}
Then the highest weight
$\wh{\fsl_I}$-module $\U(\wh{\fsl_I})\bm 1'$ is integrable and hence irreducible (see \cite{Kac}),
which implies $\U(\wh\fsl_I)\bm{1'} =\U(\wh\fsl_I)v$.
Thus
\[\bm{1'}\in \U(\wh\fsl_I)\bm{1'} =\U(\wh\fsl_I)v\subset \U(\wh\fsl_\infty)v.\]
As $V'=\U(\wh\fsl_\infty)\bm{1'}$, we get
$V'=\U(\wh\fsl_\infty)v$.
This proves that $V'$ is an irreducible $\wh{\fsl_\infty}$-module, and hence $J'=J_{\wh{\fsl_\infty}}(\ell,0)$.
Now, the proof is complete.
\end{proof}

The following notion (cf. \cite{ER}) is similar to the notion of integrable module for affine Kac-Moody algebras:

\begin{dfnt}
{\em An $\wh{\fsl_N}(\C_q)$-module $W$  is said to be {\em integrable} if for any $1\le i\ne j\le N,$
$m_0,m_1\in \Z$, $E_{i,j}t_0^{m_0}t_1^{m_1}$ acts locally nilpotently on $W$.}
\end{dfnt}

Notice that for any $1\le i\ne j\le N,\ m,n\in\Z$, from (\ref{commutator1}) we have
\begin{align}\label{special-commuting}
[\(E_{i,j}t_1^m\)(x_1),\(E_{i,j}t_1^n\)(x_2)]=0\  \   \te{in } \wh{\fgl_N}(\C_q)[[x_1^{\pm 1},x_2^{\pm 1}]]
\end{align}
and hence this is true in $\wh{\fsl_N}(\C_q)[[x_1^{\pm 1},x_2^{\pm 1}]]$.
The following is analogous to a result  for affine Kac-Moody algebras:

\begin{prpt}\label{prop:charintgmod}
Let $W$ be a restricted $\wh{\fsl_N}(\C_q)$-module of level $\ell\in \C$.
Then $W$ is integrable if and only if $\ell$ is a nonnegative integer and
\[\(E_{i,j}t_1^m\)(x)^{\ell+1}=0\  \  \te{on}\  W\]
for all $1\le i\ne j\le N$ and $m\in\Z$.
\end{prpt}

\begin{proof} Let $\{e, h, f\}$ be the standard basis of $\fsl_2$ such that
\begin{align}\label{standbasis}
[e,f]=h,\quad [h,e]=2e,\quad  [h,f]=-2f.
\end{align}
It was known (see \cite{LP}, \cite{DLM}) that a restricted module $U$ of level $\ell$
for the affine Lie algebra $\wh{\fsl_2}$  is integrable if and only if
 $\ell$ is a nonnegative integer and
\begin{align*}
e(x)^{\ell+1}=0=f(x)^{\ell+1}\   \  \te{on}\ U,
\end{align*}
where  $a(x)=\sum_{n\in \Z} \(a\ot t^n\)x^{-n-1}$ for $a\in \fsl_2$.
For $1\le i <j\le N,\  m\in \Z$, set
\begin{align*}
\wh{\mathcal A}_{i,j}(m)=\te{Span}\{ E_{i,j}t_0^nt_1^m, E_{j,i}t_0^nt_1^{-m}, q^{mn}E_{i,i}t_0^n-E_{j,j}t_0^n
+m\delta_{n,0}\bm\rk_1, \bm\rk_0\mid
n\in \Z\}.
\end{align*}
It is straightforward to show that $\wh{\mathcal A}_{i,j}(m)$ is a subalgebra of $\wh{\fsl_N}(\C_q)$
 isomorphic to $\wh{\fsl_2}$, where an isomorphism is given by
\begin{equation*}\begin{split}
&q^{mn}E_{i,j}t_0^nt_1^m\mapsto e\ot t^n,\quad
E_{j,i}t_0^nt_1^{-m}\mapsto f\ot t^n,\\
&q^{mn}E_{i,i}t_0^n-E_{j,j}t_0^n
+m\delta_{n,0}\bm\rk_1\mapsto h\ot t^n,\quad \bm\rk_0\mapsto \bm\rk
\end{split}\end{equation*}
for $n\in \Z$.
Then $W$ is an integrable $\wh{\mathcal A}_{i,j}(m)$-module if and only if $\ell$ is a nonnegative integer and
\begin{align*}
\(E_{i,j}t_1^m\)(x)^{\ell+1}=\(E_{j,i}t_1^{-m}\)(x)^{\ell+1}=0\  \  \te{on}\ W.
\end{align*}
 On the other hand, from definition $W$ is an integrable $\wh{\fsl_N}(\C_q)$-module if and only if for
 any $1\le i <j\le N,\  m\in \Z$, $W$ is an integrable $\wh{\mathcal A}_{i,j}(m)$-module.
 Then it follows immediately.
\end{proof}

The following is a special case of \cite[Corollary 5.3]{ltw-tri}:

\begin{lemt}\label{lem:special}
Let $V$ be a vertex $\Gamma$-algebra, let $a\in V$
such that $a_{n}a=0$ for $n\ge 0$, and let $\ell$ be a nonnegative
integer. Suppose that $(W,Y_{W})$ is an equivariant quasi $V$-module such that
\[[Y_W(a,x_1),Y_W(a,x_2)]=0.\]
If $(a_{-1})^{\ell+1}{\bf 1}=0$ in $V$, then
\begin{eqnarray}\label{ePnilpotent}
Y_W(a,x)^{\ell+1}=0\  \  \mbox{on }W.
\end{eqnarray}
On the other hand, the converse is also true if $(W,Y_W)$ is faithful.
\end{lemt}

As $L_{\wh{\fsl_\infty}}(\ell,0)$ is a quotient algebra of $V_{\wh{\fsl_\infty}}(\ell,0)$,
in view of Proposition \ref{prop:main1}, $(\Z,\chi_q)$-equivariant
quasi $L_{\wh{\fsl_\infty}}(\ell,0)$-modules are naturally restricted $\wh{\fsl_N}(\C_q)$-modules of level $\ell$.
The following is the first main result of this paper:

\begin{thm}\label{thm:main2}
Let  $\ell$ be a  nonnegative integer. Then  $(\Z,\chi_q)$-equivariant
quasi $L_{\wh{\fsl_\infty}}(\ell,0)$-modules exactly correspond to integrable restricted
$\wh{\fsl_N}(\C_q)$-modules of level $\ell$.
\end{thm}

\begin{proof}  Let $(W,Y_W)$ be a $(\Z,\chi_q)$-equivariant
quasi $L_{\wh{\fsl_\infty}}(\ell,0)$-module.
In view of Proposition \ref{prop:main1}, $W$ is a restricted
$\wh{\fsl_N}(\C_q)$-module of level $\ell$ with
\begin{align*}
\(E_{i,j}t_1^m\)(x)=Y_W(E_{Nm+i,j},x)\quad\te{for}\ 1\le i\ne j\le N,\, m\in \Z.
\end{align*}
Let $1\le i\ne j\le N,\ m,n\in \Z$.
Noticing that $Nm+i\ne Nn+j$, we have
\begin{align*}
[E_{Nm+i,Nn+j}(x_1),E_{Nm+i,Nn+j}(x_2)]=0
\end{align*}
in $\wh{\fsl_\infty}$, which implies
\begin{align}\label{main21}
\(E_{Nm+i,Nn+j}\)_k \(E_{Nm+i,Nn+j}\)=0\quad\te{for}\ k\ge 0
\end{align}
in $V_{\wh{\fsl_\infty}}(\ell,0)$ and hence in $L_{\wh{\fsl_\infty}}(\ell,0)$.
From Lemma \ref{lem:charintslmod} we also have
\begin{align*}
E_{Nm+i,j}(-1)^{\ell+1}\bm 1=0\   \   \mbox{ in }L_{\wh{\fsl_\infty}}(\ell,0).
\end{align*}
On the other hand, as $i\ne j$, by (\ref{special-commuting}) we have
\begin{align*}
[Y_W(E_{Nm+i,j},x_1),Y_W(E_{Nm+i,j},x_2)]
=[\(E_{i,j}t_1^m\)(x_1),\(E_{i,j}t_1^m\)(x_2)]=0.
\end{align*}
Then by Lemma \ref{lem:special} we get
$$Y_W(E_{Nm+i,j},x)^{\ell+1}=0\  \ \te{on}\ W.$$
Thus
\[\(E_{i,j}t_1^m\)(x)^{\ell+1}=Y_W(E_{Nm+i,j},x)^{\ell+1}=0\  \  \te{on}\ W.\]
In view of Proposition \ref{prop:charintgmod}, $W$ is an integrable $\wh{\fsl_N}(\C_q)$-module.

On the other hand, let $W$ be an integrable restricted $\wh{\fsl_N}(\C_q)$-module of level $\ell$.
By Propositions \ref{prop:main1} and  \ref{prop:charintgmod},
there is  a  $(\Z,\chi_q)$-equivariant
quasi $V_{\wh{\fsl_\infty}}(\ell,0)$-module structure $Y_W(\cdot,x)$ on $W$ such that
$$Y_W(E_{Nm+i,Nn+j},x)=q^{n}\(E_{i,j}t_1^m\)(q^nx)\   \   \te{ for }1\le i\ne j\le N,\ m,n\in \Z.$$
Then we have
\begin{align}\label{main24}
Y_W(E_{Nm+i,Nn+j},x)^{\ell+1}=q^{n(\ell+1)}\(E_{i,j}t_1^m\)(q^nx)^{\ell+1}=0
\end{align}
for $1\le i\ne j\le N$ and $m,n\in \Z$, noticing that
\begin{equation}\begin{split}\label{main25}
&[Y_W(E_{Nm+i,Nn+j},x_1),Y_W(E_{Nm+i,Nn+j},x_1)]\\=\, &
[q^{ n}\(E_{i,j}t_1^m\)(q^nx_1),
q^{ n}\(E_{i,j}t_1^m\)(q^nx_2)]=0.
\end{split}\end{equation}
Note that $W$ is naturally a faithful equivariant quasi module for $V_{\wh{\fsl_\infty}}(\ell,0)/\ker Y_W$.
With \eqref{main21}, \eqref{main25} and \eqref{main24}, it follows from
 Lemma \ref{lem:special} that
\[E_{Nm+i,Nn+j}(-1)^{\ell+1}\bm 1=0\quad \te{in}\ V_{\wh{\fsl_\infty}}(\ell,0)/\ker Y_W\]
for all $1\le i\ne j\le N$ and $m,n\in \Z$.
Then by Lemma \ref{lem:charintslmod} $Y_{W}(\cdot,x)$ reduces to an equivariant quasi module structure
$\overline{Y_{W}}(\cdot,x)$ for $L_{\wh{\fsl_\infty}}(\ell,0)$.
Consequently,  $W$ is an equivariant
quasi $L_{\wh{\fsl_\infty}}(\ell,0)$-module with the required property.
\end{proof}

\section{Classification of  irreducible equivariant quasi $L_{\wh{\fsl_\infty}}(\ell,0)$-modules}
The main goal of this section is to classify irreducible
$\N$-graded $(\Z,\chi_q)$-equivariant quasi $L_{\wh{\fsl_\infty}}(\ell,0)$-modules for any nonnegative integer $\ell$, or equivalently
irreducible $\N$-graded integrable   $\wh{\fsl_N}(\C_q)$-modules of level $\ell$.
Irreducible $\N$-graded $(\Z,\chi_q)$-equivariant
quasi modules for the tensor product vertex $\Z$-algebra $L_{\wh{\fsl_\infty}}(\ell_1,0)$ $\ot \cdots\ot
L_{\wh{\fsl_\infty}}(\ell_d,0)$ are also determined.

\subsection{Integrable highest weight $\wh{\fsl_N}(\C_q)$-modules}
In this subsection, we introduce a notion of highest weight $\wh{\fsl_N}(\C_q)$-module
and classify all irreducible $\N$-graded integrable highest weight $\wh{\fsl_N}(\C_q)$-modules.

Note that $\fsl_N(\C[t_0,t_0^{-1}])\oplus \C\bm\rk_0\oplus \C\bm\rd_0$ is a subalgebra of $\wt{\fsl_N}(\C_q)$,
which is isomorphic to the affine Kac-Moody algebra $\wt{\fsl_N}$ of type  $A_{N-1}^{(1)}$.
Here, we identify $\wt{\fsl_N}$ with this subalgebra of $\wt{\fsl_N}(\C_q)$:
\begin{align}
\wt{\fsl_N}=\fsl_N(\C[t_0,t_0^{-1}])\oplus \C\bm\rk_0\oplus \C\bm\rd_0\subset \wt{\fsl_N}(\C_q).
\end{align}

Set
\begin{align}
\fh=\sum_{i=1}^{N-1} \C(E_{i,i}-E_{i+1,i+1}),
\end{align}
a Cartan subalgebra of $\fsl_N(=\fsl_N(\C))$, and set
\begin{align}
H=\fh\oplus \C\bm\rk_0\oplus \C\bm\rd_0,
\end{align}
a Cartan subalgebra of the affine Kac-Moody algebra $\wt{\fsl_N}$. 

We shall use the following triangular decomposition of $\wh{\fsl_N}(\C_q)$:
\begin{align*}
\wh{\fsl_N}(\C_q)
=\wh{\fsl_N}(\C_q)^+\oplus \wh{\CH}   \oplus \wh{\fsl_N}(\C_q)^-,
\end{align*}
where
\begin{align}\label{slNq-pm}
\wh{\fsl_N}(\C_q)^\pm
=\sum_{\pm m_0\in \Z_+,m_1\in \Z} \C E_{i,j}t_0^{m_0}t_1^{m_1}+ \sum_{\pm (j-i)\in \Z_+,m\in \Z}\C E_{i,j}t_1^{m},
\end{align}
and
\begin{align}
\wh{\CH}=\sum_{1\le i\le N,\ n\in \Z}
\C h_{i,n}  +\C\bm\rk_1,
\end{align}
where
\begin{equation}\begin{split}
\label{defhin}h_{i,n}&=(E_{i,i}-E_{i+1,i+1})t_1^n\   \   \  \quad\te{for}\  1\le i\le N-1,\ n\in \Z,\\
h_{N,n}&=-q^{n}E_{1,1}t_1^n+E_{N,N}t_1^n\quad\te{for}\  n\in \Z\backslash \{0\},\\
h_{N,0}&=\bm\rk_0-(E_{1,1}-E_{N,N}).
\end{split}\end{equation}

To emphasize the dependence on the positive integer $N$, we shall also alternatively denote the subalgebra
$\wh{\CH}$ by $\wh{\CH}_N$.
Note that $\{h_{i,n}\ |\  1\le i\le N,\ n\in \Z\}\cup \{\bm\rk_1\}$
is a basis of $\wh{\CH}$.

\begin{dfnt}\label{defhwm}
{\em An $\wh{\fsl_N}(\C_q)$-module $W$ is called a {\em highest weight module with highest weight
$\lambda\in \wh{\CH}^*$}
 if there exists a nonzero vector $v_\lambda\in W$, called a {\em highest weight vector,}  such that
 $W=\U(\wh{\fsl_N}(\C_q)) v_\lambda$, $\wh{\fsl_N}(\C_q)^+ v_\lambda=0$,  and
\begin{align}\label{eta=0} h v_\lambda=\lambda(h)v_\lambda\quad \te{for}\ h\in \wh{\CH}.
\end{align}
An {\em $\N$-graded highest weight $\wh{\fsl_N}(\C_q)$-module with highest weight
$\lambda\in \wh{\CH}^*$} is a highest weight $\wh{\fsl_N}(\C_q)$-module $W$  equipped with an $\N$-grading
$W=\oplus_{n\in \N}W(n)$ such that the highest weight vector $v_\lambda\in W(0)$.}
\end{dfnt}

\begin{remt}
{\em Recall that $H=\fh+\C \bm\rk_0+\C\bm\rd_0$ is a Cartan subalgebra of $\wt{\fsl_{N}}$.
We see that an $\N$-graded highest weight $\wh{\fsl_N}(\C_q)$-module $W$
is an  $H$-weight $\wh{\fsl_N}(\C_q)\oplus \C\bm\rd_0$-module
with $\bm\rd_0$ acting trivially on the highest weight vector.
Then $W$ is a highest weight  $\wh{\fsl_N}(\C_q)\oplus \C\bm\rd_0$-module defined in \cite{ER}
with respect to the usual partial order on $H^{*}$.}
\end{remt}

Note that the condition \eqref{eta=0} implies that $\lambda(\bm\rk_1)=0$.
Now, let $\lambda\in \wh{\CH}^*$ such that $\lambda(\bm\rk_1)=0$.
Let $\C v_\lambda$ be the one-dimensional $(\wh{\fsl_N}(\C_q)^+ +\wh{\CH})$-module with
\[\wh{\fsl_N}(\C_q)^+ v_\lambda=0\quad \te{and}
\quad h v_\lambda=\lambda(h) v_\lambda\quad \te{for}\ h\in \wh{\CH}.\]
Then form an induced $\wh{\fsl_N}(\C_q)$-module
\begin{align}
V(\lambda)=\U(\wh{\fsl_N}(\C_q))\otimes_{\U(\wh{\fsl_N}(\C_q)^+ + \wh{\CH})}\C v_\lambda.
\end{align}
It is clear that $V(\lambda)$ is an $(\fh+\C\bm\rk_0+\C\bm\rk_1)$-weight module. On the other hand,
define $\deg v_\lambda=0$ to make $V(\lambda)$
an $\N$-graded highest weight $\wh{\fsl_N}(\C_q)$-module.
It follows that $V(\lambda)$ has a unique maximal $H$-weight submodule,
which we denote by $J_{\lambda}$.  Set
\begin{align}
L(\lambda)=V(\lambda)/J_{\lambda},
\end{align}
which is a  graded irreducible  $\wh{\fsl_N}(\C_q)$-module.
One can show that $L(\lambda)$ is also an irreducible $\wh{\fsl_N}(\C_q)$-module.
Furthermore, it is straightforward to see that for any $\lambda,\mu\in  \wh{\CH}^*$, $L(\lambda)\simeq L(\mu)$ as an $\N$-graded
$\wh{\fsl_N}(\C_q)$-module if and only if $\lambda=\mu$.

\begin{remt}
{\em Note that for any $\lambda\in \wh{\CH}^*$ such that $\lambda(\bm\rk_1)=0$,
 every $\N$-graded highest weight $\wh{\fsl_N}(\C_q)$-module with highest weight $\lambda$
 is naturally a homomorphism image of $V(\lambda)$ and any such irreducible  module
is isomorphic to $L(\lambda)$. 
}
\end{remt}

Next, we  identify integrable  $\wh{\fsl_N}(\C_q)$-modules among
the irreducible highest weight modules $L(\lambda)$.
We first define a class of linear functionals on $\wh{\CH}$.
Let $k$ be a positive integer. For any pair
\begin{align}\label{deflambdac}
(\bm{\lambda},\bm{c})\in (H^*)^k\times (\C^\times)^k
\end{align}
with $\bm\lambda=(\lambda_1,\dots,\lambda_k),\  \bm{c}= (c_1,\dots,c_k)$,
we define a linear functional $\eta_{\bm{\lambda},\bm{c}}$ on $\wh{\CH}$
by $\eta_{\bm{\lambda},\bm{c}}(\bm\rk_1)=0$
and
\begin{align}\label{def-eta}
\eta_{\bm{\lambda},\bm{c}}(h_{j,n})=\sum_{i=1}^k \lambda_i(h_{j,0}) c_i^n\   \quad\te{for}\ 1\le j\le N,\ n\in \Z.
\end{align}
In particular, for any $\lambda\in H^*,\ c\in \C^{\times}$, we have a linear functional $\eta_{\lambda,c}$. Then
$$\eta_{\bm{\lambda},\bm{c}}=\eta_{\lambda_1,c_1}+\cdots +\eta_{\lambda_k,c_k}.$$

Set
\begin{align}
P_+=\{\lambda\in H^*\mid \lambda(h_{i,0})\in\N,\ \lambda(\bm\rd_0)=0\quad  \te{for }1\le i\le N\},
\end{align}
the set of dominant integral weights of  $\wt{\fsl_N}$.

The following is a classification result (cf. \cite{ER,CT}):

\begin{prpt}\label{prop:intehwm}
Let $(\bm{\lambda},\bm{c})\in (P_+)^k\times (\C^\times)^k$ with $k$ a positive integer.
Then the irreducible $\N$-graded highest
weight $\wh{\fsl_N}(\C_q)$-module $L(\eta_{\bm{\lambda},\bm{c}})$ is integrable. Furthermore, every irreducible
integrable $\N$-graded highest weight $\wh{\fsl_N}(\C_q)$-module is isomorphic to a module of this form.
\end{prpt}

\begin{remt}
{\rm Note that the first assertion of  Proposition \ref{prop:intehwm} was proved
in \cite{ER}, while the second assertion was proved for Lie algebra $\wh{\fsl_N}(\C_q)\oplus \C\bm\rd_0$
under the assumption that  the $H$-weight subspaces of $W$ are finite-dimensional.}
\end{remt}

To prove Proposition \ref{prop:intehwm}, we shall need a result on level zero integrable modules for
affine Lie algebra $\wh{\fsl_2}$. Let $\{e, h, f\}$ be the standard basis of $\fsl_2$ and
let $\chi:\Z\rightarrow \C$ be any function.
An  $\wh{\fsl_2}$-module $W$ of level zero is  called a {\em loop-highest weight module
with loop-highest weight $\chi$} if there is a
nonzero vector $w$ in $W$ such that
\begin{align}
W=\U(\wh{\fsl_2})w,\quad
\(e\ot t^m\)w=0,\quad  \(h\ot t^m\)w=\chi(m)w\quad \te{for}\ m\in \Z.
\end{align}
The notion ``loop-highest weight module" (see \cite{CG} for example) is designated to distinguish
this from the standard highest weight module (see \cite{Kac}).

\begin{lemt}\label{lem:affint}
If there exists an integrable loop-highest weight $\wh{\fsl_2}$-module $W$
with a nonzero loop-highest weight $\chi$,
then there exist finitely many positive integers $p_1,\dots,p_k$ and
nonzero complex numbers $b_1,\dots,b_k$ such that
\[\chi(m)=\sum_{j=1}^k p_j\,b_j^m\quad \te{for}\ m\in \Z.\]
\end{lemt}

\begin{proof} Recall from  \cite{BZ} that $\chi$ is called an exp-polynomial if there exist finitely many
polynomials $f_1,\dots,f_k$ and
nonzero complex numbers $b_1,\dots,b_k$  such that $\chi(m)=\sum_{j=1}^k f_j(m)\,b_j^m$ for all $m\in \Z$.
A result of Kac-Jocobsen (see \cite[Proposition 6.2]{JK}) states
that if $\chi$ is not an exp-polynomial, then the Verma type loop-highest weight
$\wh{\fsl_2}$-module with loop-highest weight $\chi$ is irreducible.
On the other hand, every Verma type loop-highest weight
$\wh{\fsl_2}$-module is not integrable.
As $W$ is assumed to be integrable,  the Verma type loop-highest weight
$\wh{\fsl_2}$-module with loop-highest weight $\chi$ must be reducible and hence
$\chi$ is an exp-polynomial.

Let $w$ be a (nonzero) loop-highest weight vector in $W$.
Then $h w=nw$  for some  nonnegative integer $n$ as $w$ is a highest weight vector
 in the integrable $\fsl_2$-module $W$.
Noticing that $W=\U(\C f\otimes \C[t,t^{-1}])w$, we have
\begin{align*}
W=\oplus_{m\in \Z, m\le n} W_m\quad \mbox{ with }W_n=\C w,
\end{align*}
where $W_m=\{w\in W\mid h w=m w\}$.
Since $W$ is an integrable $\fsl_2$-module, we have $W=\oplus_{m\in \Z, |m|\le n} W_m$.
From \cite{BZ},  all the $h$-weight subspaces of the (unique) irreducible quotient module $\overline{W}$ of $W$ are
 finite-dimensional. Consequently, $\overline{W}$ is finite-dimensional.
Then this lemma follows from \cite[Proposition 2.1 (iii)]{CP}.
\end{proof}

{\bf Proof of Proposition \ref{prop:intehwm}:}
The first assertion was proved in \cite{ER}.
(This also follows from an explicit realization of $L(\eta_{\bm{\lambda},\bm{c}})$ in Section 6.)
Now, let $W$ be an irreducible integrable $\N$-graded highest weight $\wh{\fsl_N}(\C_q)$-module
with highest weight $\chi$ and highest weight vector $v$.
For $1\le i\le N-1$, set
\[ \mathfrak{a}_i=\te{Span}\{E_{i,i+1}t_1^n,\   E_{i+1,i}t_1^n, \  h_{i,n},\  \bm\rk_1\mid n\in \Z\},\]
and set
\[ \mathfrak{a}_N=\te{Span}\{q^nE_{N,1}t_0t_1^n,\   E_{1,N}t_0^{-1}t_1^n,\   h_{N,n},\  \bm\rk_1\mid n\in \Z\}.\]
It can be readily seen that ${\mathfrak a}_i$ for $1\le i\le N$ are subalgebras of $\wh{\fsl_N}(\C_q)$,
which are isomorphic to $\wh{\fsl_2}$.
Then  for every $i$,
$\U(\mathfrak a_i)v$ is an integrable loop-highest weight $\wh{\fsl_2}$-module with loop-highest weight $\chi_i$, where
$$\chi_i(n)=\chi(h_{i,n})\   \   \   \  \mbox{ for }n\in \Z.$$
By Lemma \ref{lem:affint}, for each $1\le i\le N$, there exist finitely many
nonzero complex numbers $b_{i,1},\dots, b_{i,k_i}$ and nonnegative integers $p_{i,1},\dots,p_{i,k_i}$ such that
$$\chi_i(m)=\sum_{j=1}^{k_i}p_{i,j}b_{i,j}^{m}\   \   \   \mbox{ for }m\in \Z.$$
Let $c_{1},\dots, c_{r}$ be all the distinct elements of $\{ b_{i,j}\ |\ 1\le i\le N, \ 1\le j\le k_i\}$.
Then
$$\chi_i(m)=\sum_{j=1}^{r}q_{i,j}c_{j}^{m}\   \   \   \mbox{ for }1\le i\le N,\  m\in \Z,$$
where $q_{i,j}$ are nonnegative integers which are independent of $m$.
Define $r$ linear functionals $\lambda_1,\dots,\lambda_r$ on $H$ by
$\lambda_i(h_{j,0})=q_{i,j}$ for $1\le i,j\le N$ and $\lambda_i(\bm{\rd}_0)=0$.
We have $\lambda_i\in P_{+}$ and $\chi=\eta_{{\bf \lambda},{\bf c}}$.
Then $W$ is isomorphic to $L(\eta_{\bm{\lambda},\bm{c}})$, as desired.

\begin{remt}\label{replacefun}
{\rm Recall that a fundamental dominant weight is an element $\lambda\in P_+$ such that
$\lambda(\bm\rk_0)=1$.
Let $(\bm\lambda,\bm{c})\in P_+^k\times (\C^{\times})^k$ with $\lambda_i\ne 0$ for $1\le i\le k$.
Set $\ell_i=\lambda_i(\bm\rk_0)\in  \Z_{+}$.  Furthermore, set $\ell=\ell_1+\cdots +\ell_k$.
Note that
$\eta_{\bm\lambda,\bm{c}}+\eta_{\bm\lambda',\bm{c}'}=\eta_{(\bm\lambda,\bm\lambda'),(\bm{c},\bm{c}')}$ and
$\eta_{\bm\lambda+\bm\mu,\bm{c}}=\eta_{(\bm\lambda,\bm\mu),(\bm{c},\bm{c})}$.
It then follows that there exists  a pair $(\bm{\mu},\bm{d})\in P_+^{\ell}\times (\C^{\times})^{\ell}$
such that $\eta_{\bm\lambda,\bm{c}}=\eta_{\bm\mu,\bm{d}}$,
where $\bm{\mu}=(\mu_1,\dots,\mu_\ell)$ with $\mu_{1},\dots,\mu_{\ell}$ fundamental dominant weights.  }
\end{remt}

It is known (see \cite{DLM}) that every irreducible integrable restricted module
for any (untwisted) affine Kac-Moody Lie algebra is a highest weight module.
For the extended affine Lie algebra $\wh{\frak{sl}_N}(\C_q)$, we have the following result:

\begin{prpt}\label{prop:irrintres}
Let $W=\oplus_{n\in \N}W(n)$ be a nonzero $\N$-graded integrable 
$\wh{\frak{sl}_N}(\C_q)$-module on which $\bm\rk_0$ and $\bm\rk_1$ act as  scalars $\ell_0$ and $\ell_1$, respectively.
Then $\ell_0$ is a nonnegative integer, $\ell_1=0$, and
$\Omega_W$ is a nonzero $\wh{\CH}$-submodule of $W$, where
\begin{align}
 \Omega_W=\{w\in W\mid \wh{\frak{sl}_N}(\C_q)^+w=0\}.
 \end{align}
 Furthermore, every  irreducible $\N$-graded integrable (restricted)  $\wh{\frak{sl}_N}(\C_q)$-module
  is a highest weight module.
\end{prpt}

\begin{proof} For the first assertion, by Proposition \ref{prop:charintgmod}, $\ell_0$ is a nonnegative integer.
We now prove $\ell_1=0$. Fix two integers $i,j$ with $1\le i<j\le N$.
Let $m\in \Z$. Note that the correspondence
\begin{align*}
e\mapsto E_{i,j}t_1^m, \   \   f\mapsto E_{j,i}t_1^{-m},\ \  h\mapsto E_{i,i}-E_{j,j}+m\bm\rk_1\end{align*}
gives rise to a Lie algebra embedding of $\fsl_2$ into  $\wh{\frak{sl}_N}(\C_q)$.
Under this embedding, $W$ becomes an $\fsl_2$-module on which $e$ and $f$ act locally nilpotently.
Then $W$ is an integrable $\fsl_2$-module and hence it is
a direct sum of finite-dimensional irreducible $\fsl_2$-modules. Let $v$ be a (nonzero)
highest weight vector of weight $k$. As
\[hv=(E_{i,i}-E_{j,j}+m\bm\rk_1)v= (k+m\ell_1) v,\]
 $k+m\ell_1$ must be a nonnegative integer. Thus
$k+m\ell_1\in \N$ for all $m\in \Z$, which implies $\ell_1=0$.

Now that $\bm\rk_1$ acts trivially, $W$ is an integrable module for the loop algebra $\fsl_N(\C[t_1,t_1^{-1}])$.
As $W$ is an $\N$-graded module, there exists $n\in \N$ such that $W(n)\ne 0$ and
$W(m)=0$ for all $m<n$. Then
\begin{align}\label{lowest-weight-subspace}
(E_{i,j}t_0^{m_0}t_1^{m_1})W(n)\subset W(n-m_0)=0
\end{align}
for all $1\le i,j\le N,\ m_0\ge 1,\ m_1\in \Z$.
Notice that $W(n)$ is an $\fsl_N(\C[t_1,t_1^{-1}])$-submodule of $W$.
As $W$ is an integrable restricted $\wh{\frak{sl}_N}(\C_q)$-module,
by Proposition \ref{prop:charintgmod} we have  $(E_{i,j}t_1^{m})(z)^{\ell_{0}+1}=0$ for $1\le i\ne j\le N,\ m\in \Z$.
This together with (\ref{lowest-weight-subspace}) implies  that $(E_{i,j}t_1^{m})^{\ell_{0}+1}=0$ on $W(n)$ for all $m\in \Z$.
In particular, we have $E_{i,j}^{\ell+1}=0$ on $W(n)$ for all $1\le i\ne j\le N$.
Then $W(n)$ is a direct sum of some finite-dimensional irreducible $\fsl_N$-modules $L(\lambda)$
where $\lambda$ are dominant integral weights satisfying the condition $\lambda(\theta^{\vee})\le \ell$
(with $\theta$ denoting the highest positive root).
As there are only finitely many such dominant integral weights, there exists  $\lambda\in \fh^*$ such that
$W(n)_{\lambda}\ne 0$ and  $W(n)_{\lambda+\alpha}=0$ for any positive root $\alpha$ of $\fsl_N$.
Then $(\wh{\frak{sl}_N}(\C_q)^+)W(n)_{\lambda}=0$.
Thus we have $W(n)_{\lambda}\subset \Omega_W$, proving that $\Omega_W\ne 0$.
As $[\wh{\frak{sl}_N}(\C_q)^+, \wh{\CH}]\subset \wh{\frak{sl}_N}(\C_q)^+$,
it follows that $\Omega_W$ is an $\wh{\CH}$-submodule of $W$.

Now, assume that $W$ is an irreducible $\N$-graded integrable $\wh{\frak{sl}_N}(\C_q)$-module.
Then $\bm\rk_0$ and $\bm\rk_1$ act as scalars $\ell_0$ and $\ell_1$ on $W$, respectively.
As $\bm\rk_1$ acts trivially on $W$ by the first assertion,
$\Omega_W$ becomes  a module for the abelian Lie algebra $\wh{\CH}/\C\bm\rk_1$.
Furthermore, the irreducibility of $W$ implies that $\Omega_W$ is an  irreducible $\wh{\CH}$-module.
Consequently, $\Omega_W=\C v_\lambda$ for some $v_\lambda\in \Omega_{W}\cap W_{\lambda}$.
It then follows that $W$ is a  highest weight $\wh{\frak{sl}_N}(\C_q)$-module
with highest weight vector $v_\lambda$.
\end{proof}

Combining Propositions  \ref{prop:irrintres} and \ref{prop:intehwm} we immediately have:

\begin{cort}\label{main-slNCq}
Every irreducible $\N$-graded integrable 
 $\wh{\fsl_N}(\C_q)$-module of level $\ell\in \N$ is isomorphic to
 $L(\eta_{\bm{\lambda},\bm{c}})$ for some $\bm{\lambda}=(\lambda_1,\dots,\lambda_k)\in (P_+)^k$
and $\bm{c}=(c_1,\dots,c_k)\in (\C^\times)^k$ with $k$ a positive integer such that $\sum_{i=1}^k\lambda_i(\bm\rk_0)=\ell$.
\end{cort}

Combining Corollary \ref{main-slNCq} with Remark \ref{rem:gradedver} and Theorem \ref{thm:main2},
we immediately have the  main result of this section:

\begin{thm}\label{thm:main3}
For any $\ell\in \N$, irreducible $\N$-graded $(\Z,\chi_q)$-equivariant
quasi $L_{\wh{\fsl_\infty}}(\ell,0)$-modules up to isomorphism are exactly  the irreducible integrable
$\N$-graded highest weight $\wh{\fsl_N}(\C_q)$-modules $L(\eta_{\bm{\lambda},\bm{c}})$ for
 $\bm{\lambda}=(\lambda_1,\dots,\lambda_k)\in (P_+)^k$
and $\bm{c}=(c_1,\dots,c_k)\in (\C^\times)^k$ with $k\in \Z_{+}$ such that $\sum_{i=1}^k\lambda_i(\bm\rk_0)=\ell$.
\end{thm}

Furthermore, we have:

\begin{lemt}\label{lem:isoclass}
Let $(\bm{\lambda},\bm{c})\in (H^{*})^r\times (\C^\times)^r$ with
$r$ a positive integer such that $\eta_{\bm{\lambda},\bm{c}}\ne 0$.
Then there exists
\[(\bm{\mu},\bm{d})=((\mu_1,\dots,\mu_s),(d_1,\dots,d_s))\in (H^{*})^s\times (\C^\times)^s\] for some positive integer $s$,
satisfying the condition
\begin{align}\label{lacond}
\mu_1,\dots,\mu_s\ \mbox{are nonzero} \  \mbox{ and }   \ d_1,\dots,d_s\ \mbox{are distinct},
\end{align}
such that $\eta_{\bm{\lambda},\bm{c}}=\eta_{\bm{\mu},\bm{d}}$.
Furthermore,  assume $(\bm{\mu}',\bm{d}')\in (H^{*})^{s'}\times (\C^\times)^{s'}$ is another  pair
satisfying the same condition above.
Then $\eta_{\bm{\mu},\bm{d}}=\eta_{\bm{\mu}',\bm{d}'}$
if and only if  $s=s'$ and there is a permutation $\tau\in S_s$ such that
 \[\mu_{\tau(i)}=\mu_i'\quad \te{and}\quad d_{\tau(i)}=d_i'\quad\te{for}\ i=1,\dots,s.\]
\end{lemt}

\begin{proof}  Write $\bm{\lambda}=(\lambda_1,\dots,\lambda_r)\in (H^{*})^r$
and $\bm{c}=(c_1,\dots,c_r)\in (\C^{\times})^r$. Let $c'_1,\dots,c'_k$  be all the distinct numbers
in $\{c_1,\dots,c_r\}$. Define $\bm{\lambda}'=(\lambda'_1,\dots,\lambda'_k)$ with
$\lambda'_j=\sum_{i, c_i=c'_j}\lambda_i\in H^{*}$ for $1\le j\le k$. Then
$\eta_{\bm{\lambda},\bm{c}}=\eta_{\bm{\lambda}',\bm{c}'}$. As $\eta_{\bm{\lambda},\bm{c}}\ne 0$,
we have $\lambda'_j\ne 0$ for some $j$.
Let $\mu_1,\dots,\mu_s$ be all the nonzero elements in the list $\lambda_{1}',\dots,\lambda'_k$ and let
$d_1,\dots,d_s$ be the corresponding nonzero numbers in the list $c'_1,\dots,c'_k$. Then we have
$\eta_{\bm{\lambda},\bm{c}}=\eta_{\bm{\mu},\bm{d}}$ with $\bm{\mu}=(\mu_1,\dots,\mu_s)$ and
$\bm{d}=(d_1,\dots,d_s)$, as desired.

Note that for any $(\bm{\lambda},\bm{c})\in (H^{*})^r\times (\C^\times)^r$ with $c_1,\dots,c_r$ distinct,
$\eta_{\bm{\lambda},\bm{c}}=0$ if and only if $\lambda_{i}=0$ for all $i$. Assume  $\eta_{\bm{\mu},\bm{d}}=\eta_{\bm{\mu}',\bm{d}'}$.
Set $$I=\{ 1\le i\le s\ |\  d_i\ne d'_j\  \mbox{ for all }1\le j\le s'\}.$$
If $I\ne \emptyset$, the equality $\eta_{\bm{\mu},\bm{d}}=\eta_{\bm{\mu}',\bm{d}'}$ implies
$$\sum_{i\in I}\eta_{\mu_i,d_i}+\sum_{j=1}^{s'}\eta_{\nu_j,d_j'}  =0$$
for some $\nu_j\in \fh^{*}$. This is impossible  as $d_i$ with $i\in I$ and
$d_1',\dots,d'_{s'}$ are all distinct and $\mu_i\ne 0$ for $i\in I$. Thus $I=\emptyset$, i.e.,
 $\{d_1,\dots,d_{s}\}\subset \{d_1',\dots,d'_{s'}\}$. Symmetrically, we have
 $\{d'_1,\dots,d'_{s'}\}\subset \{d_1,\dots,d_{s}\}$. Thus $\{d_1,\dots,d_{s}\}=\{d_1',\dots,d'_{s'}\}$.
  In particular, we have $s=s'$ and there exists a permutation $\tau$ such that $d_{\tau(i)}=d_i'$ for $i=1,\dots,s$.
 Furthermore, we have
 $$\eta_{\bm{\mu},\bm{d}}=\sum_{i=1}^s\eta_{\mu_{i}, d_{i}}=\sum_{i=1}^s\eta_{\mu_{\tau(i)}, d_{\tau(i)}}
 =\sum_{i=1}^s\eta_{\mu_{\tau(i)}, d'_{i}}\   \  \mbox{and }   \   \
 \eta_{\bm{\mu}',\bm{d}'}=\sum_{i=1}^s\eta_{\mu'_i,d'_{i}},$$
 which imply  $\mu_{\tau(i)}=\mu_i'$ for $i=1,\dots,s$.
 \end{proof}


\subsection{Irreducible equivariant quasi $L_{\wh{\fsl_\infty}}(\ell_1,0)\ot \cdots\ot
L_{\wh{\fsl_\infty}}(\ell_d,0)$-modules}
Let $d$ be a positive integer and let $\ell_1,\dots,\ell_d, N_1,\dots,N_d$ be positive integers with
$N_i\ge 2$ for $1\le i\le d$. Set $\bm\ell=(\ell_1,\dots,\ell_d)$.  From Example \ref{tensorex},
we have a simple $\Z$-graded vertex $\Z$-algebra
\begin{align}
L_{\wh{\fsl_\infty}}(\bm\ell,0):=(L_{\wh{\fsl_\infty}}(\ell_1,0),\rho_{N_1})\ot \cdots \ot (L_{\wh{\fsl_\infty}}(\ell_r,0),\rho_{N_d}),
\end{align}
recalling that $(L_{\wh{\fsl_\infty}}(\ell_i,0),\rho_{N_i})$ denotes the simple vertex $\Z$-algebra
$L_{\wh{\fsl_\infty}}(\ell_i,0)$ with the $\Z$-action on $\fsl_\infty$ given by $\rho_{N_i}$.
For $1\le i\le d$, we view $(L_{\wh{\fsl_\infty}}(\ell_i,0),\rho_{N_i})$ as
 a vertex $\Z$-subalgebra of $L_{\wh{\fsl_\infty}}(\bm\ell,0)$ in the obvious way.
The main goal of  this section is to prove the following (cf. \cite[Theorem 4.7.4]{FHL}):

\begin{prpt}\label{prop:tensordec}
Let $W_i$ be an irreducible  $(\Z,\chi_q)$-equivariant quasi $(L_{\wh{\fsl_\infty}}(\ell_i,0),\rho_{N_i})$-module
for  $1\le i\le d$. Then the tensor product equivariant quasi $L_{\wh{\fsl_\infty}}(\bm\ell,0)$-module
$W_1\ot \cdots\ot W_d$ is  irreducible.
On the other hand, every irreducible $\N$-graded $(\Z,\chi_q)$-equivariant quasi
$L_{\wh{\fsl_\infty}}(\bm\ell,0)$-module is isomorphic to such a tensor product module.
\end{prpt}

\begin{proof} Note that each $W_i$ is of countable dimension over $\C$, so that the Schur lemma holds.
Then the first assertion follows from the Jacobson density theorem just as in \cite{FHL}.
Now, we consider the second assertion. Let $W$ be any irreducible $\N$-graded $(\Z,\chi_q)$-equivariant quasi
$L_{\wh{\fsl_\infty}}(\bm\ell,0)$-module. Then for each $1\le i\le d$, $W$ is naturally an $\N$-graded $(\Z,\chi_q)$-equivariant quasi
$L_{\wh{\fsl_\infty}}(\ell_i,0)$-module, and for $1\le i\ne j\le d$,
the actions of $L_{\wh{\fsl_\infty}}(\ell_i,0)$ and $L_{\wh{\fsl_\infty}}(\ell_j,0)$ on $W$ commute.
By Theorem \ref{thm:main2},  $W$ is naturally an $\N$-graded  integrable and restricted
$\wh{\fsl_{N_i}}(\C_q)$-module of level $\ell_i$. Furthermore, the actions of $\wh{\fsl_{N_i}}(\C_q)$ and $\wh{\fsl_{N_j}}(\C_q)$
on $W$ commute for $1\le i\ne j\le d$.
By Proposition \ref{prop:irrintres}, we have
$\Omega_{W}^{i}\ne 0$, where
$$\Omega_{W}^{i}=\{ w\in W \mid \wh{\frak{sl}_{N_i}}(\C_q)^+w=0\}.$$
 For $1\le i\ne j\le d$, as $[ \wh{\frak{sl}_{N_i}}(\C_q), \wh{\frak{sl}_{N_j}}(\C_q)]=0$ on $W$  we have
$ \wh{\frak{sl}_{N_j}}(\C_q)\Omega_{W}^{i}\subset \Omega_{W}^{i}$.
It follows (from the first part of Proposition \ref{prop:irrintres}) that $\cap_{i=1}^d \Omega_{W}^{i}\ne 0$.
Set
\begin{align*}
\mathcal K=\wh{\fsl_{N_1}}(\C_q)\oplus \cdots \oplus \wh{\fsl_{N_d}}(\C_q).
\end{align*}
As an irreducible $\N$-graded $(\Z,\chi_q)$-equivariant quasi
$L_{\wh{\fsl_\infty}}(\bm\ell,0)$-module, $W$ is also an irreducible $\N$-graded $\mathcal K$-module.
From the proof of the second assertion of Proposition \ref{prop:irrintres},  we  conclude that $\cap_{i=1}^d \Omega_{W}^{i}=\C v$,
where $v\in \cap_{i=1}^d \Omega_{W}^{i}$ and there exists a linear functional   $\lambda$ on
$\wh\CH_{N_1}\oplus \cdots\oplus \wh\CH_{N_d}$ such that  $W=U(\mathcal K)v$ and
\[hv=\lambda(h)v\quad\te{for}\ h\in \wh\CH_{N_1}\oplus \cdots\oplus \wh\CH_{N_d}.\]

Given a linear functional $\lambda$ on $\wh\CH_{N_1}\oplus \cdots\oplus \wh\CH_{N_d}$, we say that
a $\mathcal K$-module $M$ is  a  highest weight module of highest weight  $\lambda$ if
there is a nonzero vector $w\in M$ such that $\U(\mathcal K)w=M$,
$(\wh{\fsl_{N_1}}(\C_q)^+\oplus \cdots \oplus \wh{\fsl_{N_d}}(\C_q)^+)w=0$ and
\[hw=\lambda(h)w\quad\te{for}\ h\in \wh\CH_{N_1}\oplus \cdots\oplus \wh\CH_{N_d}.\]
That is, $W$ is an $\N$-graded highest weight $\mathcal K$-module.
On the other hand,  we define Verma module $M_{\mathcal K}(\lambda)$ in the obvious way, which is a universal $\N$-graded
highest weight $\mathcal K$-module. Note that $M_{\mathcal K}(\lambda)$ has a unique irreducible $\N$-graded quotient module.
Consequently, irreducible $\N$-graded highest weight $\mathcal K$-modules of highest weight $\lambda$ are unique up to isomorphism.
 It then follows that  any irreducible $\N$-graded highest weight $\mathcal K$-module of highest weight $\lambda$ is isomorphic to
the tensor product $\mathcal K$-module
\[L(\lambda|_{\wh\CH_{N_1}})\ot \cdots \ot L(\lambda|_{\wh\CH_{N_d}}),\]
where $L(\lambda|_{\wh\CH_{N_1}})$ is the irreducible highest weight $\wh{\fsl_{N_i}}(\C_q)$-module
with highest weight $\lambda|_{\wh\CH_{N_i}}\in (\wh\CH_{N_i})^*$. This proves the second assertion.
\end{proof}

\section{Fermionic Fock modules and $(\wh{\fgl_N}(\C_q),\mathrm{GL}_{{\bf I}_{\bm a}})$-duality}
In this section, we give a family of fermionic Fock modules for $\wh{\fgl_N}(\C_q)$ and
establish a duality between $\wh{\fgl_N}(\C_q)$ and Levi subgroups of $\rGL_\ell(\C)$ .

We start by recalling  from \cite{G3} a fermionic representation of $\wh{\fgl_N}(\C_q)$.
Let $N$ and $\ell$ be positive integers with $N\ge 2$ as before.
Define $\mathcal C^{\ell}_{N}$ to be the associative (Clifford) algebra with generators
$$\psi^\mu_i(m), \   \   \bar{\psi}^\mu_i(m)\  \  (\mbox{where }1\le i\le N, \  1\le \mu\le \ell, \   m\in \Z),$$
subject to relations
\begin{align*}
&\psi^\mu_i(m)\psi^\nu_j(n)+\psi^\nu_j(n)\psi^\mu_i(m)=0,\   \   \   \  \bar\psi^\mu_i(m)\bar\psi^\nu_j(n)+\bar\psi^\nu_j(n)\bar\psi^\mu_i(m)=0,\\
&\psi^\mu_i(m)\bar\psi^\nu_j(n)+\bar\psi^\nu_j(n)\psi^\mu_i(m)=\delta_{i,j}\delta_{\mu,\nu}\delta_{m+n,0}1
\end{align*}
for $1\le i,j\le N$, $1\le \mu,\nu\le \ell$,  $m,n\in \Z$.
Form generating functions
\begin{align*}
\psi^\mu_i(x)=\sum_{m\in \Z} \psi^\mu_i(m)x^{-m}\quad\te{and}\quad
 \bar\psi_i^\mu(x)=\sum_{m\in \Z}\bar\psi_i^\mu(m)x^{-m}
\end{align*}
for $1\le i\le N$, $1\le \mu\le \ell$. It is clear that $\mathcal C^{\ell}_{N}$ is a $\Z$-graded algebra with
\begin{align}
\deg \psi_{i}^{\mu}(n)=\deg \bar\psi_{i}^{\mu}(n)=-n
\end{align}
for $1\le i\le N$, $1\le \mu\le \ell$. Set
\begin{eqnarray*}
&&(\mathcal C^{\ell}_{N})^{an}=\< \psi_i^\mu(m+1),\  \bar\psi_i^\mu(m)\ | \  1\le i\le N,\   1\le \mu\le \ell, \ m\in \N\>,\\
&&(\mathcal C^{\ell}_{N})^{cr}=\< \psi_i^\mu(-m),\  \bar\psi_i^\mu(-m-1)\ | \  1\le i\le N,\  1\le \mu\le \ell, \ m\in \N\>,
\end{eqnarray*}
(the subalgebras of $\mathcal C^{\ell}_{N}$ generated by the indicated elements). We have
$$\mathcal C^{\ell}_{N}\simeq (\mathcal C^{\ell}_{N})^{cr}\otimes (\mathcal C^{\ell}_{N})^{an}$$
as a vector space.

Denote by $\mathcal F_N^\ell$ the
$\mathcal C^{\ell}_{N}$-module generated by vector $|0\>$, subject to relations
\begin{align}
\psi_i^\mu(m+1)|0\>=\bar\psi_i^\mu(m)|0\>=0\quad \te{for}\  1\le i\le N,\   1\le \mu\le \ell, \ m\in \N.
\end{align}
Then $\mathcal F_N^\ell= (\mathcal C^{\ell}_{N})^{cr}$
as a $(\mathcal C^{\ell}_{N})^{cr}$-module. It is well known that $\mathcal F_N^\ell$ is
an irreducible $\mathcal C^{\ell}_{N}$-module.
Defining $\deg |0\>=0$,  we make $\mathcal F_N^\ell$ an $\N$-graded $\mathcal C^{\ell}_{N}$-module
\begin{align}\label{grading-FNell}
\mathcal F_N^\ell=\oplus_{n\in \N}\mathcal F_{N}^\ell(n)
\end{align}
with $\mathcal F_{N}^\ell(n)$
denoting the homogeneous subspace of degree $n$. It is straightforward to show that $\mathcal F_{N}^\ell(n)$
is finite-dimensional for every $n\in \N$.

%

Following  \cite{G3}, we introduce the following normal ordering:
\begin{equation}\label{normal-ordering-Gao}
:\psi_i^\mu(m)\bar\psi^\nu_j(n):= \begin{cases} \psi_i^\mu(m)\bar\psi^\nu_j(n)\ \ \ &\text{if}\ m\leq n,\\
-\bar\psi_j^\nu(n)\psi_i^\mu(m) &\text{if}\ m>n
\end{cases} \end{equation}
for $1\le i,j\le N$, $1\le \mu,\nu\le \ell,\  m,n\in \Z$.

\begin{remt}\label{normal-ordering-2}
{\em For $1\le i,j\le N$, $1\le \mu,\nu\le \ell,\  m,n\in \Z$, it is straightforward to show that
the normal ordering defined in (\ref{normal-ordering-Gao}) coincides with the one defined by
\begin{equation*}
:\psi_i^\mu(m)\bar\psi^\nu_j(n):= \begin{cases} \psi_i^\mu(m)\bar\psi^\nu_j(n)\ \ \ &\text{if}\ n\geq 0,\\
-\bar\psi_j^\nu(n)\psi_i^\mu(m) &\text{if}\ n<0.
\end{cases}
\end{equation*}}
\end{remt}

For $1\le i,j\le N,\  m\in \Z$ and $\bm{a}=(a_1,\dots,a_\ell)\in (\C^\times)^\ell$,   define a field
\begin{equation*}
\Psi_{i,j}^{\bm{a}}(m,x)
=\begin{cases}
\sum_{p=1}^{\ell}:\psi_i^p(x)\bar\psi_j^p(x):
& \mbox{ if }m=0\\
\sum_{p=1}^\ell a_p^m\(:\psi_i^p(x)\bar\psi_j^p(q^mx):
+\delta_{i,j}\frac{q^m}{1-q^m}\)& \mbox{ if }m\ne 0.
\end{cases}\end{equation*}
Furthermore,  use the expansion
\begin{align}
\Psi_{i,j}^{\bm{a}}(m,x)=\sum_{n\in \Z}\Psi_{i,j}^{\bm{a}}(n,m)x^{-n}
\end{align}
to define operators $\Psi_{i,j}^{\bm{a}}(n,m)\in \mathrm{End}\(\mathcal F_N^\ell\)$.
Explicitly, we have
\begin{equation}
\begin{split}\label{defpsiija}\Psi_{i,j}^{\bm{a}}(n,0)
&=\sum_{p=1}^\ell\sum_{k\in \Z}:\psi_i^p(n-k)\bar\psi_j^p(k):,\\
\Psi_{i,j}^{\bm{a}}(n,m)
&=\sum_{p=1}^\ell a_p^m\bigg(\sum_{k\in \Z}q^{-mk}:\psi_i^p(n-k)\bar\psi_j^p(k):+\delta_{n,0}\delta_{i,j}\frac{q^m}{1-q^m}\bigg)\end{split}
\end{equation}
for $m\ne 0$.
The following is a generalization of a result of \cite{G3}:

\begin{prpt}\label{lem:algact}
Let $\ell$ be a positive integer and
let $\bm{a}\in \(\C^\times\)^\ell$. Then
$\mathcal F_N^\ell$ is a $\wh{\fgl_N}(\C_q)$-module with $\bm\rk_0=\ell$,  $\bm\rk_1= 0$, and
\begin{align}\label{nuaction}
E_{i,j}t_0^{m_0}t_1^{m_1}= \Psi_{i,j}^{\bm{a}}(m_0,m_1)\quad\te{ for }1\le i,j\le N,\  m_0,m_1\in \Z.
\end{align}
 Furthermore,
 $\mathcal F_N^\ell$ is an integrable and restricted  $\wh{\fsl_N}(\C_q)$-module which is $\N$-graded with
 finite-dimensional homogeneous subspaces.
\end{prpt}

\begin{proof}  In case that $\ell=1$ and $\bm{a}=a_{1}=1$, this was proved in \cite{G3}.
Denote this particular $\wh{\fgl_N}(\C_q)$-module simply by $\mathcal F_N$.
Now, we consider the general case. Associated to the $\ell$-tuple $\bm{a}$, we have an evaluation
 $\wh{\fgl_N}(\C_q)$-module $(\mathcal F_N)(a_{1})\ot \cdots\ot (\mathcal F_N)(a_{\ell})$, where
\[(\mathcal F_N)(a_{1})\ot \cdots\ot (\mathcal F_N)(a_{\ell})
=\mathcal F_N^{\ot \ell}\]
as a vector space and where $\bm\rk_0=\ell,\   \   \bm\rk_1=0$, and
\begin{align*}
E_{i,j}t_0^{m_0}t_1^{m_1}(v_1\ot \cdots\ot v_\ell)
=\sum_{r=1}^\ell a_{r}^{m_1}\left(v_1\ot \cdots\ot E_{i,j}t_0^{m_0}t_1^{m_1} v_r\ot \cdots \ot v_\ell\right)
\end{align*}
for $1\le i,j\le N$, $m_0,m_1\in \Z$, and $v_r\in \mathcal F_N$.
Then through the natural identification of  $\mathcal F_N^\ell$
with $(\mathcal F_N)(a_{1})\ot \cdots\ot (\mathcal F_N)(a_{\ell})$,
we obtain a $\wh{\fgl_N}(\C_q)$-module structure on $\mathcal F_N^\ell$ with the action given by \eqref{nuaction}
as desired.

For the second assertion, using the fact that $\mathcal F_N^\ell$ is a
 tensor product of level $1$ $\wh{\fsl_N}(\C_q)$-modules, it suffices to deal with the case with $\ell=1$.
It is clear from \eqref{defpsiija} that the $\wh{\fsl_N}(\C_q)$-module $\mathcal F_N$ is restricted.
For the integrability, in view of Proposition \ref{prop:charintgmod}, it suffices to show that
for $1\le i\ne j\le N$ and $m_1\in \Z$,
\begin{align*}
\Psi_{i,j}^{\bm{a}}(m_1,x_1)\Psi_{i,j}^{\bm{a}}(m_1,x_2)
=\Psi_{i,j}^{\bm{a}}(m_1,x_2)\Psi_{i,j}^{\bm{a}}(m_1,x_1)
\end{align*}
and $\Psi_{i,j}^{\bm{a}}(m_1,x)^2=0$ on $ \mathcal F_N$.
Note that
\begin{align*}
\psi_i^{1}(x_1)\psi_i^{1}(x_2)=-\psi_i^{1}(x_2)\psi_i^{1}(x_1),\   \   \   \
\bar\psi_j^{1}(x_1)\bar\psi_j^{1}(x_2)=-\bar\psi_j^{1}(x_2)\bar\psi_j^{1}(x_1),
\end{align*}
which imply $\psi_i^{1}(x)^2=0$ and $\bar\psi_j^{1}(x_1)^2=0$.
On the other hand, as $i\ne j$, we also have
$\psi_i^{1}(x_1)\bar\psi_j^{1}(x_2)=-\bar\psi_j^{1}(x_2)\psi_i^{1}(x_1)$,
which implies
$$:\psi_i^{1}(x)\bar\psi_j^{1}(q^{m_1}x):=\psi_i^{1}(x)\bar\psi_j^{1}(q^{m_1}x).$$
Then
$$\Psi_{i,j}^{\bm{a}}(m_1,x)=a_{1}^{m_1}:\psi_i^{1}(x)\bar\psi_j^{1}(q^{m_1}x):
=a_{1}^{m_1}\psi_i^{1}(x)\bar\psi_j^{1}(q^{m_1}x).$$
It follows that
\begin{align*}
\Psi_{i,j}^{\bm{a}}(m_1,x_1)\Psi_{i,j}^{\bm{a}}(m_1,x_2)
=\Psi_{i,j}^{\bm{a}}(m_1,x_2)\Psi_{i,j}^{\bm{a}}(m_1,x_1)
\end{align*}
and
\begin{align*}
\Psi_{i,j}^{\bm{a}}(m_1,x_1)\Psi_{i,j}^{\bm{a}}(m_1,x_2)
=-a_{1}^{2m_1}\psi_i^{1}(x_1)\psi_i^{1}(x_2)\bar\psi_j^{1}(q^{m_1}x_1)\bar\psi_j^{1}(q^{m_1}x_2),
\end{align*}
which implies $\Psi_{i,j}^{\bm{a}}(m_1,x)^2=0$.
This completes the proof.
\end{proof}

\begin{dfnt}
{\em For $\bm{a}\in \(\C^\times\)^\ell$, denote by
\begin{align}
\rho_{\bm{a}}: \  \wh{\fgl_N}(\C_q)\longrightarrow \mbox{End} (\mathcal F_N^\ell)
\end{align}
the representation of $\wh{\fgl_N}(\C_q)$ on $\mathcal F_N^\ell$
obtained in Proposition \ref{lem:algact}, and we also denote by $\mathcal F_{N}^{\bm{a}}$
the $\wh{\fgl_N}(\C_q)$-module.}
\end{dfnt}

Next, we present an analogue  in the toroidal setting of the skew $(\fgl_N(\C),\mathrm{GL}_\ell(\C))$-duality.
Let $n$ be a positive integer. Set  $\rGL_{n}=\mathrm{GL}_{n}(\C)$,
let $\rH_n$ be the subgroup of $\rGL_{n}$, consisting of diagonal matrices, and
let $\rN_n^+$ be the subgroup consisting of upper-triangular unipotent matrices.
Following \cite{GW}, set
\begin{align}
\Z_{++}^n=\{\bm{\mu}=(\mu_1,\dots,\mu_n)\in \Z^n\mid \mu_1\ge \mu_2\ge \cdots \ge \mu_n\}
\end{align}
and view each $\bm{\mu}=(\mu_1,\dots,\mu_n)\in \Z_{++}^n$ as a dominant regular character of $\rH_n$ by
\begin{align*}
 \bm{\mu}(h)= h^{\bm\mu}=h_1^{\mu_1}\cdots h_n^{\mu_n}\quad\te{for}\ h=\mathrm{diag}\{h_1,\dots,h_n\}\in \rH_{n}.
\end{align*}
 For  $\bm{\mu}\in \Z_{++}^n$, denote by $L_{\rGL_n}(\bm{\mu})$ the
irreducible highest weight $\rGL_n$-module of highest weight $\bm{\mu}$.
A basic fact is that  $L_{\rGL_n}(\bm{\mu})$ is regular, finite-dimensional, and the space of
$\rN_n^+$-fixed vectors of weight $\bm\mu$ is $1$-dimensional.

More generally, let $\ell$ be a positive integer and let  ${\bf I}$ be a partition  of $\{1,\dots,\ell\}$
with the associated equivalence relation denoted by $\sim$.
Define
\begin{align}
\rGL_{\bf I}=\{(c_{ij})_{i,j=1}^\ell\in \rGL_\ell\mid c_{ij}=0\ \te{if}\ i\not\sim j\},
\end{align}
 the Levi subgroup of $\rGL_{\ell}$ associated to $\bf I$.
Correspondingly, we have the following linear algebraic subgroups of $\rGL_{\bf I}$:
\begin{align}
\rN^+_{\bf I}=\rN^+_\ell\cap \rGL_{\bf I}\quad\te{and}\quad \rH_{\bf I}=\rH_\ell\cap \rGL_{\bf I}.
\end{align}
Then we identify dominant regular characters of $\rH_{\bf I}$ with elements of
\begin{align}
\Z_{++}^{\bf I}:=\{ \bm \mu=(\mu_1,\dots,\mu_\ell)\in \Z^\ell\ |\  \mu_i\ge \mu_j\   \mbox{ for }1\le i,j\le \ell,\  i\sim j,\ i< j\}.
\end{align}
For each $\bm{\mu}\in \Z_{++}^{\bf I}$,
denote  by $L_{\rGL_{\bf I}}(\bm{\mu})$
 the   finite-dimensional irreducible highest weight $\rGL_{\bf{I}}$-module of highest weight $\bm{\mu}$.

Let  $W$ be any locally regular $\rGL_{\bf I}$-module of countable dimension.
Then  $W$ is a direct sum of  finite-dimensional irreducible regular
$\rGL_{\bf I}$-modules (cf. \cite{GW}):
\begin{align*}
W= \bigoplus_{\bm\mu\in P(W)} W^{\rN_{\bf I}^+}(\bm\mu)
\ot L(\bm{\mu}),
\end{align*}
where $W^{\rN_{\bf I}^+}(\bm\mu)$ denotes
the space of   $\rN_{\bf I}^+$-fixed vectors  of weight  $\bm\mu$ in $W$
and $P(W)=\{\bm\mu\in\Z_{++}^{\bf I}\mid W^{\rN_{\bf I}^+}(\bm\mu)\ne 0\}$.

\begin{dfnt}
{\em Let $\CL$ be a Lie algebra and let $V$ be a vertex $\Gamma$-algebra.
We say that $(\mathcal L,\rGL_{\bf I})$ (resp.\,$(V,\rGL_{\bf I})$) is a (Howe) dual pair on a
locally regular $\rGL_{\bf I}$-module $W$ if \te{(i)}
$W$ is an $\mathcal L$-module (resp.  equivariant quasi $V$-module)
such that the actions of  $\mathcal L$ (resp. $V$) and $\rGL_{\bf I}$ commute.
\te{(ii)} For every  $\bm\mu\in P(W)$, the $\CL$-submodule (resp. equivariant quasi $V$-submodule)
$W^{\rN_{\bf I}^+}(\bm\mu)$ of $W$ is irreducible.
\te{(iii)} For any $\bm\mu,\bm\nu\in P(W)$,
 $W^{\rN_{\bf I}^+}(\bm\mu)\cong W^{\rN_{\bf I}^+}(\bm\nu)$ if and only if
 $\bm\mu=\bm\nu$.}
 \end{dfnt}

Recall from \cite{FF} that $\mathcal F_N^\ell$ is a module for the general Lie algebra $\fgl_\ell$  with
\begin{align}\label{laction}
E_{r,s}=\sum_{i=1}^N\sum_{n\in \Z}:\psi_i^r(-n)\bar\psi_i^{s}(n):
\end{align}
for $1\le r,s\le \ell$.
It is clear that $\mathcal F_N^\ell=\oplus_{n\in \Z}\mathcal F_{N}^\ell(n)$
(recall (\ref{grading-FNell})) is  a direct sum
of finite-dimensional $\fgl_\ell$-submodules.
Then this  $\fgl_\ell$-module structure on $\mathcal F_N^\ell$ gives rise to
a module structure for  $\rGL_\ell$.
Furthermore, for any partition ${\bf I}$ of $\{1,\dots,\ell\}$, $\mathcal F_N^\ell$ is
a locally regular $\rGL_{\bf{I}}$-module.


\begin{dfnt}
{\em For any positive integer $\ell$, define
\begin{align*}
(\C^\times)_q^\ell=\{(a_1,\dots,a_\ell)\in (\C^\times)^\ell\mid \te{either}\ a_i=a_j \ \te{or}\ a_i\notin a_j\Gamma_q\ \te{for}\ 1\le i,j\in \ell\},
\end{align*}
where $\Gamma_q=\{ q^n\ |\ n\in \Z\}$. Furthermore, for any $\bm{a}\in (\C^\times)^\ell_q$, we define
an equivalence relation $\sim_{\bm a}$ on $\{1,\dots,\ell\}$ by  $i\sim_{\bm a} j$ if and only if $a_i=a_j$ for $1\le i,j\le \ell$, and denote by
 ${\bf I}_{\bm a}$ the   partition of $\{1\,\dots,\ell\}$ associated to $\sim_{\bm a}$.}
\end{dfnt}


The following is the main result of this section:

\begin{thmt}\label{prop:dualitygl}
For any $\bm{a}\in (\C^\times)^{\ell}_{q}$ with $\ell$ a  positive integer,
 $(\wh{\fgl_N}(\C_q),\rGL_{{\bf I}_{\bm a}})$ is a dual pair on the Fock space $\mathcal F_N^\ell$
with the $\wh{\fgl_N}(\C_q)$-module structure afforded by $\rho_{\bm{a}}$.
\end{thmt}

The rest of this section is devoted to the proof of this theorem.
We start with a duality result of \cite{F}.
 Set
 \begin{align}
 \overline\fgl_{\infty}=\fgl_\infty\oplus \C\bm\rk,
 \end{align}
  a one-dimensional central extension of Lie algebra $\fgl_\infty$, where
 \begin{align*}
 [A,B]=AB-BA+\sum_{i\le 0}\mathrm{tr}([E_{i,i},A]B) \bm\rk
 \end{align*}
for $A,B\in \fgl_\infty$ (see \cite{DJKM}).
It is known (see \cite{F}) that
$\mathcal F_N^\ell$ is a $\overline\fgl_{\infty}$-module of level $\ell$
(with $\bm\rk=\ell$), where
 \begin{align}\label{degree2gen}
E_{mN+i,nN+j}=
\sum_{p=1}^\ell :\psi_i^p(-m)\bar\psi_j^p(n):
\end{align}
for $1\le i,j\le N,\   m,n\in \Z$.
Furthermore, we have (see \cite{F}; cf. \cite{FKRW,W}):

\begin{prpt}\label{duality1}
Let  $N$ and $\ell$ be positive integers.
Then $(\overline\fgl_{\infty},\rGL_\ell)$ is a dual pair on $\mathcal F_N^\ell$.
Furthermore, for any partition ${\bf I}=\{S_1,\dots,S_s\}$ of $\{1,\dots,\ell\}$,
$(\overline\fgl_{\infty}^{\oplus s},\rGL_{\bf I})$ is a dual pair on
$\mathcal F_N^\ell$.
\end{prpt}

Note that for Proposition \ref{duality1} the $\overline\fgl_{\infty}^{\oplus s}$-action on $\mathcal F_N^{\ell}$
 is given by
  \begin{align}\label{fglsaction}
 \imath_r(\bm\rk)= |S_r|, \  \  \imath_r(E_{Nm+i,Nn+j})&=
\sum_{p\in S_r} :\psi_{i}^p(-m)\bar\psi_{j}^p(n):
\end{align}
for $1\le r\le s$, $1\le i,j\le N,\  m,n\in \Z$, where  $\imath_r$ denotes
the identification map of $\overline\fgl_{\infty}$ with the $r$-th component of
$\overline\fgl_{\infty}^{\oplus s}$.

\begin{lemt}\label{lem:samesubmod}
Let  $\ell$ be a positive integer and  let $\bm{a}\in (\C^\times)^{\ell}_{q}$ with the associated partition
${\bf I}_{\bm a}=\{S_1,\dots, S_s\}$.
Then a subspace $U$ of $\mathcal F_N^\ell$
is a $\wh{\fgl_N}(\C_q)$-submodule through $\rho_{\bm{a}}$ if and only if
$U$ is a $\overline\fgl_{\infty}^{\oplus s}$-submodule.
 Furthermore, for any  $\wh{\fgl_N}(\C_q)$-submodules $V$ and $W$
of $\mathcal F_{N}^{\ell}$, a linear map $f$ from $V$ to $W$ is a $\wh{\fgl_N}(\C_q)$-module homomorphism
 if and only if $f$ is a $\overline\fgl_{\infty}^{\oplus s}$-module homomorphism.
\end{lemt}

\begin{proof} For $1\le i,j\le N,\   m_0,m_1\in \Z$, set
$$E_{i,j}(m_0,0)=E_{i,j}t_0^{m_0},$$
\begin{align*}
E_{i,j}(m_0,m_1)=E_{i,j}t_0^{m_0}t_1^{m_1}+\delta_{i,j}\delta_{m_0,0}\frac{q^{m_1}}
{q^{m_1}-1}\left(\sum_{r=1}^\ell a_r^{m_1}\right) \frac{\bm\rk_0}{\ell}\in \wh{\fgl_N}(\C_q)
\end{align*}
for $m_1\ne 0$. Then a subspace $U$ of $\mathcal F_{N}^{\ell}$ $(=\mathcal F_{N}^{\bm{a}})$
is a $\wh{\fgl_N}(\C_q)$-submodule
if and only if  $E_{i,j}(m_0,m_1)U\subset U$ for all $1\le i,j\le N,\   m_0,m_1\in \Z$.
On the other hand, from \eqref{defpsiija} and \eqref{nuaction} we have
\begin{align}\label{wteij}
E_{i,j}(m_0,m_1)v=\sum_{k\in \Z}\sum_{r=1}^s (b_rq^{-k})^{m_1}
\(\sum_{p\in S_r}:\psi_i^p(m_0-k)\bar\psi_j^p(k):v\)
\end{align}
for $v\in \mathcal F_{N}^{\bm{a}}$, where $b_r=a_p$ for $p\in S_r$. This together with \eqref{fglsaction} implies that
every $\overline\fgl_{\infty}^{\oplus s}$-submodule of $\mathcal F_N^\ell$  is
a $\wh{\fgl_N}(\C_q)$-submodule of $\mathcal F_{N}^{\bm{a}}$.

Now, let $V$ be a
$\wh{\fgl_N}(\C_q)$-submodule of $\mathcal F_{N}^{\bm{a}}\ (=\mathcal F_N^\ell)$.
Let $1\le i,j\le N$, $m_0\in \Z$ and $v\in V$ be fixed. Then there is a finite subset $K$ of $\Z$
such that
\[\sum_{p\in S_r}:\psi_i^p(m_0-k)\bar\psi_j^p(k):v=0\quad\te{for all}\ 1\le r\le s,\  k\notin K.\]
Using this we get
\begin{align*}
\sum_{k\in K}\sum_{r=1}^s (b_rq^{-k})^{m_1}
\(\sum_{p\in S_r}:\psi_i^p(m_0-k)\bar\psi_j^p(k):v\)\in V
\end{align*}
for $1\le m_1\le s|K|$.
As $q$ is not a root of unity and $b_i\notin b_j\Gamma_q$ for $1\le i\ne j\le s$, we have  that
$b_rq^{-k}$ are distinct for all $1\le r\le s,\ k\in \Z$.
Then we obtain
\[\sum_{p\in S_r}:\psi_i^p(m_0-k)\bar\psi_j^p(k):v\in V\   \   \mbox{ for }1\le r\le s,\ k\in K.\]
Note that this is also true for $k\notin K$.
Combining this with \eqref{fglsaction}, we conclude that $V$ is a $\overline\fgl_{\infty}^{\oplus s}$-submodule
of $\mathcal F_N^\ell$.
This proves the first assertion.
The second assertion can be proved similarly.
 \end{proof}

%
%

{\bf Proof of Theorem \ref{prop:dualitygl}:} Now we can conclude the proof:
From Proposition \ref{duality1}, the $\fgl_\ell$-action given by \eqref{laction} on $\mathcal F_N^\ell$ with
$$E_{r,s}=\sum_{i=1}^{N}\sum_{n\in \Z}:\psi_i^r(-n)\bar{\psi}_i^s(n):$$
for $1\le r,s\le \ell$ commutes with
the action of $\overline\fgl_{\infty}$ with
$$E_{mN+i,nN+j}=\sum_{r=1}^\ell:\psi_i^r(-m)\bar{\psi}_j^r(n):$$
for $m,n\in \Z,\ 1\le i,j\le N$.
Then we see that the $\fgl_\ell$-action given by \eqref{laction} on $\mathcal F_N^\ell$ commutes
with the $\wh{\fgl_N}(\C_q)$-action given by \eqref{nuaction} and (\ref{defpsiija})
through $\rho_{\bm{a}}$.  Since
$(\overline\fgl_{\infty}^{\oplus s},\rGL_{{\bf I}_{\bm a}})$ is a dual pair on $\mathcal F_N^\ell$, it follows immediately
from Lemma \ref{lem:samesubmod} that $(\wh{\fgl_N}(\C_q),\rGL_{{\bf I}_{\bm a}})$ is a dual pair on
$\mathcal F_N^\ell$.

\begin{remt}
{\rm  With $\fgl_N(\C)$ a subalgebra of $\wh{\fgl_N}(\C_q)$, the $\wh{\fgl_N}(\C_q)$-module
$\mathcal F_{N}^{\bm{a}}$ is naturally a $\fgl_N(\C)$-module.
Note that the $\fgl_N(\C)$-action on $\mathcal F_N^{\ell}$ is independent of $\bm{a}$
and the degree-zero subspace $\mathcal F_{N}^\ell(0)$ is a $\fgl_N(\C)$-submodule of
$\mathcal F_{N}^{\bm{a}}$.  It was known (see \cite{H2}) that
$(\fgl_N(\C),\rGL_\ell)$ is a dual pair on $\mathcal F_{N}^{\ell}(0)$.
In view of this, Theorem \ref{prop:dualitygl} can be viewed as an extension of the skew  $(\fgl_N(\C),\rGL_\ell)$-duality.}
\end{remt}

\section{$(L_{\wh{\fsl_\infty}}(\ell,0),\mathrm{GL}_{{\bf I}_{\bm a}})$-duality and Fermionic realization}
Let $\ell$ be a positive integer as before.
Recall from Lemma \ref{lem:algact} that for every $\bm{a}\in (\C^\times)^\ell$,
 we have an $\N$-graded  integrable restricted $\wh{\fsl_N}(\C_q)$-module $\mathcal F_{N}^{\bm{a}}\ (=\mathcal F_N^\ell)$
 of level $\ell$.
Then by Theorem \ref{thm:main2}, $\mathcal F_{N}^{\bm{a}}$ is naturally an $\N$-graded
 $(\Z,\chi_q)$-equivariant quasi $L_{\wh{\fsl_\infty}}(\ell,0)$-module.
In this section, we continue to show that for any $\bm{a}\in (\C^{\times})^{\ell}_{q}$,
$\mathcal F_{N}^{\bm{a}}$
is a direct sum of irreducible quasi $L_{\wh{\fsl_\infty}}(\ell,0)$-submodules with finite multiplicities.
Furthermore, we prove that every irreducible $\N$-graded $(\Z,\chi_q)$-equivariant
quasi $L_{\wh{\fsl_\infty}}(\ell,0)$-module is isomorphic to an irreducible submodule of $\mathcal F_{N}^{\bm{a}}$
for some $\bm{a}\in (\C^{\times})^{\ell}_{q}$.


We start with some notations. First, for $\mu\in \Z$, we define integers $\dot{\mu}$ and $\ddot{\mu}$  by
\begin{align}\label{mu-decomp}
\mu=\dot{\mu}N+\ddot{\mu}\quad\te{with}\  1\le \ddot{\mu}\le N.
\end{align}

\begin{dfnt}\label{defetamua}
{\em  For any $\bm{\mu}=(\mu_1,\dots,\mu_\ell)\in \Z^\ell$ and $\bm{a}=(a_1,\dots,a_\ell)\in (\C^\times)^\ell$, we define
a linear functional $\eta_{\bm\mu,\bm a}$ on $\wh{\CH}$ by
$\eta_{\bm\mu,\bm a}(\bm\rk_1)=0$ and
\begin{align}
\eta_{\bm\mu,\bm a}(h_{i,n})=\sum_{k=1}^\ell (a_k\,q^{-\dot{\mu}_k})^{n}\delta_{\ddot{\mu}_k,i}
\end{align}
for $1\le i\le N,\ n\in \Z$.}
\end{dfnt}

\begin{remt}\label{connection-12}
{\em For $1\le i\le N$, define $\Lambda_i\in \mathcal H^{*}$ by
\begin{eqnarray}
\Lambda_i(h_{j,0})=\delta_{i,j}\   \   \mbox{ for }1\le j\le N.
\end{eqnarray}
For $\bm\mu\in \Z^\ell$ and $\bm{a}\in (\C^\times)^\ell$, set
$\bm{\lambda}=(\Lambda_{\ddot{\mu}_{1}},\dots,\Lambda_{\ddot{\mu}_{\ell}})$ and
$\bm{c}=(a_{1}q^{-\dot{\mu}_{1}},\dots, a_{\ell} q^{-\dot{\mu}_{\ell}})$.
Then we have $\eta_{\bm{\mu},\bm{a}}=\eta_{\bm{\lambda},\bm{c}}$ which was defined in (\ref{def-eta}).
Furthermore, let $\bm\nu\in \Z^\ell$ and $\bm{b}\in (\C^\times)^\ell$.
Then it follows from Lemma \ref{lem:isoclass}
 that $\eta_{{\bm\mu},\bm{a}}=\eta_{{\bm\nu},\bm{b}}$ if  and only if
 there is a permutation
$\sigma$ on $\{1,\dots,\ell\}$ and an $\bm m=(m_1,\dots,m_\ell)\in \Z^\ell$ such that
\[\sigma(\bm \mu)=\bm\nu+N\bm m=(\nu_1+Nm_1,\dots,\nu_\ell+Nm_\ell),\]
and that \[\sigma(\bm{a})=q^{\bm m}\bm b:
=(q^{m_1}b_1,\dots,q^{m_\ell}b_\ell).\]
 }
\end{remt}

\begin{remt}\label{remark-J}
{\em For $1\le i\le N,\ \mu\in \Z_{+}$, set
\begin{align}
J_{i,\mu}=\{ k\in \Z_{\ge 0}\ |\  kN+i\le \mu\}.
\end{align}
From definition, for $1\le i\le N-1$ we have
\begin{align}
J_{i,\mu}=\begin{cases}
J_{i+1,\mu}\cup \{ \dot{\mu}\} &\te{ if }i=\ddot{\mu}\\
J_{i,\mu}=J_{i+1,\mu}&\te{ otherwise}.\end{cases}
\end{align}
We also have
\begin{align}
J_{1,\mu}=\{0,1,\dots,\dot{\mu}\}=
\begin{cases}
J_{N,\mu} &\te{ if } \ddot{\mu}= N\\
J_{N,\mu}\cup \{ \dot{\mu}\} &\te{ otherwise}.
\end{cases}
\end{align}

On the other hand, for a negative integer $\mu$, we set
\begin{align}
\bar{J}_{i,\mu}=\{ k\in \Z_{\ge 0}\ |\  (k+1)N-i+1\le -\mu\}=J_{N+1-i,-\mu}.
\end{align}
For $1\le i\le N-1$, we have
$\bar{J}_{i+1,\mu}=\bar{J}_{i,\mu}\cup \{ -\dot{\mu}-1\}$ if $i=\ddot{\mu}$ and
$\bar{J}_{i,\mu}=\bar{J}_{i+1,\mu}$ otherwise.
We also have
\begin{align}
\bar{J}_{N,\mu}=
\begin{cases}
\{0,1,\dots,-\dot{\mu}-2\}=\bar{J}_{1,\mu} &\te{ if } \ddot{\mu}= N\\
\{0,1,\dots,-\dot{\mu}-1\}=\bar{J}_{1,\mu}\cup \{ -\dot{\mu}-1\} &\te{ otherwise}.
\end{cases}
\end{align}}
\end{remt}

Set $I_{N}=\{ 1,2,\dots,N\}$ for convenience.  Define maps $\pi$ and $\bar{\pi}: \Z\times I_N\rightarrow \Z$ by
\begin{align}
\pi(n,i)=nN-i,\  \  \  \    \bar{\pi}(n,i)=nN+i-1\   \  \mbox{ for }(n,i)\in \Z\times I_N.
\end{align}
Note that both $\pi$ and $\bar{\pi}$ are bijections and
$$\pi(\Z_{+}\times I_N)=\Z_{\ge 0},\  \  \pi(\Z_{-}\times I_N)=\Z_{-},\  \
\bar{\pi}(\Z_{\ge 0}\times I_N)=\Z_{\ge 0},\  \  \bar{\pi}(\Z_{-}\times I_N)=\Z_{-}.$$
For convenience, we define $\phi^p_m,\ \bar{\phi}^p_m$ for $1\le p\le \ell,\ m\in \Z$ by
$$\phi^p_{\pi(n,i)}=\psi_i^p(n),\   \   \   \ \bar{\phi}^p_{\bar{\pi}(n,i)}=\bar{\psi}^p_i(n)$$
for $1\le i\le N,\ n\in \Z$, i.e.,
\begin{eqnarray}
\phi^p_{nN-i}=\psi_i^p(n),\   \   \   \   \   \  \bar{\phi}^p_{nN+i-1}=\bar{\psi}^p_i(n).
\end{eqnarray}
Then
\begin{eqnarray}
&&\phi^p_{m}\phi^q_n+\phi^q_n\phi^p_m=0,\   \   \    \   \bar{\phi}^p_{m}\bar{\phi}^q_n+\bar{\phi}^q_n\bar{\phi}^p_m=0,\\
&&\phi^p_{m}\bar{\phi}^q_n+\bar{\phi}^q_n\phi^p_m=\delta_{p,q}\delta_{m+n+1,0}
\end{eqnarray}
for $1\le p,q\le \ell,\ m,n\in \Z$.

For $\mu\in \Z,\ 1\le p\le \ell$, we define an element $A^p(\mu)$ of the Clifford algebra $\mathcal C_N^\ell$
in terms of $\phi^p_n$ and $\bar{\phi}^p_n$ for $n\in \Z$ by
\begin{eqnarray}
A^{p}(\mu)=\begin{cases}\phi^{p}_{-\mu}\cdots \phi^p_{-2} \phi^p_{-1}&\mbox{ for }\mu\ge 1,\\
1&\mbox{ for }\mu=0,\\
\bar{\phi}^p_{\mu}\cdots \bar{\phi}^p_{-2} \bar{\phi}^p_{-1}&\mbox{ for }\mu\le -1.
\end{cases}
\end{eqnarray}
Let $\mu$ be a positive integer. Then
\begin{align}
\phi^p_{n}A^p(\mu)|0\>=0\   \   &\mbox{ for }n\ge 0,\   \   \bar{\phi}^p_{n}A^p(\mu)|0\>=0\   \mbox{ for }n\ge \mu,\\
\phi^p_{-n}A^p(\mu)=0\   \  & \mbox{ for }1\le n\le \mu.
\end{align}
(Note that $\phi^p_{n}A^p(\mu)=(-1)^{\mu}A^p(\mu)\phi^p_n$ and $\phi^p_{-n}\phi^p_{-n}=0$.)
Similarly, we have
\begin{align}
\bar{\phi}^p_{n}A^p(-\mu)|0\>=0\   \   &\mbox{ for }n\ge 0,\   \   \phi^p_{n}A^p(-\mu)|0\>=0\   \mbox{ for }n\ge \mu,
\label{property-01}\\
\bar{\phi}^p_{-n}A^p(-\mu)=0\   \  & \mbox{ for }1\le n\le \mu.\label{property-02}
\end{align}
We also have
\begin{align}\label{property-03}
A^p(\mu)A^q(\nu)=\pm A^q(\nu)A^p(\mu)\   \   \   \mbox{ for } 1\le p\ne q\le \ell,\ \mu,\nu\in \Z.
\end{align}

Using these properties, we have (cf. \cite{W}):

\begin{prpt}\label{lem:jointhwv}
Let ${\bm a}\in (\C^{\times})^{\ell}_{q}$ with $\ell$ a positive integer and let
$\bm{\mu}\in \Z_{++}^{{\bf I}_{\bm a}}$.
Set
\begin{align}
v_{{\bm{\mu}}}=A(\bm{\mu})|0\>=A^1(\mu_1)\cdots A^\ell(\mu_\ell)|0\>\in \mathcal F_N^\ell.
\end{align}
Then $v_{\bm{\mu}}$ is a joint highest weight vector in
$\mathcal F_{N}^{\bm{a}}=\mathcal F_{N}^{\ell}$ in the sense that
\begin{align}\label{hwv1}
\wh{\fsl_N}(\C_q)^+ v_{\bm{\mu}}=0,\quad
h v_{\bm{\mu}}=\eta_{\bm{\mu},\bm{a}}(h)v_{\bm{\mu}}\quad \te{for}\ h\in \wh{\CH}
\end{align}
and that
\begin{align}\label{hwv2}
\rN_{{\bf I}_{\bm a}}^+ v_{\bm{\mu}}=v_{\bm{\mu}},\quad
hv_{\bm{\mu}}=\bm{\mu}(h) v_{\bm{\mu}}\quad
\te{for}\ h\in \rH_{{\bf I}_{\bm a}}.
\end{align}
\end{prpt}

\begin{proof} First, from (\ref{property-01})-(\ref{property-03}) we have the following facts for $1\le p\le \ell,\ n\in \Z$:
\begin{align}
\phi^{p}_{n}A(\bm\mu)|0\>=0& \   \   \mbox{ if }n\ge -\mu_p,\label{propert-11}\\
\bar{\phi}^{p}_{n}A(\bm\mu)|0\>=0& \   \   \mbox{ if  } n\ge \mu_p.\label{propert-12}
\end{align}

(A) Annihilation property $\wh{\fsl_N}(\C_q)^+ v_{\bm{\mu}}=0$.

Let $1\le i,j\le N,\ n\in \Z$ such that either $n>0$ or $n=0$ and $i<j$.
We claim
\begin{align}
:\psi_i^{p}(n-k)\bar\psi_j^p(k): A(\bm\mu)|0\>=0\quad\te{for all}\  1\le p\le \ell,\  k\in \Z.
\end{align}

(I) The case with $n=0$ and $i<j$. Note that it is equivalent to prove
$$\sum_{k\in \Z}f(k):\psi_i^p(-k)\bar{\psi}_j^p(k):A(\bm\mu)|0\>=0$$
for any complex-valued function $f$ on $\Z$. From definition, for $1\le p,q\le \ell$ we have
\begin{eqnarray}\label{general-exp}
&&\sum_{k\in \Z}f(k):\psi_i^p(-k)\bar{\psi}_j^q(k):\\
&=&\sum_{k\ge 0}f(k)\psi_i^p(-k)\bar{\psi}_j^q(k)
-\sum_{k\le -1}f(k)\bar{\psi}_j^q(k)\psi_i^p(-k)\nonumber\\
&=&\sum_{k\ge 0}f(k)\phi^p_{-kN-i}\bar{\phi}^q_{kN+j-1}
-\sum_{k\le -1}f(k)\bar{\phi}^q_{kN+j-1}\phi^p_{-kN-i}\nonumber\\
&=&\sum_{k\ge 0}\left(f(k)\phi^p_{-kN-i}\bar{\phi}^q_{kN+j-1}-f(-k-1)\bar{\phi}^q_{-(k+1)N+j-1}\phi^p_{(k+1)N-i}\right).
\nonumber
\end{eqnarray}
Let $k\ge 0$. Assume  $\mu_p\ge 0$.
We have  $\phi^p_{(k+1)N-i}A(\bm\mu)|0\>=0$ as $(k+1)N-i\ge 0$.
One the other hand, if  $kN+j>\mu_p$, then $\bar{\phi}^p_{kN+j-1}A(\bm\mu)|0\>=0$.
If $kN+j\le \mu_p$, then $kN+i<kN+j\le \mu_p$.
In this case, we have $\phi^p_{-kN-i}\bar{\phi}^p_{kN+j-1}A(\bm\mu)|0\>=0$ as
$$\phi^p_{-kN-i}\bar{\phi}^p_{kN+j-1}A(\bm\mu)=-\bar{\phi}^p_{kN+j-1}\phi^p_{-kN-i}A(\bm\mu)
=0.$$
Thus $\sum_{k\in \Z}f(k):\psi_i^p(-k)\bar{\psi}_j^p(k):A(\bm\mu)|0\>=0$ from (\ref{general-exp}).
The case with $\mu_p< 0$ follows from a similar argument.

(II) The case with $n>0$. Note that
$$:\psi_i^{p}(n-k)\bar\psi_j^p(k):=\psi_i^{p}(n-k)\bar\psi_j^p(k)= -\bar\psi_j^p(k)\psi_i^{p}(n-k).$$
We here prove this for the case $\mu_p>0$ while the case $\mu_p<0$ follows from a similar argument.
If $k<n$, we have
$$\psi_i^{p}(n-k)A(\bm\mu)|0\>=\pm A(\bm\mu)\psi_i^{p}(n-k)|0\>
=0.$$
Assuming $k\ge n$, we have
$$\psi_i^{p}(n-k)\bar\psi_j^p(k)A(\bm\mu)|0\>=\phi^p_{(n-k)N-i}\bar{\phi}^p_{kN+j-1}A(\bm\mu)|0\>.$$
If $kN+j> \mu_p$, we have $\bar{\phi}^p_{kN+j-1}A(\bm\mu)|0\>=0.$
If $kN+j\le \mu_p$, we have
$$(n-k)N-i=-kN-j+(nN+j-i)> -\mu_p,$$ so that $\phi^p_{(n-k)N-i}A(\bm\mu)=0$. Hence
$$\phi^p_{(n-k)N-i}\bar{\phi}^p_{kN+j-1}A(\bm\mu)|0\>=-\bar{\phi}^p_{kN+j-1}\phi^p_{(n-k)N-i}A(\bm\mu)|0\>=0.$$
Therefore we have $:\psi_i^{p}(n-k)\bar\psi_j^p(k):A(\bm\mu)|0\>=0$ for $\mu_p>0$.
 Then it follows from (\ref{slNq-pm}) and (\ref{nuaction}) (and (\ref{defpsiija})) that $\wh{\fsl_N}(\C_q)^+ v_{\bm{\mu}}=0$.

(B) Determination of the highest weight.  

For $1\le i\le N,\ n\in \Z$, from  \eqref{defpsiija} and \eqref{nuaction} we have
$$E_{i,i}t_1^n=
\sum_{p=1}^{\ell}\sum_{k\in \Z}(a_pq^{-k})^n:\psi_i^p(-k)\bar{\psi}_i^p(k):+(1-\delta_{n,0})\frac{a_p^nq^n}{1-q^n}.
$$
Then
\begin{align*}
h_{i,n}&=(E_{i,i}-E_{i+1,i+1})t_1^{n}\\
&=\sum_{p=1}^{\ell}\sum_{k\in \Z}(a_pq^{-k})^n\left(:\psi_i^p(-k)\bar{\psi}_i^p(k):
-:\psi_{i+1}^p(-k)\bar{\psi}_{i+1}^p(k):\right)
\end{align*}
 for $1\le i\le N-1,\ n\in \Z$,
\begin{align*}
h_{N,0}={\bf k}_0-E_{1,1}+E_{N,N}=\ell+\sum_{p=1}^{\ell}\sum_{k\in \Z}\left(:\psi_N^p(-k)\bar{\psi}_N^p(k):
-:\psi_{1}^p(-k)\bar{\psi}_{1}^p(k):\right),
\end{align*}
\begin{align*}
h_{N,n}&=-q^nE_{1,1}t_1^n+E_{N,N}t_1^n\\
& =\sum_{p=1}^{\ell}\sum_{k\in \Z}(a_pq^{-k})^n\left(
:\psi_{N}^p(-k)\bar{\psi}_{N}^p(k):-q^n:\psi_1^p(-k)\bar{\psi}_1^p(k):\right)+a_p^nq^n
\end{align*}
for $n\ne 0$.

{\bf Claim B:} For $1\le p\le \ell$, we have
\begin{equation*}
\sum_{k\in \Z}f(k):\psi_i^p(-k)\bar\psi_i^p(k):A(\bm\mu)|0\>=\begin{cases}
\left(\sum_{k\in J_{i,\mu_p}}f(k)\right)A(\bm\mu)|0\>\ &\te{if}\ \mu_p>0,\\
-\left(\sum_{k\in \bar{J}_{i,\mu_p}}f(k)\right)A(\bm\mu)|0\>\ &\te{if}\ \mu_p<0,
\end{cases}
\end{equation*}
where $f$ is any complex-valued function on $\Z$.

We here give a proof for the case $\mu_p>0$ while the case $\mu_p<0$ follows from a similar argument.
Let $1\le i\le N$, $1\le p\le \ell$, $k\ge 0$.  We have
$$\bar{\phi}^p_{-kN-i}\phi^p_{kN+i-1}A(\bm\mu)|0\>=\pm \bar{\phi}^p_{-kN-i}A(\bm\mu)\phi^p_{kN+i-1}|0\>
=0.$$
On the other hand, we have
 $$\phi^p_{-kN-i}\bar{\phi}^p_{kN+i-1}A(\bm\mu)|0\>=\begin{cases}0\   \  &\mbox{ if }kN+i>\mu_p,\\
A(\bm\mu)|0\>\   \  &\mbox{ if }kN+i\le \mu_p,
 \end{cases} $$
noticing that $\bar{\phi}^p_{kN+i-1}A(\bm\mu)|0\>=\pm A(\bm\mu)\bar{\phi}^p_{kN+i-1}|0\>=0$ if $kN+i> \mu_p$ and
$$\phi^p_{-kN-i}\bar{\phi}^p_{kN+i-1}A(\bm\mu)|0\>
 =-\bar{\phi}^p_{kN+i-1}\phi^p_{-kN-i}A(\bm\mu)|0\>+A(\bm\mu)|0\>=A(\bm\mu)|0\>$$
 if $kN+i\le \mu_p$, as $\phi^p_{-kN-i}A(\bm\mu)=0$.
Then Claim B follows from (\ref{general-exp}) in this case.

Now, using Claim B and Remark \ref{remark-J}, we get
$h_{i,n} A(\bm\mu)|0\>=\eta_{\bm\mu,\bm a}(h_{i,n}) A(\bm\mu)|0\>$ for $1\le i\le N,\ n\in \Z$, as desired.

The other assertions can be proved similarly and we omit the details.
\end{proof}

As the main result of this section, we have:

\begin{thm}\label{thm:main4}
Let $\ell$ be a positive integer and let
 $\bm{a}\in (\C^\times)^\ell_{q}$.
Then both
$(\wh{\fsl_N}(\C_q),\rGL_{{\bf I}_{\bm a}})$  and $(L_{\wh{\fsl_\infty}}(\ell,0),\rGL_{{\bf I}_{\bm a}})$ are dual pairs on
$\mathcal F^{\bm a}_{N}$. Furthermore, we have
\begin{align}\label{decomfnella}
\mathcal F_{N}^{\bm{a}}= \bigoplus_{{\bm{\mu}}\in \Z_{++}^{\bf I}}
 L(\eta_{{\bm{\mu}},\bm{a}})\ot L_{\rGL_{{\bf I}_{\bm a}}}({\bm{\mu}})
\end{align}
as an $(\wh{\fsl_N}(\C_q),\rGL_{{\bf I}_{\bm a}})$-module.
\end{thm}

\begin{proof} We write $\bf I$ simply for ${\bf I}_{\bm a}$.
From Theorem \ref{prop:dualitygl} we have
\begin{align}\label{decomfnella1}
\mathcal F_{N}^{\bm{a}}= \bigoplus_{{\bm\mu}\in \Z_{++}^{\bf I}}
(\mathcal F_{N}^{\bm{a}})^{\rN_{\bf I}^+}({\bm\mu})\ot L_{\rGL_{\bf I}}({\bm\mu})
\end{align}
as a $(\wh{\fgl_N}(\C_q),\rGL_{\bf I})$-module, where  $(\mathcal F_{N}^{\bm{a}})^{\rN_{\bf I}^+}({\bm\mu})$
if nonzero is an irreducible $\wh{\fgl_N}(\C_q)$-module.
As
$(\mathcal F_{N}^{\bm{a}})^{\rN_{\bf I}^+}({\bm\mu})\ne 0$
for ${\bm\mu}\in \Z_{++}^{\bf I}$ by Proposition \ref{lem:jointhwv},
it must be an irreducible $\wh{\fgl_N}(\C_q)$-module, so that
the identity matrix $I_N$ acts as a scalar.
Then it follows from \eqref{decwhfgl} that
 $(\mathcal F_{N}^{\bm{a}})^{\rN_{\bf I}^+}({\bm\mu})$ is  an irreducible $\wh{\fsl_N}(\C_q)$-module.
Also from Proposition \ref{lem:jointhwv}, $v_{{\bm{\mu}}}$ is a (nonzero) highest weight vector
 in $(\mathcal F_{N}^{\bm{a}})^{\rN_{\bf I}^+}({\bm\mu})$
 of weight  $\eta_{{\bm{\mu}},\bm{a}}$.
 It follows that $(\mathcal F_{N}^{\bm{a}})^{\rN_{\bf I}^+}({\bm\mu})$ is a highest weight
 irreducible $\wh{\fsl_N}(\C_q)$-module isomorphic to $L(\eta_{{\bm{\mu}},\bm{a}})$.
 Thus we have the  decomposition  \eqref{decomfnella}.
 Finally, the dual pair assertion follows from the fact that
 $L(\eta_{{\bm{\mu}},\bm{a}})\cong L(\eta_{{\bm{\nu}},\bm{a}})$ if and only if
 $\eta_{{\bm{\mu}},\bm{a}}=\eta_{{\bm{\nu}},\bm{a}}$,
 which is equivalent to ${\bm{\mu}}={\bm{\nu}}$ by Remark \ref{connection-12}.
\end{proof}


\begin{remt}
{\em Note that for any partition $\bf I$ of $\{1,\dots,\ell\}$, there exists
 $\bm{a}\in (\C^\times)_q^\ell$ such that ${\bf I}={\bf I}_{\bm a}$.
 It then follows from Theorem \ref{thm:main4} that every irreducible regular module for the
Levi subgroup $\rGL_{\bf I}$ of $\rGL_\ell$ occurs in the decomposition \eqref{decomfnella1}.
On the other hand, every Levi subgroup of $\rGL_\ell$ is conjugate to $\rGL_{\bf I}$
for some partition $\bf I$ of $\{1,\dots,\ell\}$. }
\end{remt}

As the first application of Theorem \ref{thm:main4}, we have the following realization
of irreducible $\N$-graded integrable restricted $\wh{\fsl_N}(\C_q)$-modules of level $\ell$,
or equivalently irreducible $\N$-graded $(\Z,\chi_q)$-equivariant quasi $L_{\wh{\fsl_\infty}}(\ell,0)$-modules:

\begin{thm}\label{thm:ferreal}
Let $\ell$ be a positive integer.
Then for any ${\bm a}\in (\C^{\times})^{\ell}_{q}$ and
$\bm\mu\in \Z_{++}^{{\bf I}_{\bm a}}$,  the space
$(\mathcal F_{N}^{\bm{a}})^{\rN_{{\bf I}_{\bm a}}^+}(\bm\mu)$ of
the $\rN_{{\bf I}_{\bm a}}^+$-fixed points of weight
 $\bm\mu$ in $\mathcal F_{N}^{\bm{a}}$
is an irreducible $\N$-graded $(\Z,\chi_q)$-equivariant quasi
$L_{\wh{\fsl_\infty}}(\ell,0)$-module.
On the other hand, every  irreducible $\N$-graded $(\Z,\chi_q)$-equivariant quasi
$L_{\wh{\fsl_\infty}}(\ell,0)$-module can be realized this way.
\end{thm}

\begin{proof} The first assertion follows from  Theorem \ref{thm:main4}.  As for the second assertion,
recall from (the proof of) Theorem \ref{thm:main4} that $(\mathcal F_{N}^{\bm{a}})^{\rN_{{\bf I}_{\bm a}}^+}(\bm\mu)
\cong L(\eta_{\bm\mu,\bm{a}})$.
Thus it suffices to show that for any pair $(\bm\lambda,\bm{c})\in (P_+)^k\times (\C^\times)^k$ in Theorem \ref{thm:main3},
there exist  $\bm{a}\in (\C^\times)^\ell_{q}$  and
$\bm\mu\in \Z_{++}^{{\bf I}_{\bm a}}$
such that $\eta_{\bm\mu,\bm{a}}=\eta_{\bm{\lambda},\bm{c}}$.
In view of Remark \ref{replacefun}, it suffices to consider the case
 with $k=\ell$ and $\bm{\lambda}=(\lambda_1,\dots,\lambda_\ell)$ such that
 $\lambda_1,\dots,\lambda_\ell\in P_+$ are all fundamental.

Define an equivalence relation on $\{1,\dots,\ell\}$ such that $i\sim j$ if and only if $c_i=q^n c_j$
for some $n\in \Z$. Let ${\bf I}=\{S_1,\dots, S_s\}$ be the  partition of $\{1,\dots,\ell\}$ associated to $\sim$.
For $1\le r\le s$, pick a representative $b_r\in \{c_i\ |\ i\in S_r\}$.
Define $\bm{a}=(a_1,\dots,a_\ell)\in (\C^\times)^\ell$ where $a_k=b_r$ if $k\in S_r$.
Then $\bm{a}$ lies in $(\C^\times)^\ell_{q}$ and
${\bf I}={\bf I}_{\bm a}$ is the partition of $\{1,\dots,\ell\}$ associated to $\bm{a}$.
As $\lambda_1,\dots,\lambda_\ell\in P_+$ are all fundamental,
for each $1\le i\le \ell$ there is a unique $m_i\in \{1,\dots,N\}$ such that
$\lambda_i(h_{j,0})=\delta_{j,m_i}$ for $1\le j\le N$.
On the other hand, for each $1\le i\le \ell$, there exist unique $1\le r\le s$ and $n_i\in \Z$
 such that $c_i=b_r q^{-n_i}$. Then define $\mu_i=N n_i+m_i$.
Since $\eta_{\tau\bm{\lambda},\tau\bm{c}}=\eta_{\bm{\lambda},\bm{c}}$ for any permutation $\tau$ on $\{1,\dots,\ell\}$,
 we can assume $\mu_i\ge \mu_{j}$ for $i\sim j$ and $i<j$.
Then the $\ell$-tuple $\bm{\mu}:=(\mu_1,\dots,\mu_\ell)$ lies in $\Z_{++}^{{\bf I}}$,
and we have $\eta_{\bm{\mu},\bm{a}}=\eta_{\bm\lambda,\bm{c}}$, as desired.
\end{proof}

\begin{remt}
{\rm Note that although the $\Z$-graded vertex algebra $L_{\wh{\fsl_\infty}}(\ell,0)$ has infinite-dimensional homogenous subspaces,
it follows from Theorem \ref{thm:ferreal} and Proposition \ref{lem:algact} that all the
 homogenous subspaces of every irreducible $\N$-graded $(\Z,\chi_q)$-equivariant quasi
$L_{\wh{\fsl_\infty}}(\ell,0)$-module are finite-dimensional.}
\end{remt}

\section{$L_{\wh{\fsl_\infty}}(\ell,0)\otimes L_{\wh{\fsl_\infty}}(\ell',0)\rightarrow
L_{\wh{\fsl_\infty}}(\ell+\ell',0)$ branching}

Recall that for any nonzero complex number $\ell$, the $\Z$-graded vertex algebra $L_{\wh{\fsl_\infty}}(\ell,0)$
contains $\fsl_\infty$ as its degree-one subspace which
generates $L_{\wh{\fsl_\infty}}(\ell,0)$ as a vertex algebra.
Let $\ell$ and $\ell'$ be positive integers which are fixed throughout this section.
 It is known that $L_{\wh{\fsl_\infty}}(\ell+\ell',0)$ as an $\wh{\fsl_\infty}$-module is isomorphic to the submodule
 generated by the vacuum vector
 ${\bf 1}\otimes {\bf 1}$ in $L_{\wh{\fsl_\infty}}(\ell,0)\ot L_{\wh{\fsl_\infty}}(\ell',0)$.
 This gives rise to a $\Z$-graded vertex algebra embedding of $L_{\wh{\fsl_\infty}}(\ell+\ell',0)$ into
 $L_{\wh{\fsl_\infty}}(\ell,0)\ot L_{\wh{\fsl_\infty}}(\ell',0)$, which is uniquely determined by
 \begin{align}\label{vaembed}
u\mapsto u\ot  \bm{1}+\bm{1}\ot  u\quad\te{for}\ u\in \fsl_\infty.
\end{align}
 It is clear that this is also a vertex $\Z$-algebra embedding with the particular $\Z$-module structures
 defined in Section 3.2.
 We then view $L_{\wh{\fsl_\infty}}(\ell+\ell',0)$ as a vertex $\Z$-subalgebra of
 $L_{\wh{\fsl_\infty}}(\ell,0)\ot L_{\wh{\fsl_\infty}}(\ell',0)$.
 In this section, we shall determine the $L_{\wh{\fsl_\infty}}(\ell,0)\otimes L_{\wh{\fsl_\infty}}(\ell',0)\rightarrow
L_{\wh{\fsl_\infty}}(\ell+\ell',0)$ branching rule for irreducible equivariant quasi modules.

In view of Proposition \ref{prop:tensordec} and Theorem \ref{thm:ferreal},
every irreducible $\N$-graded $(\Z,\chi_q)$-equivariant quasi
module for $L_{\wh{\fsl_\infty}}(\ell,0)\ot L_{\wh{\fsl_\infty}}(\ell',0)$ is isomorphic to
\[L(\eta_{{\bm\mu},\bm{a}})\ot L(\eta_{{\bm\nu},\bm{b}})\]
 for some pairs
\begin{align*}
 (\bm{a},\bm{\mu})
\in (\C^{\times})_{q}^{\ell}\times \Z_{++}^{{\bf I}_{\bm a}}\quad \te{and}\quad
(\bm{b},{\bm\nu})
\in  (\C^{\times})_{q}^{\ell'}\times \Z_{++}^{{\bf I}_{\bm b}}.
\end{align*}
For any $1\le i\le \ell$, set $m_i=0$ if $a_i\notin b_j\Gamma_q$ for any $j$ and otherwise set $m_i$ to be the unique integer such that
 $q^{m_i}a_i\in \{b_1,\dots,b_{\ell'}\}$.
Furthermore, set $\bm{m}=(m_1,\dots,m_\ell)\in \Z^\ell$.
  Recall from Remark \ref{connection-12} that
  \[\eta_{{\bm\mu},\bm{a}}=\eta_{\bm\mu+N\bm{m},q^{\bm m}\bm{a}}.\]
It is straightforward to see that $q^{\bm m}\bm{a}\in (\C^{\times})^\ell_q$, $\bm\mu+N\bm{m}\in \Z_{++}^{{\bf I}_{q^{\bm m}\bm a}},$ and
 \[(q^{\bm{m}}\bm{a},\bm{b})\in (\C^\times)^{\ell+\ell'}_q.\]
 In view of this, it suffices to consider $L(\eta_{{\bm\mu},\bm{a}})\ot L(\eta_{{\bm\nu},\bm{b}})$ with
$(\bm{a},\bm{b})\in (\C^\times)^{\ell+\ell'}_q$.

View ${\rm GL}_{\ell}\times {\rm GL}_{\ell'}$ as a subgroup of ${\rm GL}_{\ell+\ell'}$ in the obvious way.
Then we have a natural group embedding
\[\rGL_{{\bf I}_{\bm a}}\times \rGL_{{\bf I}_{\bm b}}(\subset {\rm GL}_{\ell}\times {\rm GL}_{\ell'})\hookrightarrow
\rGL_{{\bf I}_{(\bm a,\bm b)}}(\subset {\rm GL}_{\ell+\ell'}).\]
Write the $\rGL_{{\bf I}_{(\bm a,\bm b)}}\rightarrow \rGL_{{\bf I}_{\bm a}}\times \rGL_{{\bf I}_{\bm b}}$ branching as
\begin{align}\label{branchingrg}
\Res_{\rGL_{{\bf I}_{\bm a}}\times \rGL_{{\bf I}_{\bm b}}}^{\rGL_{{\bf I}_{(\bm a,\bm b)}}}
L_{\rGL_{{\bf I}_{(\bm a,\bm b)}}}(\bm{\xi})= \bigoplus_{\bm{\mu}\in \Z_{++}^{{\bf I}_{\bm a}},\
\bm{\nu}\in \Z_{++}^{{\bf I}_{\bm b}}} D_{\bm{\mu}, \bm{\nu}}^{\bm{\xi}}
 \left(L_{\rGL_{{\bf I}_{\bm a}}}(\bm{\mu})\ot
L_{\rGL_{{\bf I}_{\bm b}}}(\bm{\nu})\right)
\end{align}
for $\bm{\xi}\in \Z_{++}^{{\bf I}_{(\bm a,\bm b)}}$, where
$D_{\bm{\mu},\bm{\nu}}^{\bm{\xi}}$ denotes
the (finite) multiplicity  of
the $\rGL_{{\bf I}_{\bm a}}\times \rGL_{{\bf I}_{\bm b}}$-module $L_{\rGL_{{\bf I}_{\bm a}}}(\bm{\mu})\ot
L_{\rGL_{{\bf I}_{\bm b}}}(\bm{\nu})$ in $L_{\rGL_{{\bf I}_{(\bm a,\bm b)}}}(\bm{\xi})$.
Using this decomposition we have the following main result of this section:

\begin{thm}\label{thm:tensordec}
Let $\ell,\ell'$ be positive integers and let
$$\bm a\in  (\C^\times)^\ell_q,\   \   \bm b\in (\C^\times)^{\ell'}_q,\   \
{\bm\mu}\in \Z_{++}^{{\bf I}_{\bm a}},\   \
{\bm\nu}\in   \Z_{++}^{{\bf I}_{\bm b}} $$
such that  $(\bm a, \bm b)\in (\C^\times)^{\ell+\ell'}_q$.
Then
$L(\eta_{{\bm\mu},\bm{a}})\ot L(\eta_{{\bm\nu},\bm{b}})$
as a quasi $L_{\wh{\fsl_\infty}}(\ell+\ell',0)$-module decomposes into irreducible submodules as
\begin{align}\label{coffximu}
\Res_{L_{\wh{\fsl_\infty}}(\ell+\ell',0)}^{L_{\wh{\fsl_\infty}}(\ell,0)\ot L_{\wh{\fsl_\infty}}(\ell',0)}
L(\eta_{{\bm\mu},\bm{a}})\ot L(\eta_{{\bm\nu},\bm{b}})
= &\bigoplus_{\bm{\xi}\in \Z_{++}^{{\bf I}_{(\bm a,\bm b)}}}
D_{\bm{\mu}, \bm{\nu}}^{\bm{\xi}}\,
L(\eta_{{\bm\xi},(\bm{a},\bm{b})}).
\end{align}
\end{thm}

\begin{proof} Identify the Fock space $\mathcal F_N^{\bm{a}}\ot \mathcal F_N^{\bm{b}}$ $(=\mathcal F_N^{\ell}\ot \mathcal F_N^{\ell'})$
 with $\mathcal F_N^{(\bm a,\bm b)}$  $(=\mathcal F_N^{\ell+\ell'})$ canonically.
By Theorem \ref{thm:main4}, we have the following two dual pairs
\begin{align}\label{seesaw1}
(L_{\wh{\fsl_\infty}}(\ell,0)\ot L_{\wh{\fsl_\infty}}(\ell',0), \rGL_{{\bf I}_{\bm a}}\times \rGL_{{\bf I}_{\bm b}})
\quad\te{and}\quad (L_{\wh{\fsl_\infty}}(\ell+\ell',0), \rGL_{{\bf I}_{(\bm a,\bm b)}})
\end{align}
 on
 $\mathcal F_N^{\bm{a}}\ot \mathcal F_N^{\bm{b}}\  (=\mathcal F_N^{(\bm a,\bm b)})$.
Recall the following canonical embeddings
\begin{align*}
L_{\wh{\fsl_\infty}}(\ell+\ell',0)\hookrightarrow L_{\wh{\fsl_\infty}}(\ell,0)\ot L_{\wh{\fsl_\infty}}(\ell',0)
\quad\te{and}\quad \rGL_{{\bf I}_{\bm a}}\times \rGL_{{\bf I}_{\bm b}}\hookrightarrow \rGL_{{\bf I}_{(\bm a,\bm b)}}.
\end{align*}
Notice that the $L_{\wh{\fsl_\infty}}(\ell+\ell',0)$-action on $\mathcal F_N^{(\bm{a},\bm{b})}$ coincides with that on
$\mathcal F_N^{\bm{a}}\ot \mathcal F_N^{\bm{a}}$ via the embedding
$L_{\wh{\fsl_\infty}}(\ell+\ell',0)\hookrightarrow L_{\wh{\fsl_\infty}}(\ell,0)\ot L_{\wh{\fsl_\infty}}(\ell',0)$, while
the $\rGL_{{\bf I}_{\bm a}}\times \rGL_{{\bf I}_{\bm b}}$-action on $\mathcal F_N^{(\bm{a},\bm{b})}$
via the embedding
$\rGL_{{\bf I}_{\bm a}}\times \rGL_{{\bf I}_{\bm b}}\hookrightarrow \rGL_{{\bf I}_{(\bm a,\bm b)}}$ coincides with that on
$\mathcal F_N^{\bm{a}}\ot \mathcal F_N^{\bm{a}}$.
That is, the two dual pairs in  \eqref{seesaw1} form a seesaw pair in the sense of \cite{Ku}.
The rest essentially follows from  the  reciprocity law attached to this seesaw pair
 as in the classical case (see \cite{H2}).

First, by \eqref{decomfnella} we have
\begin{align*}
\mathcal F_{N}^{\bm{a}}= \bigoplus_{\bm{\mu}\in \Z_{++}^{{\bf I}_{\bm a}}}
 L(\eta_{\bm{\mu},\bm{a}})\ot L_{\rGL_{{\bf I}_{\bm a}}}(\bm{\mu})\quad\te{and}\quad
 \mathcal F_{N}^{\bm{b}}= \bigoplus_{\bm{\nu}\in \Z_{++}^{{\bf I}_{\bm b}}}
 L(\eta_{\bm{\nu},\bm{n}})\ot L_{\rGL_{{\bf I}_{\bm b}}}(\bm{\nu}).
\end{align*}
Then
\begin{align}\label{tenprodec11}
\mathcal F_{N}^{\bm{a}}\ot
\mathcal F_{N}^{\bm{b}}
= \bigoplus_{\bm{\mu}\in \Z_{++}^{{\bf I}_{\bm a}},\  \bm{\nu}\in \Z_{++}^{{\bf I}_{\bm b}}}
\(L(\eta_{\bm{\mu},\bm{a}})\ot L(\eta_{\bm{\nu},\bm{a}})\)
\ot \(L_{\rGL_{{\bf I}_{\bm a}}}(\bm{\mu})\ot L_{\rGL_{{\bf I}_{\bm b}}}(\bm{\nu})\)
\end{align}
as an $(L_{\wh{\fsl_\infty}}(\ell+\ell',0),\rGL_{{\bf I}_{\bm a}}\times \rGL_{{\bf I}_{\bm b}})$-module.
On the other hand, by \eqref{decomfnella}  we have
\begin{align*}
\mathcal F_N^{{\bf I}_{(\bm a,\bm b)}}
= \bigoplus_{\bm{\xi}\in \Z_{++}^{{\bf I}_{(\bm a,\bm b)}}}
L(\eta_{\bm{\xi},(\bm{a},\bm{b})})\ot L_{\rGL_{{\bf I}_{(\bm a,\bm b)}}}(\bm{\xi}).
\end{align*}
Combining this with \eqref{branchingrg}, we get
\begin{align}\label{tenprodec12}
\mathcal F_N^{{\bf I}_{(\bm a,\bm b)}}=
\bigoplus_{\bm{\mu}\in \Z_{++}^{{\bf I}_{\bm a}},\  \bm{\nu}\in \Z_{++}^{{\bf I}_{\bm b}}}
\bigoplus_{\bm{\xi}\in \Z_{++}^{{\bf I}_{(\bm a,\bm b)}}}
D_{\bm{\mu},\bm{\nu}}^{\bm{\xi}}\, L(\eta_{\bm{\xi},(\bm{a},\bm{b})})\ot
(L_{\rGL_{{\bf I}_{\bm a}}}(\bm{\mu})\ot L_{\rGL_{{\bf I}_{\bm b}}}(\bm{\nu}))
\end{align}
as an $(L_{\wh{\fsl_\infty}}(\ell+\ell',0),\rGL_{{\bf I}_{\bm a}}\times \rGL_{{\bf I}_{\bm b}})$-module.
With the identification $\mathcal F_N^{(\bm{a},\bm{b})}
=\mathcal F_N^{\bm{a}}\ot \mathcal F_N^{\bm{b}}$,
combining \eqref{tenprodec11} with \eqref{tenprodec12} we obtain (\ref{coffximu}) as desired.
\end{proof}

It is known (cf. \cite{GW}) that for any positive integer $n$, the branching
from $\rGL_{n+1}$ to $\rGL_{n}\times \rGL_1$ on every finite-dimensional irreducible module is multiplicity free.
Then as an immediate consequence of Theorem \ref{thm:tensordec} (with  $\ell'=1$) we have:

\begin{cort}
The branching from $L_{\wh{\fsl_\infty}}(\ell,0)\ot L_{\wh{\fsl_\infty}}(1,0)$ to
$L_{\wh{\fsl_\infty}}(\ell+1,0)$ on every irreducible $\N$-graded $(\Z,\chi_q)$-equivariant quasi module
is multiplicity free.
\end{cort}

It is important to note that the branching from $L_{\wh{\fsl_\infty}}(\ell,0)\ot L_{\wh{\fsl_\infty}}(1,0)$ to
$L_{\wh{\fsl_\infty}}(\ell+1,0)$ on irreducible $\N$-graded modules
are {\em not} multiplicity free.

\begin{remt}
{\rm In view of Theorem \ref{thm:main2}, Theorem \ref{thm:tensordec} also gives an explicit tensor product decomposition
for any two irreducible integrable highest weight $\wh{\fsl_N}(\C_q)$-modules.
Let $\mathcal O_{int}^{f}$ be the category of $\wh{\fsl_N}(\C_q)$-modules which are sums of finitely many irreducible
integrable highest weight modules. By Theorem \ref{thm:tensordec} we
 conclude that $\mathcal O_{int}^{f}$ is  a semisimple abelian tensor category.
 This can be viewed as a $q$-analogue of a result of Wang \cite[Theorem 5.2]{W}.}
\end{remt}

\section{$L_{\wh{\fsl_\infty}}(\ell,0)\rightarrow L_{\wh{\fsl_{\infty,{\bf N}}}}(\ell,0)$ branching law}
Let $\ell,d$ be positive integers  and
 let ${\bf N}=(N_1,\dots,N_d)\in \Z_+^d$ with $N_1\ge 2,\dots,N_d\ge 2$.
Set
\begin{align}
N_{(0)}=0,\   \  N_{(r)}=N_1+\cdots+N_r\   \   \mbox{  for }1\le r\le d.
\end{align}
 Especially, set
\begin{align}
N=N_{(d)}=N_1+\cdots+N_d.
\end{align}

For $1\le r\le d$,  set
\begin{align}
\fsl_{\infty,{\bf N}}^{(r)}=\<E_{mN+i,nN+j}\ |\ m,n\in \Z,\ N_{(r-1)}+1\le i\ne j\le N_{(r)}\>,
\end{align}
a Lie subalgebra of $\fsl_{\infty}$.
Then set
\begin{align}
\fsl_{\infty,{\bf N}}=\sum_{r=1}^d \fsl_{\infty,{\bf N}}^{(r)},
\end{align}
a subalgebra of $\fsl_{\infty}$, which actually is a direct sum of the Lie algebras $\fsl_{\infty,{\bf N}}^{(r)}$.
Notice that each subalgebra $\fsl_{\infty,{\bf N}}^{(r)}$ is isomorphic to $\fsl_\infty$
with an isomorphism $\pi_r$ given by
\begin{align}\label{pi-r}
\pi_{r}(E_{mN+i,nN+j})=E_{mN_r+\bar{i},nN_r+\bar{j}}
\end{align}
for $m,n\in \Z, \  N_{(r-1)}+1\le i\ne j\le N_{(r)},$ where $\bar{i}=i-N_{(r-1)},\  \bar{j}=j-N_{(r-1)}$
(so that $1\le \bar{i}\ne \bar{j}\le N_r$).
Then $\fsl_{\infty,{\bf N}}$ is naturally isomorphic to $\fsl_\infty^{\oplus d}$ with an isomorphism
$$\pi=\pi_1\times \pi_2\times \cdots \times \pi_d.$$
Equip $\fsl_{\infty,{\bf N}}$ with the bilinear form inherited from $\fsl_\infty$, and on the other hand,
equip $\fsl_\infty^{\oplus d}$ with the canonical bilinear form associated to the bilinear form on $\fsl_\infty$.
It is straightforward to see that the isomorphism preserves  the invariant bilinear forms.
Consequently, the isomorphism $\pi$ extends uniquely to an affine Lie algebra isomorphism $\hat{\pi}$ from
$\wh{\fsl_{\infty,{\bf N}}}$ to $\wh{\fsl_\infty^{\oplus d}}$, preserving the canonical central elements.
Furthermore, $\hat{\pi}$ gives rise to a vertex algebra isomorphism
\begin{align}\label{pi-va}
\pi_{va}: \  L_{\wh{\fsl_\infty,{\bf N}}}(\ell,0)\longrightarrow L_{\wh{\fsl_\infty^{\oplus d}}}(\ell,0)
\cong L_{\wh{\fsl_\infty}}(\ell,0)^{\ot d}.
\end{align}

Recall from Definition \ref{sigma-def} the automorphism $\sigma$ of $\fgl_{\infty}$, where
\begin{align*}
\sigma(E_{m,n})=E_{m+1,n+1}\   \   \   \   \mbox{ for }m,n\in \Z.
\end{align*}
It can be readily seen that $\sigma^N$ preserves
the subalgebras $\fsl_{\infty,{\bf N}}^{(r)}$ for $1\le r\le d$, hence
the subalgebra $\fsl_{\infty,{\bf N}}$. We then view $\sigma^N$ as an automorphism of the $\Z$-graded vertex algebra
$L_{\wh{\fsl_\infty,{\bf N}}}(\ell,0)$.
Let $R_{N}$ be the representation of $\Z$ on
$L_{\wh{\fsl_\infty,{\bf N}}}(\ell,0)$ given by
\begin{align}
R_N(n)=\sigma^{nN}\chi_{q}(n)^{-L(0)}=\sigma^{nN}q^{-nL(0)}\   \   \   \mbox{ for }n\in \Z.
\end{align}
Then $L_{\wh{\fsl_\infty,{\bf N}}}(\ell,0)$ is a $\Z$-graded vertex $\Z$-algebra.
 Note that $L_{\wh{\fsl_\infty}}(\ell,0)$ contains $L_{\wh{\fsl_{\infty,{\bf N}}}}(\ell,0)$ as a $\Z$-graded vertex subalgebra.

Notice that under the Lie algebra isomorphism $\pi_r$ from $\fsl_{\infty,{\bf N}}^{(r)}$ to $\fsl_\infty$ (see (\ref{pi-r})),
the automorphism $\sigma^N$ of $\fsl_{\infty,{\bf N}}^{(r)}$ corresponds to the automorphism $\sigma^{N_r}$ of
$\fsl_\infty$.

\begin{dfnt}
{\em For $1\le r\le d$, with $n\in \Z$ acting as $\sigma^{nN_r}$ on
the $\Z$-graded vertex algebra $L_{\wh{\fsl_\infty}}(\ell,0)$, we obtain a vertex $\Z$-algebra,
 which we denote by $(L_{\wh{\fsl_\infty}}(\ell,0),\rho_{N_r})$. }
\end{dfnt}

With this, the vertex algebra isomorphism $\pi_{va}$ (see (\ref{pi-va}) is actually
a vertex $\Z$-algebra isomorphism from $(L_{\wh{\fsl_\infty,{\bf N}}}(\ell,0),\rho_N)$
 to the tensor product vertex $\Z$-algebra
\begin{align}
(L_{\wh{\fsl_\infty}}(\ell,0), \rho_{N_1})\ot \cdots \ot (L_{\wh{\fsl_\infty}}(\ell,0),\rho_{N_d})
\end{align}
 (cf. Example \ref{tensorex} and Lemma \ref{lem:charintslmod}).

Our main goal of this section is to determine
 the $L_{\wh{\fsl_\infty}}(\ell,0)\rightarrow L_{\wh{\fsl_{\infty,{\bf N}}}}(\ell,0)$ branching law.
Recall that equivariant quasi $(L_{\wh{\fsl_\infty}}(\ell,0), \rho_{N_r})$-modules
correspond to integrable and restricted $\wh{\fsl_{N_r}}(\C_q)$-modules of level $\ell$ for $1\le r\le d$.
Recall also from Definition \ref{defetamua} that the linear functional $\eta_{\bm\mu,\bm a}$ on $\wh{\CH}$ also
depends on the fixed positive integer $N$. In the following, we shall need such linear functionals for various positive integers $N$.
For this reason, we here shall denote the functional $\eta_{\bm\mu,\bm a}$ by
$\eta^{(N)}_{\bm\mu,\bm a}$, to show its dependence on $N$.

\begin{remt}
{\em Note that with the vertex $\Z$-algebra isomorphism $\pi_{va}$, it follows from Proposition \ref{prop:tensordec} and
Theorem \ref{thm:ferreal} that each irreducible $\N$-graded $(\Z,\chi_q)$-equivariant quasi
$L_{\wh{\fsl_\infty,{\bf N}}}(\ell,0)$-module is of the form
\underline{}\begin{align}\label{overlineNmod}
L(\eta_{{\bm\mu}_1,\bm{a}_1}^{(N_1)})\ot \cdots \ot
L(\eta_{{\bm\mu}_d,\bm{a}_d}^{(N_d)}),
\end{align}
where $\bm{a}_r\in (\C^{\times})_q^{\ell},\  {\bm\mu}_r\in \Z_{++}^{{\bf I}_{\bm{a}_r}}$ for $1\le r\le d$.}
\end{remt}

Let $\bf I$ be a partition of $\{1,\dots,\ell\}$.
Set $\rGL^{d}_{\bf I}=\rGL_{\bf I}\times \cdots\times \rGL_{\bf I}$ ($d$-times) and
view $\rGL_{\bf I}$ as the diagonal subgroup.
Write the $\rGL_{\bf I}^{d}\rightarrow \rGL_{\bf I}$ branching
(tensor product decomposition) as
\begin{align}\label{cximu}
\Res_{\rGL_{\bf I}}^{\rGL_{\bf I}^{d}} L({\bm\mu}_1)\ot
\cdots  \ot L({\bm\mu}_d)= \bigoplus_{\bm\xi\in \Z_{++}^{\bf I}}
C_{{\bm\mu}_1,\dots,{\bm\mu}_d}^{\bm\xi}\, L(\bm\xi),
\end{align}
where ${\bm\mu}_1,\dots,{\bm\mu}_d\in \Z_{++}^{\bf I}$ and
$C_{{\bm\mu}_1,\dots,{\bm\mu}_d}^{\bm\xi}$
denotes the indicated multiplicity.

\begin{remt}\label{special-cases}
{\em Let ${\bf I}=\{S_1,\dots,S_s\}$ be a partition of $\{1,\dots,\ell\}$.
Note that in case $s=\ell$, we have $\rGL_{\bf I}=\rGL_1^{\ell}$, so the tensor product decomposition
above is always multiplicity free. On the other hand, if $\ell=2$ and $s=1$,
 we have $\rGL_{\bf I}=\rGL_2$. It is known that the tensor product decomposition
above with $d=2$ is also always multiplicity free.}
\end{remt}

As the main result of this section, we have:

\begin{thm}\label{thm:bl2}
Let $\ell,d$ be positive integers and let
${\bf N}=(N_1,\dots,N_d)\in \Z_+^d$ with $N_i\ge 2$. Set $N=N_1+\cdots +N_d$.
Then for any  $\bm{a}\in (\C^{\times})_q^\ell, \  {\bm\xi}\in \Z_{++}^{{\bf I}_{\bm a}}$,
$L(\eta_{{\bm\xi},\bm{a}}^{(N)})$ decomposes into
 irreducible equivariant quasi $L_{\wh{\fsl_{\infty,{\bf N}}}}(\ell,0)$-submodules with multiplicities as
\begin{align}\label{Nd-main}
\Res_{L_{\wh{\fsl_{\infty,{\bf N}}}}(\ell,0)}^{L_{\wh{\fsl_\infty}}(\ell,0)} L(\eta_{{\bm\xi},\bm{a}}^{(N)})
= \bigoplus_{{\bm\mu}_1,\dots,{\bm\mu}_d\in \Z_{++}^{{\bf I}_{\bm a}}}
C_{{\bm\mu}_1,\dots,{\bm\mu}_d}^{{\bm\xi}}\,
L(\eta_{{\bm\mu}_1,\bm{a}}^{(N_1)})\ot \cdots \ot
L^d(\eta_{{\bm\mu}_d,\bm{a}}^{(N_d)}),
\end{align}
where the multiplicities $C_{{\bm\mu}^1,\dots,{\bm\mu}_d}^{{\bm\xi}}$ are the same as in (\ref{cximu}).
\end{thm}

\begin{proof}  Recall that $\mathcal F_N^{\bm{a}}$ is an equivariant quasi module
for $(L_{\wh{\fsl_\infty}}(\ell,0),\rho_N)$
with $\mathcal F_N^{\bm{a}}=\mathcal F_N^{\ell}$ as a vector space.
Then we have an equivariant quasi $L_{\wh{\fsl_{\infty,{\bf N}}}}(\ell,0)$-module
$\mathcal F_N^{\bm{a}}$ via the natural embedding of
$L_{\wh{\fsl_{\infty,{\bf N}}}}(\ell,0)$ into $(L_{\wh{\fsl_\infty}}(\ell,0),\rho_N)$.
On the other hand, we have an equivariant quasi $L_{\wh{\fsl_{\infty,{\bf N}}}}(\ell,0)$-module
$\mathcal F_{N_1}^{\bm{a}}\ot \cdots\ot \mathcal F_{N_d}^{\bm{a}}$ via the vertex algebra isomorphism
$$\pi_{va}: \
(L_{\wh{\fsl_{\infty,{\bf N}}}}(\ell,0),\rho_N)\cong
(L_{\wh{\fsl_\infty}}(\ell,0),\rho_{N_1})\ot \cdots \ot (L_{\wh{\fsl_\infty}}(\ell,0),\rho_{N_d}).$$
It is straightforward to show that there exists a (Clifford) algebra isomorphism
\begin{align}
\theta_{\bf N}^{\ell}:\  \  {\mathcal C}_{N_1}^{\ell}\ot \cdots \ot  {\mathcal C}_{N_d}^{\ell}
\longrightarrow {\mathcal C}_{N}^{\ell}
\end{align}
such that
$$\theta_{\bf N}^{\ell}(\psi_{i}^{p}(n))=\psi_{i+N_{(r-1)}}^{p}(n),\   \   \   \
\theta_{\bf N}^{\ell}(\bar{\psi}_{i}^{p}(n))=\bar{\psi}_{i+N_{(r-1)}}^{p}(n) $$
for $1\le r\le d,\  1\le i\le N_r,\ 1\le p\le \ell,\ n\in \Z$. Furthermore, this gives rise to
a Fock space identification
\begin{align}
\Theta_{\bf N}^{\ell}:\   \
 \mathcal F_{N_1}^{\ell}\ot \cdots\ot \mathcal F_{N_d}^{\ell}\longrightarrow \mathcal F_N^{\ell}.
\end{align}
It is easy to see that  $\Theta_{\bf N}^{\ell}$ is an  isomorphism of quasi $L_{\wh{\fsl_{\infty,{\bf N}}}}(\ell,0)$-modules
from $\mathcal F_{N_1}^{\bm{a}}\ot \cdots\ot \mathcal F_{N_d}^{\bm{a}}$
to $\mathcal F_N^{\ell}$.
It can be readily seen that  $\Theta_{\bf N}^{\ell}$ is also a $\rGL_{{\bf I}_{\bm a}}$-module isomorphism.
Now, by Theorem \ref{thm:main4}, we get a  seesaw pair
\begin{align*}
((L_{\wh{\fsl_\infty}}(\ell,0),\rho_{N_1})\ot \cdots \ot (L_{\wh{\fsl_\infty}}(\ell,0),\rho_{N_d}), \rGL_{{\bf I}_{\bm a}}^{d})
\quad\te{and}\quad ((L_{\wh{\fsl_\infty}}(\ell,0),\rho_N), \rGL_{{\bf I}_{\bm a}})
\end{align*}
 on $\mathcal F_N^{\ell}$ with the embeddings
\begin{align*}
(L_{\wh{\fsl_\infty}}(\ell,0),\rho_{N_1})\ot \cdots \ot (L_{\wh{\fsl_\infty}}(\ell,0),\rho_{N_d})\cong
(L_{\wh{\fsl_{\infty,{\bf N}}}}(\ell,0),\rho_N)
\hookrightarrow (L_{\wh{\fsl_\infty}}(\ell,0),\rho_N)
\end{align*}
and $\rGL_{{\bf I}_{\bm a}}\hookrightarrow \rGL_{{\bf I}_{\bm a}}^{d}$.
Using the first dual pair and \eqref{cximu}, we obtain the following
  $(L_{\wh{\fsl_{\infty,{\bf N}}}}(\ell,0),\rGL_{{\bf I}_{\bm a}})$-module decomposition
\begin{align*}
&\mathcal F_{N_1}^{\bm{a}}\ot \cdots\ot \mathcal F_{N_d}^{\bm{a}}\\
=\,&\(\oplus_{{\bm\mu}_1\in \Z_{++}^{{\bf I}_{\bm a}}}
L(\eta_{{\bm\mu}_1,\bm{a}}^{(N_1)})\ot L_{\rGL_{{\bf I}_{\bm a}}}({\bm\mu}_1)\)
\ot \cdots \ot  \(\oplus_{{\bm\mu}_d\in \Z_{++}^{{\bf I}_{\bm a}}}
L(\eta_{{\bm\mu}_d,\bm{a}}^{(N_d)})\ot L_{\rGL_{{\bf I}_{\bm a}}}({\bm\mu}_d)\)\\
\cong\,& \oplus_{{\bm\mu}^1,\dots,{\bm\mu}^d\in \Z_{++}^{{\bf I}_{\bm a}}}
\(L(\eta_{{\bm\mu}_1,\bm{a}}^{(N_1)})\ot \cdots \ot
L(\eta_{{\bm\mu}_d,\bm{a}}^{(N_d)})\)\ot \(L_{\rGL_{{\bf I}_{\bm a}}}({\bm\mu}_1)\ot \cdots \ot
L_{\rGL_{{\bf I}_{\bm a}}}({\bm\mu}_d)\)\\
=\,& \oplus_{{\bm\xi},{\bm\mu}_1,\dots,{\bm\mu}_d\in \Z_{++}^{{\bf I}_{\bm a}}}
C_{{\bm\mu}_1,\dots,{\bm\mu}_d}^{{\bm\xi}}\,
(L(\eta_{{\bm\mu}_1,\bm{a}}^{(N_1)})\ot \cdots \ot
L^d(\eta_{{\bm\mu}_d,\bm{a}}^{(N_d)}))\ot L_{\rGL_{{\bf I}_{\bm a}}}({\bm\xi}).
\end{align*}
Combining this with  the decomposition for the dual pair
$((L_{\wh{\fsl_\infty}}(\ell,0),\rho_N), \rGL_{{\bf I}_{\bm a}})$ on  $\mathcal F_N^{\ell}$
(see \eqref{decomfnella}) we obtain (\ref{Nd-main}).
\end{proof}

As immediate consequences of Theorem \ref{thm:bl2} and Remark \ref{special-cases} we have:

\begin{cort}
The branching from $(L_{\wh{\fsl_\infty}}(1,0),\rho_N)$ to  $L_{\wh{\fsl_{\infty,{\bf N}}}}(1,0)$
in Theorem \ref{thm:bl2} on equivaraint quasi modules is multiplicity free.
\end{cort}

\begin{cort}
Assume $\ell=2$ and ${\bf N}=(N_1,N_2)$ (with $d=2$). Then the branching from
$L_{\wh{\fsl_\infty}}(2,0)$ to  $L_{\wh{\fsl_{\infty,{\bf N}}}}(2,0)$ in Theorem \ref{thm:bl2}
is multiplicity free.
\end{cort}

\begin{remt}
{\em Here, we formulate a Lie algebra analogue of Theorem \ref{thm:bl2}.
For $1\le r\le d$, let $\wh{\fsl_{{\bf N}}}^{(r)}(\C_q)$
denote the subalgebra of $\wh{\fsl_{N}}(\C_q)$, generated by the elements
\[ E_{i,j}t_0^{m_0}t_1^{m_1}\quad \te{for}\ m_0,m_1\in \Z,\  N_{(r-1)}+1\le i\ne j\le N_{(r)}.\]
Set
$$\wh{\fsl_{{\bf N}}}(\C_q)=\sum_{r=1}^{d}{\wh{\fsl_{{\bf N}}}}^{(r)}(\C_q),$$
which is a subalgebra of $\wh{\fsl_N}(\C_q)$.
It is easy to see that $\wh{\fsl_{{\bf N}}}^{(r)}(\C_q)$ is isomorphic to $\wh{\fsl_{N_r}}(\C_q)$ (as $N_r=N_{(r)}-N_{(r-1)}$)
and for $1\le r\ne s\le d$,
$$[\wh{\fsl_{{\bf N}}}^{(r)}(\C_q),\wh{\fsl_{{\bf N}}}^{(s)}(\C_q)]=0.$$
Then it follows that every irreducible $\N$-graded integrable restricted $\wh{\fsl_{{\bf N}}}(\C_q)$-module has the form of \eqref{overlineNmod}.
Under the setting of Theorem \ref{thm:bl2}, we have the $\wh{\fsl_N}(\C_q)\rightarrow \wh{\fsl_{{\bf N}}}(\C_q)$ branching law:
\begin{align*}
\Res_{\wh{\fsl_{{\bf N}}}(\C_q)}^{\wh{\fsl_N}(\C_q)}\, L(\eta_{{\bm\xi},\bm{a}}^{(N)})
= \bigoplus_{{\bm\mu}_1,\dots,{\bm\mu}_d\in \Z_{++}^{{\bf I}_{\bm a}}}
C_{{\bm\mu}_1,\dots,{\bm\mu}_d}^{{\bm\xi}}\,
L(\eta_{{\bm\mu}_1,\bm{a}}^{(N_1)})\ot \cdots \ot
L(\eta_{{\bm\mu}_d,\bm{a}}^{(N_d)}).
\end{align*}}
\end{remt}

\section{$\wh{\fsl_N}(\C_q)\rightarrow \wh{\fsl_N}(\C_q[t_0^{\pm M_0},t_1^{\pm M_1}])$ branching law}
Let $M_0$ and $M_1$ be positive integers, which are fixed throughout this section.
 Denote by $\C_q[t_0^{\pm M_0},t_1^{\pm M_1}]$
the subalgebra of $\C_q=\C_q[t_0^{\pm 1},t_1^{\pm 1}]$, generated by $t_0^{\pm M_0}$ and $t_1^{\pm M_1}$.
Set
\begin{align}
\wh {\fsl_N}(\C_q[t_0^{\pm M_0},t_1^{\pm M_1}])=\fsl_N(\C_q[t_0^{\pm M_0},t_1^{\pm M_1}])\oplus \C\bm\rk_0\oplus \C\bm\rk_1\subset \wh{\fsl_N}(\C_q).
\end{align}
Note that $\wh {\fsl_N}(\C_q[t_0^{\pm M_0},t_1^{\pm M_1}])$ is simply the subalgebra of $\wh{\fsl_N}(\C_q)$ generated by
$$E_{i,j}t_0^{m_0M_0}t_1^{m_1M_1}\  \  \mbox{ for }1\le i\ne j\le N, \  m_0,m_1\in \Z.$$
The main goal of this section is to determine the
$\wh{\fsl_N}(\C_q)\rightarrow \wh{\fsl_N}(\C_q[t_0^{\pm M_0},t_1^{\pm M_1}])$
branching law.

First of all, it is straightforward to see that as an algebra $\C_q[t_0^{\pm M_0},t_1^{\pm M_1}]$
is isomorphic to $\C_{q^{M_0M_1}}[t_0^{\pm 1},t_1^{\pm1}]$ with
$t_0^{n_0M_0}t_1^{n_1M_1}$ corresponding to $t_0^{n_0}t_1^{n_1}$ for $n_0,n_1\in \Z$.
Furthermore, we have the following straightforward result:

\begin{lemt}
 There is a Lie algebra isomorphism
$$\pi_{M_0,M_1}:\  \wh{\fsl_N}(\C_q[t_0^{\pm M_0},t_1^{\pm M_1}]) \rightarrow \wh{\fsl_N}(\C_{q^{M_0M_1}}),$$
which is uniquely determined by the assignment
\begin{align}\label{m0m1iso}
E_{i,j}t_0^{m_0M_0}t_1^{m_1M_1}\mapsto E_{i,j}t_0^{m_0}t_1^{m_1},\quad M_0\bm\rk_0\mapsto \bm\rk_0,
\quad M_1\bm\rk_1\mapsto \bm\rk_1
\end{align}
for $1\le i\ne j\le N$, $m_0,m_1\in \Z$.
\end{lemt}

%
%

Let ${\bf I}=\{S_1,\dots,S_s\}$  be a partition of $\{1,2,\dots,\ell\}$. Set
\begin{align}
{\bf I}^{M_0}=\{  k\ell +S_i\  |\  0\le k\le M_0-1,\  1\le i\le s\},
\end{align}
where by definition $k\ell+S_i=\{ k\ell+p\ |\ p\in S_i\}$.
Then ${\bf I}^{M_0}$ is a partition of $\{1,2,\dots,M_0\ell\}$.
For $\bm{a}=(a_1,\dots,a_\ell)\in (\C^{\times})^\ell$, we define $({\bm{a}}q)_{M_0}^{M_1}\in (\C^{\times})^{M_0\ell}$ by
\begin{align}\label{def-aqM}
\(({\bm{a}}q)_{M_0}^{M_1} \)_{k\ell+r}=(a_rq^{-k})^{M_1}
\end{align}
for $0\le k\le M_0-1,\ 1\le r\le \ell$. Then we have:

\begin{lemt}
Let $M_0,M_1$ and $\ell$ be positive integers, let
$\bm{a}\in (\C^{\times})^\ell_q$ and let $\bm\mu\in \Z_{++}^{{\bf I}_{\bm a}}$.
Then there exist  $\bm b\in (\C^{\times})^\ell_q$ and $\bm\nu\in \Z_{++}^{{\bf I}_{\bm b}}$ such that
\begin{align*}
\eta_{\bm\mu,\bm a}=\eta_{\bm\nu,\bm b},\quad ({\bm{b}}q)_{M_0}^{M_1}\in (\C^{\times})^{M_0\ell}_{q^{M_0M_1}}\quad \te{and}\quad
{\bf I}_{({\bm{b}}q)_{M_0}^{M_1}}={\bf I}_{\bm b}^{M_0}.
\end{align*}
\end{lemt}

\begin{proof} Define an equivalence relation $\sim$ on $I:=\{1,2,\dots,\ell\}$ by claiming $i\sim j$ if and only if
$a_i^{M_1}=q^{nM_1}a_j^{M_1}$ for some $n\in \Z$. Let $\{ I_1,\dots,I_s\}$ be the partition of $I$,
associated to this equivalence relation.
For $1\le i\le \ell$, set $A_i=\{ a_i\ |\ i\in I_i\}$ and choose a representative $c_i$ from $A_i$.
Then for each $j\in I$ there exist unique $1\le i\le s$ and $n_j\in \Z$ such that
$q^{n_jM_1}a_j^{M_1}=c_i^{M_1}$. Set $\bm{n}=(n_1,\dots,n_\ell)\in \Z^\ell$ and
$$\bm b=q^{\bm n}\bm{a}=(q^{n_1}a_{1},\dots,q^{n_\ell}a_{\ell})\in (\C^{\times})^{\ell},$$
$$\bm\nu=\bm\mu+\bm n N=(\mu_1+n_1N,\dots,\mu_\ell+n_\ell N)\in \Z^{\ell}.$$
It is straightforward to see that $\bm b\in (\C^\times)_q^\ell$ (with ${\bf I}_{\bm b}={\bf I}_{\bm a}$) and $\bm \nu\in \Z_{++}^{{\bf I}_{\bm b}}$.
Furthermore, for $1\le i,j\le \ell$, we have either $b_i^{M_1}=b_j^{M_1}$  or
$b_i^{M_1}\ne q^{nM_1}b_j^{M_1}$ for any $n\in \Z$.
This implies that $({\bm{b}}q)_{M_0}^{M_1}$ lies in $(\C^{\times})^{M_0\ell}_{q^{M_1}}$ (and hence in
$(\C^{\times})^{M_0\ell}_{q^{M_0M_1}}$)
 and
${\bf I}_{({\bm{b}}q)_{M_0}^{M_1}}={\bf I}_{\bm b}^{M_0}$. Finally,  it follows from Remark \ref{connection-12} that $\eta_{\bm\mu,\bm a}=\eta_{\bm\nu,\bm b}$,
as desired.
\end{proof}

Let $\bf I$ be a partition of $\{1,\dots,\ell\}$. From definition, we have
$$\rGL_{{\bf I}^{M_0}}\cong \rGL_{\bf I}\times \cdots \times \rGL_{\bf I}=(\rGL_{\bf I})^{M_0}.$$
View  $\rGL_{\bf I}$ as a subgroup of $\rGL_{{\bf I}^{M_0}}$ through the diagonal embedding.
Write the $\rGL_{{\bf I}^{M_0}}\rightarrow \rGL_{\bf I}$ branching as
\begin{align}
\Res_{\rGL_{\bf I}}^{\rGL_{{\bf I}^{M_0}}}\, L({\bm\xi})
=\bigoplus_{{\bm\mu}\in \Z_{++}^{\bf I}} E_{{\bm\mu}}^{{\bm\xi}}\,
L({\bm\mu})
\end{align}
for  ${\bm\xi}\in \Z_{++}^{{\bf I}^{M_0}}$, where $E_{{\bm\mu}}^{{\bm\xi}}$
denotes the multiplicity  of the $\rGL_{\bf I}$-module
$L({\bm\mu})$ in the $\rGL_{{\bf I}^{M_0}}$-module $L({\bm\xi})$.


%

The following is the main result of this section:

\begin{thm}\label{thm:bl3}
Let $\ell,M_0,M_1$ be positive integers, let $\bm a\in (\C^\times)_q^\ell$
 such that $(\bm{a}q)_{M_0}^{M_1}\in (\C^\times)^{M_0\ell}_{q^{M_0M_1}}$  and
${\bf I}_{(\bm{a}q)_{M_0}^{M_1}}={\bf I}_{\bm a}^{M_0}$,  and let $\bm{\mu}\in \Z_{++}^{{\bf I}_{\bm a}}$.
Then
the $\wh{\fsl_N}(\C_q)$-module $L(\eta_{\bm\mu,\bm{a}})$ viewed as an
$\wh{\fsl_N}(\C_q[t_0^{\pm M_0},t_1^{\pm M_1}])$-module 
decomposes into irreducible submodules  with multiplicities as
\begin{align}\label{M0M1-branching}
L(\eta_{\bm\mu,\bm{a}})
= \bigoplus_{\bm\xi\in \Z_{++}^{{{\bf I}_{\bm a}}^{M_0}}}
E_{\bm\mu}^{\bm\xi}\,
L(\eta_{\bm\xi,(\bm{a} q)_{M_0}^{M_1}}).
\end{align}
\end{thm}

\begin{proof} Recall that we have Clifford algebras $\mathcal C_N^{M_0\ell}$ and $\mathcal C_N^\ell$.
It is straightforward to show that the assignment
\begin{align}\label{identifylast}
  \psi_i^{k\ell+p}(n) \mapsto  \psi_i^{p}(M_0n-k),\quad  \quad
 \bar\psi_i^{k\ell+p}(n)\mapsto \bar\psi_i^{p}(M_0n+k)
\end{align}
for $1\le i\le N$, $0\le k\le M_0-1$, $1\le p\le \ell$, $n\in \Z$ extends uniquely to
an algebra isomorphism from $\mathcal C_N^{M_0\ell}$ to $\mathcal C_N^\ell$, denoted by $\phi$.
(Note that every integer $1\le p'\le M_0\ell$ can be written uniquely as $p'=k\ell+p$ with
$1\le p\le \ell,\ 0\le k\le M_0-1$.)
We have
\begin{eqnarray}\label{isonormal}
\phi\left(:\psi_i^{k\ell+p}(n)\bar\psi_{i'}^{k'\ell+p'}(n'):\right)=:\psi_i^{p}(M_0n-k)\bar\psi_{i'}^{p'}(M_0n'+k'):
\end{eqnarray}
for $1\le i,i'\le N$, $1\le p,p'\le \ell$, $0\le k,k'\le M_0-1$ and $n,n'\in \Z$,
recalling Remark \ref{normal-ordering-2} and noticing that $M_0n'+k'\ge 0$ if and only if $n'\ge 0.$
Via this isomorphism $\phi$, the $\mathcal C_N^{\ell}$-module  $\mathcal F_N^\ell$ becomes a $\mathcal C_N^{M_0\ell}$-module,
which is denoted by $(\mathcal F_N^\ell)^{\phi}$. Notice that for $0\le k\le M_0-1,\ n\in \Z$,
$M_0n-k>0$ if and only if $n>0$, and $M_0n+k\ge 0$ if and only if $n\ge 0$.
It then follows from the construction of $\mathcal F_N^{M_0\ell}$ that
$(\mathcal F_N^\ell)^{\phi}\simeq \mathcal F_N^{M_0\ell}$ as a $\mathcal C_N^{M_0\ell}$-module.
That is, we have a linear isomorphism $\Phi: \    \mathcal F_N^{M_0\ell}\rightarrow \mathcal F_N^\ell $
such that $\Phi(|0\>)=|0\>$ and
\begin{eqnarray*}
\Phi(\psi_i^{k\ell+p}(n)w)=\psi_i^{p}(M_0n-k)\bar{\phi}(w),\   \   \   \
\Phi(\bar{\psi}_i^{k\ell+p}(n)w)=\bar{\psi}_i^{p}(M_0n+k)\bar{\phi}(w)
\end{eqnarray*}
for $1\le i\le N,\ 1\le p\le \ell,\  0\le k\le M_0-1,\ n\in \Z,\ w\in \mathcal F_N^\ell$.

Now, with $\bm{a}=(a_1,\dots,a_\ell)\in (\C^\times)^\ell$,
we have an $\wh{\fsl_N}(\C_q)$-module $\mathcal F_N^{\bm{a}}$, which is naturally
 an $\wh{\fsl_N}(\C_q[t_0^{\pm M_0},t_1^{\pm M_1}])$-module.
For $1\le i\ne j\le N$, $m_0,m_1\in \Z$, recall from \eqref{wteij} that
$E_{i,j}t_0^{M_0m_0}t_1^{M_1m_1}$ (in $\wh{\fsl_N}(\C_q[t_0^{\pm M_0},t_1^{\pm M_1}])$)
 acts on $\mathcal F_N^{\bm{a}}$ as
\begin{align*}\label{wteijM}
&\sum_{n\in \Z}\sum_{r=1}^{\ell} (a_rq^{-n})^{M_1m_1}:\psi_i^r(m_0M_0-n)\bar\psi_j^r(n):\\
=&\sum_{n\in \Z}\sum_{r=1}^{\ell}\sum_{k=0}^{M_0-1} \left((a_rq^{-k})^{M_1}q^{-nM_0M_1}\right)^{m_1}
:\psi_i^r((m_0-n)M_0-k)\bar\psi_j^r(nM_0+k):.
\end{align*}

On the other hand,  via the Lie algebra isomorphism $\pi_{M_0,M_1}$,
the $\wh{\fsl_N}(\C_{q^{M_0M_1}})$-module $\mathcal F_N^{(\bm{a}q)_{M_0}^{M_1}}$ $(=\mathcal F_{N}^{M_0\ell})$
becomes an $\wh{\fsl_N}(\C_q[t_0^{\pm M_0},t_1^{\pm M_1}])$-module, where
the element $E_{i,j}t_0^{M_0m_0}t_1^{M_1m_1}$  acts as
\begin{eqnarray*}\label{teijaction}
\quad\sum_{n\in \Z}\sum_{k=0}^{M_0-1}\sum_{r=1}^{\ell} \left((a_rq^{-k})^{M_1}q^{-nM_0M_1}\right)^{m_1}
:\psi_i^{k\ell+r}(m_0-n)\bar\psi_j^{k\ell+r}(n):.
\end{eqnarray*}
It then follows from \eqref{isonormal} that the Fock space identification
\[\Psi:\  \mathcal F_N^{(\bm{a}q)_{M_0}^{M_1}}=\mathcal F_N^{M_0\ell}
\simeq^{\Phi}\mathcal F_N^\ell=\mathcal F_N^{\bm{a}}\]
is an $\wh{\fsl_N}(\C_q[t_0^{\pm M_0},t_1^{\pm M_1}])$-module isomorphism.

Next, we show that $\Psi$ is also a $\rGL_{{\bf I}_{\bm a}}$-module isomorphism.
Recall that the action of $\rGL_{{\bf I}_{\bm a}}$ on $\mathcal F_N^{\ell}$
 is given by the action of the Lie algebra $\fgl_{{\bf I}_{\bm a}}$ from (\ref{laction}), where
 $$\fgl_{{\bf I}_{\bm a}}={\rm span}\{ E_{p,p'}\ |\  1\le p,p'\le \ell,\ p\sim_{\bm a} p'\}\subset \fgl_{\ell}.$$
Similarly, the action of $\rGL_{{\bf I}^{M_0}}$ on $\mathcal F_N^{M_0\ell}$ is given by the action of
the Lie algebra $\fgl_{{{\bf I}_{\bm a}}^{M_0}}$ $(\subset \fgl_{M_0\ell}$).
Furthermore, $\fgl_{{\bf I}_{\bm a}}$ acts on $\mathcal F_N^{M_0\ell}$ through the embedding of $\fgl_{{\bf I}_{\bm a}}$ into $\fgl_{{{\bf I}_{\bm a}}^{M_0}}$
with
 $$E_{p,p'}\mapsto \sum_{k=0}^{M_0-1}E_{k\ell+p,k\ell+p'}$$
 for $1\le p,p'\le \ell$ with $p\sim_{\bm a} p'$.
We see that the element $E_{p,p'}$ of the Lie algebra $\fgl_{{\bf I}_{\bm a}}$
acts on $\mathcal F_N^{M_0\ell}$ as
$$\sum_{k=0}^{M_0-1}\sum_{i=1}^{N}:\psi_{i}^{k\ell+p}(-n)\bar{\psi}_{i}^{k\ell+p'}(n):.$$
Then it follows  from \eqref{isonormal} that $\Psi$ is a $\fgl_{{\bf I}_{\bm a}}$-module isomorphism.
Thus
$\Psi$ is also a $\rGL_{{\bf I}_{\bm a}}$-module isomorphism.

Finally, by using $\Psi$ we transport the  $\rGL_{{\bf I}_{\bm a}^{M_0}}$-module structure from
$\mathcal F_{N}^{(\bm{a}q)_{M_0}^{M_1}}$ to $\mathcal F_N^{\bm a}$, obtaining
 a seesaw pair
\begin{align*}
(\wh{\fsl_N}(\C_q),\rGL_{{\bf I}_{\bm a}})\quad\te{and}\quad
(\wh{\fsl_N}(\C_{q^{M_0M_1}}),\rGL_{{\bf I}_{\bm a}^{M_0}})
\end{align*}
on the Fock space $\mathcal F_N^{\bm a}$.  Then just as in the proof of Theorem \ref{thm:bl2},
(\ref{M0M1-branching}) follows from this seesaw pair.
\end{proof}


\begin{thebibliography}{AAGBP}

\bibitem[AABGP] {AABGP}
B. Allison, S. Azam, S. Berman, Y. Gao, and A. Pianzola,
Extended Affine Lie Algebras and Their Root Systems, {\em  Memoirs Amer. Math. Soc.}
{\bf 126},1997. 

\bibitem[ABFP]{ABFP}
B. Allison, S. Berman, J. Faulkner and A. Pianzola,  Multiloop realization of extended
affine Lie algebras and Lie tori, {\em Trans. Amer. Math. Soc.} {\bf361} (2009), 4807-4842.

\bibitem[ABGP]{ABGP}
B. Allison, S. Berman, Y. Gao, and A. Pianzola, {\em A characterization of affine
Kac-Moody Lie algebras}, {\em Commun. Math. Phys.} {\bf 185} (1997),  671-688.

\bibitem[BGK] {BGK}
S. Berman, Y. Gao, and Y. Krylyuk, Quantum tori and the
structure of elliptic quasi-simple Lie algebras, {\em J. Funct. Anal.} {\bf 135} (1996) 339-389.

\bibitem[BGT] {BGT}
S. Berman, Y. Gao, and S. Tan, A unified view of some vertex operator
constructions, {\em Israel J. Math.} {\bf 134} (2003) 29-60.

\bibitem[BBS] {BBS}
S. Berman, Y. Billig and J. Szmigielski, Vertex operator algebras and the representation theory of toroidal algebras,
in {\em Recent Developments in Infinite-Dimensional Lie Algebras and Conformal Field Theory (Charlottesville, VA, 2000),}
Contemporary Math. {\bf 297}, Amer. Math. Soc., Providence, 2002, 1-26.

\bibitem[BS] {BS}
S. Berman and J. Szmigielski, Principal realization for extended affine
Lie algebra of type $\fsl_2$ with coordinates in a simple quantum torus with
two variables, Contemporary Math., 248, (1999), 39-67.

\bibitem[B1]{B1} Y. Billig,  A category of modules for the full toroidal Lie algebra, {\em Int. Math. Res. Not.} (2006), Art. ID 68395, 46 pp.

\bibitem[B2]{B2}
Y. Billig, Representations of toroidal extended affine Lie algebras, {\em J. Algebra} {\bf 308} (2007) 252-269.

\bibitem[BL]{BL}
Y. Billig and M. Lau, Irreducible modules for extended affine Lie algebras, {\em J. Algebra} {\bf 327} (2011) 208-235.

\bibitem[BZ]{BZ}
Y. Billig and K. Zhao,  Weight modules over exp-polynomial Lie algebra,
{\em J. Pure Appl. Algebra} {\bf 191} (2004) 23-42.

\bibitem[CG]{CG}
V. Chari and J. Greenstein,  Graded level zero integrable representations of affine
Lie algebras, {\em Tran. Amer. Math. Soc.}
Vol. {\bf 360} (2008) 2923-2940.

\bibitem[CP]{CP}
V. Chari and A. Pressley,  Weyl modules for classical and quantum affine algebras,
{\em Representation Theory} {\bf  5} (2001) 191-223.

\bibitem[CLT]{CLT}
F. Chen, Z. Li and S. Tan, Integrable representations for toroidal extended affine Lie algebras,
{\em J. Algebra}  {\bf 519} (2019) 228-252.

\bibitem[CT]{CT}
F. Chen and S. Tan, Integrable representations for extended affine Lie algebras
coordinated by quantum tori, {\em J. Lie Theory} {\bf 23 (2)} (2013) 383-405.

\bibitem [DJKM]{DJKM}
E. Date, M. Jimbo, M. Kashiwara and T. Miwa, Operator approach to the Kadomtsev-Petviashvili equation.
Transformation groups for soliton equations III, {\em J. Phys. Soc. Japan} {\bf 50} (1981) 3806-3812.



\bibitem[DLM]{DLM}
C. Dong, H.-S. Li and G. Mason, Regularity of rational vertex
operator algebras, {\em Advances in Math.} {\bf 132} (1997) 148-166.

\bibitem[ER] {ER}
S. Eswara Rao, A class of integrable modules for the core of EALA coordinatized by quantum
tori, {\em J. Algebra} {\bf 275} (2004) 59-74.

\bibitem[FF]{FF}
A. Feingold, I. Frenkel, Classical affine algebras, {\em Advances in Math.} {\bf 56} (1985) 117-172.


\bibitem[FKRW]{FKRW}
E. Frenkel, V. Kac, A. Radul, and W. Wang, $W_{1+\infty}$ and $W(gl_N)$ with central charge $N$,
{\em Commun. Math. Phys.} {\bf 170} (1995) 337-357.




\bibitem[F]{F}
I. B. Frenkel, Representations of affine Lie algebras, Hecke modular forms and Korteweg-de Vires type
equations, Lecture Notes in Mathematics, Vol. 933, Springer-Verlag, Berlin, (1982), 71-110.



\bibitem[FHL]{FHL}
I. B. Frenkel, Y. Huang and J. Lepowsky, On Axiomatic Approaches to Vertex Operator Algebras and
Modules, {\em  Memoirs Amer. Math. Soc.} {\bf 104}, 1993.

\bibitem[FLM]{FLM}
I. B. Frenkel, J. Lepowsky and A. Meurman, {\em Vertex Operator
Algebras and the Monster,} Pure and Applied Math., Vol. 134,
Academic Press, Boston, 1988.

\bibitem[FZ]{FZ}
I. B. Frenkel and Y.-C. Zhu,  Vertex operator algebras associated to
representations of affine and Virasoro algebras, {\em Duke Math. J.}
{\bf 66} (1992) 123-168.

\bibitem[G1] {G1}
Y. Gao, Representations of extended affine Lie algebras coordinatized
by certain quantum tori, {\em Compositio Math.} {\bf 123} (2000) 1-25.

\bibitem[G2] {G2}
Y. Gao, Vertex operators arising from the homogeneous realization for
$\hat{gl}_N$, {\em Commun. Math. Phys.} {\bf 211} (2000) 745-777.

\bibitem[G3] {G3}
Y. Gao, Fermionic and bosonic representations of the extended affine
Lie algebra $\widetilde{gl_N(\mathbb{C}_q)}$, {\em Canada Math. Bull.} {\bf 45} (2002) 623-633.

\bibitem[GZ] {GZ}
Y. Gao and Z. Zeng, Hermitian representations of the extended affine
Lie algebra $\widetilde{gl_2(\mathbb{C}_q)}$, {\em Advances in Math.} {\bf 207} (2006) 244-265.

\bibitem[G-KK]{gkk}
M. Golenishcheva-Kutuzova and V. Kac, $\Gamma$-conformal algebras,
{\em J. Math. Phys.} {\bf 39} (1998) 2290-2305.

\bibitem[GW]{GW} R. Goodman and R. Wallach,
{\em Symmetry, Representations, and Invariants,} Graduate Texts in Mathematics, 255.




\bibitem[H-KT] {H-KT}
R. H{\o}egh-Krohn and B. Torresani, Classification and construction of
quasisimple Lie algebras, {\em J. Funct. Anal.}  {\bf 89} (1990) 106-136.

\bibitem[H1]{H1}
R. Howe, Remarks on classical invariant theory, {\em Trans. Amer. Math. Soc.} {\bf 313} (1989) 539-570.

\bibitem[H2]{H2}
R. Howe, Perspectives on invariant theory: Schur duality, multiplicity-free actions and beyond,
Israel Math. Conf. Proc., 8 (1992), 1-182.


\bibitem[JK]{JK}
H. Jakobsen and V. Kac, A new class of unitarizable highest weight representations of infinite dimensional
Lie algebras, in: Nonlinear Equations in Classical and Quantum Field Theory, Meudon/Paris,
1983/1984, Springer, Berlin, (1985), 1-20.

\bibitem[JLin]{JLin}
C. Jiang and Z. Lin, Tensor Decomposition, Parafermions, Level-Rank Duality and
Reciprocity Law for vertex operator algebras, arXiv:1406.4191.

\bibitem[JLam]{JLam}
C. Jiang and C.-H. Lam, Level-Rank Duality for Vertex Operator Algebras of types
B and D, {\em Bull. Inst. Math. Acad. Sin.(N.S.)} {\bf 14} (2019) 31-54.

\bibitem [K]{Kac}
V. G. Kac, {\it Infinite-dimensional Lie Algebras}, 3rd ed.,
Cambridge Univ. Press, Cambridge, 1990.

\bibitem[KMPX]{KMPX}
 V. G. Kac, P. M$\ddot{o}$seneder Frajria, P. Papi, F. Xu,
 Conformal embeddings and simple current extensions, {\em Int. Mat. Res. Not.,}  {\bf 14} (2015), 5229-5288.


\bibitem [KR]{kr1}
V. G. Kac and A. Radul, Quasifinite highest weight modules over the
Lie algebra of differential operators on the circle,  {\em Commun.
Math. Phys.} {\bf 157} (1993) 429-457.

\bibitem[Ku]{Ku}
S. Kulda, Seesaw dual reductive pairs, in ``Automorphic forms
of several variables", Proceedings of the 1983 Taniguchi Symposium (I. Satake and Y. Morita. eds.),
Birkhauser, Basel, (1984) 244-268.

\bibitem[LL]{LL}
J. Lepowsky and H.-S. Li, {\em Introduction to Vertex Operator
Algebras and Their Representations}, Progress in Math. {\bf 227},
Birkh\"auser, Boston, 2004.

\bibitem[LP]{LP}
J. Lepowsky and M. Primc, {\em Structure of the Standard Modules for
the Affine Lie Algebra $A_{1}^{(1)}$,} Contemporary Math. {\bf 46},
Amer. Math. Soc., Providence, 1985.


\bibitem[Li1]{li-local}
H.-S. Li,  Local systems of vertex operators, vertex superalgebras
and modules,  {\em J. Pure Appl. Algebra} {\bf 109} (1996) 143-195.

\bibitem[Li2]{li-twisted}
H.-S. Li,  Local systems of twisted vertex operators, vertex superalgebras
and twisted modules, Contemporary Math. {\bf 193},
Amer. Math. Soc., Providence, 1996, 203-236.

\bibitem[Li3]{li-gamma}
H.-S. Li, A new construction of vertex algebras and quasi modules
for vertex algebras, {\em Advances in Math.} {\bf 202} (2006) 232-286.

\bibitem[Li4]{li-tlie}
H.-S. Li, On certain generalizations of twisted affine Lie algebras
and quasi modules for $\Gamma$-vertex algebras, {\em J. Pure Appl.
Algebra} {\bf 209} (2007) 853-871.

\bibitem[Li5]{li-twisted-quasi}
H.-S. Li, Twisted modules and quasi-modules for vertex operator
algebras, Contemporary Math. {\bf 422}, Amer. Math. Soc.,
Providence, 2007, 389-400.

\bibitem[LTW]{ltw-tri}
H.-S. Li, S. Tan and Q. Wang, Trigonometric Lie algebras, affine Lie algebras, and vertex algebras,
{\em Advances in  Math.}  {\bf 363} (2020) 106985.

\bibitem[N]{N}
E. Neher, Extended affine Lie algebras, {\em C. R. Math. Acad. Sci. Soc. R. Can.}  {\bf 26} (2004)  90-96.



\bibitem[VV1]{VV1}
M. Varagnolo and E. Vasserot, Schur duality in the toroidal setting,
{\em Commun. Math. Phys.} {\bf 182} (1996) 469-484.

\bibitem[VV2]{VV2}
M.~Varagnolo and E.~Vasserot, Double-loop algebras and the {F}ock space,
 {\em Invent. Math.} \textbf{133} (1998) 133-159.

\bibitem[W1]{W}
W. Wang, Dual pairs and tensor categories of modules over the Lie algebras $\wh{gl}_\infty$ and $W_{1+\infty}$,
{\em J. Algebra}  {\bf 216} (1999) 159-177.

\bibitem[W2]{W2}
W. Wang, Duality in infinite dimensional Fock representations, {\em Commun. Contem. Math.} {\bf 1}
(1999) 155-199.

%

\end{thebibliography}
\end{document}